\documentclass[oneside,english]{amsart}
\usepackage[T1]{fontenc}
\usepackage[latin9]{inputenc}
\usepackage{geometry}
\usepackage{color}
\usepackage{babel}
\usepackage{refstyle}
\usepackage{float}
\usepackage{units}
\usepackage{amstext}
\usepackage{amsthm}
\usepackage{amssymb}
\usepackage{stmaryrd}
\usepackage{graphicx}
\PassOptionsToPackage{normalem}{ulem}
\usepackage{ulem}
\usepackage[unicode=true,pdfusetitle,
 bookmarks=true,bookmarksnumbered=false,bookmarksopen=false,
 breaklinks=false,pdfborder={0 0 0},pdfborderstyle={},backref=false,colorlinks=true]
 {hyperref}
\hypersetup{
 citecolor=blue}

\makeatletter

\AtBeginDocument{\providecommand\subsecref[1]{\ref{subsec:#1}}}
\AtBeginDocument{\providecommand\partref[1]{\ref{part:#1}}}
\AtBeginDocument{\providecommand\figref[1]{\ref{fig:#1}}}
\AtBeginDocument{\providecommand\lemref[1]{\ref{lem:#1}}}
\AtBeginDocument{\providecommand\secref[1]{\ref{sec:#1}}}
\AtBeginDocument{\providecommand\thmref[1]{\ref{thm:#1}}}
\AtBeginDocument{\providecommand\defref[1]{\ref{def:#1}}}
\AtBeginDocument{\providecommand\corref[1]{\ref{cor:#1}}}
\AtBeginDocument{\providecommand\eqref[1]{\ref{eq:#1}}}
\AtBeginDocument{\providecommand\assuref[1]{\ref{assu:#1}}}
\RS@ifundefined{subsecref}
  {\newref{subsec}{name = \RSsectxt}}
  {}
\RS@ifundefined{thmref}
  {\def\RSthmtxt{theorem~}\newref{thm}{name = \RSthmtxt}}
  {}
\RS@ifundefined{lemref}
  {\def\RSlemtxt{lemma~}\newref{lem}{name = \RSlemtxt}}
  {}

\numberwithin{equation}{section}
\numberwithin{figure}{section}
\theoremstyle{plain}
\newtheorem{thm}{\protect\theoremname}
\theoremstyle{definition}
\newtheorem{defn}[thm]{\protect\definitionname}
\theoremstyle{definition}
\newtheorem{problem}[thm]{\protect\problemname}
\theoremstyle{remark}
\newtheorem{claim}[thm]{\protect\claimname}
\theoremstyle{plain}
\newtheorem{cor}[thm]{\protect\corollaryname}
\theoremstyle{remark}
\newtheorem{rem}[thm]{\protect\remarkname}
\theoremstyle{plain}
\newtheorem{lem}[thm]{\protect\lemmaname}
\theoremstyle{definition}
\newtheorem{example}[thm]{\protect\examplename}
\theoremstyle{remark}
\newtheorem{assumption}[thm]{\protect\assumptionname}

\@ifundefined{date}{}{\date{}}
\makeatother

\providecommand{\assumptionname}{Assumption}
\providecommand{\claimname}{Claim}
\providecommand{\corollaryname}{Corollary}
\providecommand{\definitionname}{Definition}
\providecommand{\examplename}{Example}
\providecommand{\lemmaname}{Lemma}
\providecommand{\problemname}{Problem}
\providecommand{\remarkname}{Remark}
\providecommand{\theoremname}{Theorem}

\begin{document}
\global\long\def\norm#1{\left\Vert #1\right\Vert }%
\global\long\def\AA{\mathbb{A}}%
\global\long\def\QQ{\mathbb{Q}}%
\global\long\def\PP{\mathbb{P}}%
\global\long\def\CC{\mathbb{C}}%
\global\long\def\HH{\mathbb{H}}%
\global\long\def\ZZ{\mathbb{Z}}%
\global\long\def\NN{\mathbb{N}}%
\global\long\def\KK{\mathbb{K}}%
\global\long\def\RR{\mathbb{R}}%
\global\long\def\FF{\mathbb{F}}%
\global\long\def\oo{\mathcal{O}}%
\global\long\def\aa{\mathcal{A}}%
\global\long\def\bb{\mathcal{B}}%
\global\long\def\ff{\mathcal{F}}%
\global\long\def\mm{\mathcal{M}}%
\global\long\def\limfi#1#2{{\displaystyle \lim_{#1\to#2}}}%
\global\long\def\pp{\mathcal{P}}%
\global\long\def\qq{\mathcal{Q}}%
\global\long\def\da{\mathrm{da}}%
\global\long\def\dt{\mathrm{dt}}%
\global\long\def\dg{\mathrm{dg}}%
\global\long\def\ds{\mathrm{ds}}%
\global\long\def\dm{\mathrm{dm}}%
\global\long\def\dmu{\mathrm{d\mu}}%
\global\long\def\dx{\mathrm{dx}}%
\global\long\def\dy{\mathrm{dy}}%
\global\long\def\dz{\mathrm{dz}}%
\global\long\def\dnu{\mathrm{d\nu}}%
\global\long\def\flr#1{\left\lfloor #1\right\rfloor }%
\global\long\def\nuga{\nu_{\mathrm{Gauss}}}%
\global\long\def\diag#1{\mathrm{diag}\left(#1\right)}%
\global\long\def\bR{\mathbb{R}}%
\global\long\def\Ga{\Gamma}%
\global\long\def\PGL{\mathrm{PGL}}%
\global\long\def\GL{\mathrm{GL}}%
\global\long\def\PO{\mathrm{PO}}%
\global\long\def\SL{\mathrm{SL}}%
\global\long\def\PSL{\mathrm{PSL}}%
\global\long\def\SO{\mathrm{SO}}%
\global\long\def\mb#1{\mathrm{#1}}%
\global\long\def\wstar{\overset{w^{*}}{\longrightarrow}}%
\global\long\def\vphi{\varphi}%
\global\long\def\av#1{\left|#1\right|}%
\global\long\def\inv#1{\left(\mathbb{Z}/#1\mathbb{Z}\right)^{\times}}%
\global\long\def\cH{\mathcal{H}}%
\global\long\def\cM{\mathcal{M}}%
\global\long\def\bZ{\mathbb{Z}}%
\global\long\def\bA{\mathbb{A}}%
\global\long\def\bQ{\mathbb{Q}}%
\global\long\def\bP{\mathbb{P}}%
\global\long\def\eps{\epsilon}%
\global\long\def\on#1{\mathrm{#1}}%
\global\long\def\nuga{\nu_{\mathrm{Gauss}}}%
\global\long\def\set#1{\left\{  #1\right\}  }%
\global\long\def\smallmat#1{\begin{smallmatrix}#1\end{smallmatrix}}%
\global\long\def\len{\mathrm{len}}%
\global\long\def\idealeq{\trianglelefteqslant}%

\title{Shearing in the space of adelic lattices}
\author{Ofir David}
\begin{abstract}
In this notes we show how a problem regarding continued fractions
of rational numbers, lead to several phenomena in number theory and
dynamics, and eventually to the problem of shearing of divergent diagonal
orbits in the space of adelic lattices. Finding these ideas quite interesting,
the first half of these notes is about explaining theses ideas, the
intuition and motivation behind them, and the second contains the
details and proofs.
\end{abstract}

\email{eofirdavid@gmail.com}
\maketitle

\section{Introduction}

\subsection{The main results}

The connection between number theory and homogeneous dynamics is well
established - many problems in number theory have found elegant formulations
and solutions in the language of homogeneous dynamics, and in particular
the dynamics of the space of unimodular Euclidean lattices $\SL_{n}\left(\ZZ\right)\backslash\SL_{n}\left(\RR\right)$.
One of the main examples, which led eventually to this paper, is the
problem of finding good Diophantine approximations. 

It is well known that these rational approximations can be read as
the prefixes of the continued fraction expansion of any given number
\[
x=\left[a_{0};a_{1},a_{2},...\right]=a_{0}+\frac{1}{a_{1}+\frac{1}{a_{2}+\ddots}}.
\]
This continued fraction presentation comes with the natural Gauss
map, which is simply the shift left map $T\left(\left[0;a_{1},a_{2},a_{3},...\right]\right):=\left[0;a_{2},a_{3},...\right]$
(or equivalently $T\left(x\right):=\frac{1}{x}-\flr{\frac{1}{x}}$),
and many problems in Diophantine approximation are studied via this
map. In particular, one of the mail tools in this area is the ergodicity
of this map with respect to the Gauss measure $\nuga=\frac{1}{\ln\left(2\right)}\cdot\frac{1}{1+t}\cdot\dt$.
This allows us to use the Pointwise Ergodic Theorem, which states
that almost every $x\in\left[0,1\right]$ is generic, namely for every
continuous function $f:\left[0,1\right]\to\RR$ we have that 
\[
\frac{1}{N}\sum_{0}^{N-1}f\left(T^{i}\left(x\right)\right)\to\nuga\left(f\right)=\frac{1}{\ln\left(2\right)}\int_{0}^{1}f\left(t\right)\frac{1}{1+t}\dt.
\]

In this case we say that the $T$-orbit of $x$ equidistributes. However,
this is not true in general, and in particular this fails for rational
$x$ which has a finite continued fraction expansion. In this case
(and in others as well), instead of studying an orbit of a single
point $x$, we usually study the ``finite'' orbits of certain naturally
defined finite families $\ff_{i}$ of points (and in this paper the
families of rationals $\left\{ \frac{p}{q}\;\mid\;1\leq p\leq q,\;\left(p,q\right)=1\right\} $
for $q\in\NN$). Then, our question is if taking the orbits of each
family together, do they equidistribute as $i\to\infty$.

It is well known that any $T$-orbit has a continuous analogue as
an orbit of the diagonal subgroup $A\leq\SL_{2}\left(\RR\right)$
in $\SL_{2}\left(\ZZ\right)\backslash\SL_{2}\left(\RR\right)$. We
can reformulate the problem above in this new continuous language,
where the ``finite'' orbits become divergent $A$-orbits. This already
give us more tools to work with, and in particular we can use unipotent
flows which are much more understood than $A$-flows. 

As it turns out, an even more natural point of view for this kind
of questions is actually over the Adeles, where these finite families
of $A$-orbits are combined together to a translation of a single
orbit of the diagonal matrices over the adeles $\AA$, which is known
in the literature as the shearing process.

In this notes we show that these translations of a single orbit over
the adeles through the origin always equidistribute, as long as there
are no trivial reasons for them not to, or formally we have the following.
\begin{thm}
\label{thm:main_theorem} Let $X_{\AA}=\Gamma_{\AA}\backslash G_{\AA}$
where $\Gamma_{\AA}=\GL_{2}\left(\QQ\right)$ and 
\begin{align*}
G_{\AA} & =\left\{ \left(g^{\left(\infty\right)},g^{\left(2\right)},g^{\left(3\right)},...\right)\in\GL_{2}\left(\AA\right)\;\mid\;\prod_{\nu}\left|\det\left(g^{\left(\nu\right)}\right)\right|_{\nu}=1\right\} .
\end{align*}
Denote by $A_{\AA}\leq G_{\AA}$ the diagonal subgroup and let $\delta_{\Gamma_{\AA}A_{\AA}}$
be the $A_{\AA}$-invariant orbit measure on $\Gamma_{\AA}A_{\AA}$.
Then for any sequence $g_{i}\in G_{\AA}$ such that $g_{i}/A_{\AA}$
diverges in $\nicefrac{G_{\AA}}{A_{\AA}}$, the sequence $g_{i}\left(\delta_{\Gamma_{\AA}A_{\AA}}\right)$
equidistributes, i.e. for any $f_{1},f_{2}\in C_{c}\left(X_{\AA}\right)$
with $\mu_{Haar,\AA}\left(f_{2}\right)\neq0$ we have that 
\[
\frac{\left(g_{i}\delta_{\Gamma_{\AA}A_{\AA}}\right)\left(f_{1}\right)}{\left(g_{i}\delta_{\Gamma_{\AA}A_{\AA}}\right)\left(f_{2}\right)}\to\frac{\mu_{Haar,\AA}\left(f_{1}\right)}{\mu_{Haar,\AA}\left(f_{2}\right)}.
\]
\end{thm}

Once the theorem above is proved over the adeles, we automatically
obtain similar results for spaces which are defined naturally as projections
of $X_{\AA}$ (see \subsecref{Adelic_lattices} for the definition).
One of the main examples is the space of unimodular lattices $X_{\RR}=\SL_{2}\left(\ZZ\right)\backslash\SL_{2}\left(\RR\right)$.
The discussion about equidistribution of $T$-orbits of rational points,
is a specific case of the theorem above where the translation is only
in the finite places, and then projecting to $X_{\RR}$. This specific
case was proven in \cite{david_equidistribution_2018} by the author
together with Uri Shapira.

More specifically, let $q$ be some positive integer and set $\ff_{q}=\left\{ 1\leq p\leq q\;\mid\;\left(p,q\right)=1\right\} $.
For $p\in\ff_{q}$, let $\len\left(\frac{p}{q}\right)$ be the length
of the continued fraction expansion of $\frac{p}{q}$. Letting $T$
be the Gauss map on $\left(0,1\right)$, we can define the average
of the ``$T$-orbit'' of $\frac{p}{q}$ to be the probability measure
\[
\nu_{p/q}=\frac{1}{\len\left(p/q\right)}\sum_{0}^{\len\left(p/q\right)}\delta_{T^{i}\left(\frac{p}{q}\right)}.
\]
We then define the average
\[
\nu_{q}=\frac{1}{\left|\ff_{q}\right|}\sum_{p\in\ff_{q}}\nu_{p/q}.
\]

\begin{thm}
\cite{david_equidistribution_2018} The measures $\nu_{q}$ equidistribute,
namely $\nu_{q}\wstar\nuga$, where $\nuga=\frac{\dt}{\ln\left(2\right)\left(1+t\right)}$
is the Gauss measure on $\left(0,1\right)$.
\end{thm}

\newpage{}

One of the main tools to show equidistribution when translating diagonal
orbits, is using shearing. This process is well known, however when
trying to solve the main theorem above, we will encounter three main
problems:
\begin{enumerate}
\item The orbit measure $\delta_{x_{\AA}A_{\AA}}$ and its translations
are not probability measures. This leads to the definition and study
of divergent orbits which are $A_{\AA}$-invariant and locally finite.
\item While the behavior of translations over the finite (prime) places
and the infinite (real) place behave similarly, they are not quite
the same and we need to ``glue'' them together.
\item Finally, the translation is over the adeles, and in particular the
number of primes in which we translate is nontrivial and can grow
to infinity.
\end{enumerate}
The study of translations of a fixed divergent $A$-orbit in the real
place, was first done by Shah and Oh in \cite{oh_limits_2014} for
dimension 2 over $\RR$, where they give a quantitative result. The
high dimension result over $\RR$ was done by Shapira and Cheng in
\cite{shapira_limiting_nodate}. The proof for translation in the
finite prime places in dimension 2 was done by the author and Shapira
in \cite{david_equidistribution_2018} where it was later generalized
to high dimension for certain type of translations in \cite{david_equidistribution_nodate}.

In this paper we combine the results for the translations in the finite
and infinite places for dimension 2 to give the full equidistribution
theorem.

\subsection{The intuition and the proofs}

The paper is composed of two main parts. \partref{Intuition} contains
the main ideas of the proof, while in \partref{Proofs_Details} we
complete the details and the more technical parts of the proof. As
mentioned above, the two ``parts'' of the proof - the translation
in the real place, and the translation in the finite places, were
already done previously and here we just combine them together. However,
we believe that the story leading to the final result is the interesting
part of this work, as it goes through several interesting areas of
number theory and dynamics utilizing some of the central results in
a natural way. As such, the emphasis of this notes is on \partref{Intuition}
and it was written with newcomers to this areas in mind, starting
with the original problem in continued fractions, and ending in the
equidistribution result in the language of the adelic numbers.

\subsection{Acknowledgments}

I would like to thank Uri Shapira for introducing me to the interesting
land residing between number theory and dynamics, and in particular
to the problem studied in this notes. The research leading to these
results has received funding from the European Research Council under
the European Union Seventh Framework Programme (FP/2007-2013) / ERC
Grant Agreement n. 335989.

\newpage{}

\part{\label{part:Intuition}Intuition and sketch of the proof}

\global\long\def\norm#1{\left\Vert #1\right\Vert }%
\global\long\def\AA{\mathbb{A}}%
\global\long\def\QQ{\mathbb{Q}}%
\global\long\def\PP{\mathbb{P}}%
\global\long\def\CC{\mathbb{C}}%
\global\long\def\HH{\mathbb{H}}%
\global\long\def\ZZ{\mathbb{Z}}%
\global\long\def\NN{\mathbb{N}}%
\global\long\def\KK{\mathbb{K}}%
\global\long\def\RR{\mathbb{R}}%
\global\long\def\FF{\mathbb{F}}%
\global\long\def\oo{\mathcal{O}}%
\global\long\def\aa{\mathcal{A}}%
\global\long\def\bb{\mathcal{B}}%
\global\long\def\ff{\mathcal{F}}%
\global\long\def\mm{\mathcal{M}}%
\global\long\def\limfi#1#2{{\displaystyle \lim_{#1\to#2}}}%
\global\long\def\pp{\mathcal{P}}%
\global\long\def\qq{\mathcal{Q}}%
\global\long\def\da{\mathrm{da}}%
\global\long\def\dt{\mathrm{dt}}%
\global\long\def\dg{\mathrm{dg}}%
\global\long\def\ds{\mathrm{ds}}%
\global\long\def\dm{\mathrm{dm}}%
\global\long\def\dmu{\mathrm{d\mu}}%
\global\long\def\dx{\mathrm{dx}}%
\global\long\def\dy{\mathrm{dy}}%
\global\long\def\dz{\mathrm{dz}}%
\global\long\def\dnu{\mathrm{d\nu}}%
\global\long\def\flr#1{\left\lfloor #1\right\rfloor }%
\global\long\def\nuga{\nu_{\mathrm{Gauss}}}%
\global\long\def\diag#1{\mathrm{diag}\left(#1\right)}%
\global\long\def\bR{\mathbb{R}}%
\global\long\def\Ga{\Gamma}%
\global\long\def\PGL{\mathrm{PGL}}%
\global\long\def\GL{\mathrm{GL}}%
\global\long\def\PO{\mathrm{PO}}%
\global\long\def\SL{\mathrm{SL}}%
\global\long\def\PSL{\mathrm{PSL}}%
\global\long\def\SO{\mathrm{SO}}%
\global\long\def\PSO{\mathrm{PSO}}%
\global\long\def\mb#1{\mathrm{#1}}%
\global\long\def\wstar{\overset{w^{*}}{\longrightarrow}}%
\global\long\def\vphi{\varphi}%
\global\long\def\av#1{\left|#1\right|}%
\global\long\def\inv#1{\left(\mathbb{Z}/#1\mathbb{Z}\right)^{\times}}%
\global\long\def\cH{\mathcal{H}}%
\global\long\def\cM{\mathcal{M}}%
\global\long\def\bZ{\mathbb{Z}}%
\global\long\def\bA{\mathbb{A}}%
\global\long\def\bQ{\mathbb{Q}}%
\global\long\def\bP{\mathbb{P}}%
\global\long\def\eps{\epsilon}%
\global\long\def\on#1{\mathrm{#1}}%
\global\long\def\nuga{\nu_{\mathrm{Gauss}}}%
\global\long\def\dnuga{\mathrm{d\nu_{\mathrm{Gauss}}}}%
\global\long\def\set#1{\left\{  #1\right\}  }%
\global\long\def\smallmat#1{\begin{smallmatrix}#1\end{smallmatrix}}%
\global\long\def\len{\mathrm{len}}%
\global\long\def\idealeq{\trianglelefteqslant}%
\global\long\def\ss#1{{\scriptstyle #1}}%
\global\long\def\sss#1{{\scriptscriptstyle #1}}%

In this part we give the main ideas of the proof and we defer the
details themselves to \partref{Proofs_Details}.

In \subsecref{continuoued_fraction} we start with a problem of equidistribution
of continued fractions. This dynamic system is one of the first examples
when learning about ergodic theory. However, we will be interested
in points in the system where the ergodic theorems fail e the rational
points. In \subsecref{continuous_analogue}, we will recall the connection
between continued fractions, and diagonal orbits in the space of unimodular
lattice $\SL_{2}\left(\ZZ\right)\backslash\SL_{2}\left(\RR\right)$.
Not only can we reformulate our problem there, we will also show why
it is a more natural language to use when trying to solve such problem.

In \subsecref{Symmetries_horospheres}, we will use some of the symmetries
that can be seen much more naturally in this new way, and more over,
we will see how not only our measures are defined using diagonal orbits,
but they also have some horocyclic nature which we can utilize. In
particular, the combination of both diagonal and horocyclic nature
of the problem will suggest the use of equidistribution of expanding
horocycles, which will be one of the main results needed in this notes.

We continue in \subsecref{The_p_adic_motivation} to find an even
better language for our problem. While in the space of Euclidean lattices
we have an average of finitely many diagonal orbits, in \subsecref{The_p_adic_motivation}
we will see how their definitions suggest an even bigger world where
they are all combined into a single diagonal orbit. This bigger world
will eventually lead to the definition of $p$-adic numbers, for which
we provide the main definitions, results and the intuition needed
for the main theorem. In this new language of $p$-adic numbers, our
problem turn into a well known phenomenon called shearing e translating
the diagonal orbit using a unipotent matrix.

The shearing process uses the results about equidistribution of expanding
horocycles in order to prove that such translations equidistribute
in themselves. In \subsecref{Shearing} we will show how shearing
lead naturally to thinking about expanding horocycles, and we will
use the shearing in $\SL_{2}\left(\ZZ\right)\backslash\SL_{2}\left(\RR\right)$
as an example in order to visualize it. This example will eventually
be part of the proof of the main theorem e this is the shearing in
the real place, while the problem of continued fractions of rational
numbers is shearing the the finite (prime) places.

Lastly, in \subsecref{shearing_Adelic_intuition} we show how to combine
the real place and all the prime places in order to form the adelic
numbers. The language of adelic numbers is very common in problems
relating to number theory and this one is not different. Both the
original problem of equidistribution of continued fractions of rational
numbers, and the shearing above, can be thought of projections of
an analogue problem over the adeles. While the ideas mentioned until
that point can be used to show equidistribution in the projection
to $\SL_{2}\left(\ZZ\right)\backslash\SL_{2}\left(\RR\right)$, in
this section we will talk about how to lift this solution to all of
the adeles. In particular we will see how to use either entropy or
classification of unipotent invariant measures to do this lifting.

\newpage{}

\section{\label{subsec:continuoued_fraction}The continued fraction motivation}

The starting point for out story is with continued fractions and the
problem of finding good rational approximations. We give the main
ideas and results here, and for more details on continued fractions,
and their connection to diagonal orbits, which we discuss in \subsecref{continuous_analogue},
the reader is referred to \cite{einsiedler_ergodic_2010}.

Recall that for an irrational $\alpha\in\RR\backslash\QQ$, a Diophantine
approximation is a rational $\frac{p}{q}\in\QQ$ such that $\left|\alpha-\frac{p}{q}\right|<\frac{1}{q^{2}}$.
The famous Dirichlet's theorem for Diophantine approximations shows
that there are infinitely many distinct solutions to this inequality.
Trying to actually find these solutions, we are led to the continued
fraction expansion (CFE). 

Given $a_{0}\in\ZZ$ and $a_{i}\in\NN$ positive for $i\geq1$, we
define 
\begin{align*}
\left[a_{0};a_{1},a_{2},...,a_{k}\right] & =a_{0}+\frac{1}{a_{1}+\frac{1}{a_{2}+\frac{1}{a_{3}+\frac{1}{\ddots\;\frac{1}{a_{k}}}}}}=\frac{p_{k}}{q_{k}},\\
\left[a_{0};a_{1},a_{2},a_{3},...\right] & =a_{0}+\frac{1}{a_{1}+\frac{1}{a_{2}+\frac{1}{a_{3}+\ddots}}}:=\limfi k{\infty}\left[a_{0};a_{1},a_{2},...,a_{k}\right].
\end{align*}
It is well known that the convergents $\frac{p_{k}}{q_{k}}\in\QQ$
always converge, and every $\alpha\in\RR$ has such an expression
as a CFE. Moreover, the convergents of $\alpha$ satisfy $\left|\alpha-\frac{p_{k}}{q_{k}}\right|<\frac{1}{q_{k}^{2}}$,
and in a sense these are the best possible Diophantine approximations.

Studying these approximations, we can restrict our attention to $\left[0,1\right]$,
namely $a_{0}=0$, so we are left with the $\NN$-valued sequences
$\left(a_{1},a_{2},...\right)$. While the finite prefixes correspond
to these convergent $\frac{p_{n}}{q_{n}}$, ``most'' of the information
is in the tails. This leads to the Gauss map, which is basically the
shift left map:
\begin{align*}
T\left(\left[0;a_{1},a_{2},a_{3},...\right]\right) & =\left[0;a_{2},a_{3},...\right]\\
T\left(x\right) & =\frac{1}{x}-\flr{\frac{1}{x}}.
\end{align*}

This Gauss map is ergodic with respect to the Gauss probability measure
$\nuga:=\frac{1}{\ln\left(2\right)}\frac{\dt}{\left(1+t\right)}$.
Recall that the \emph{Mean Ergodic Theorem (MET)} in this case states
that 
\[
\forall f\in L^{1}\;:\;\left|\frac{1}{N}\sum_{0}^{N-1}f\circ T^{i}-\int_{0}^{1}f\dnuga\right|_{1}\to0.
\]
An upgrade of this theorem, the \emph{Pointwise Ergodic Theorem (PET)},
states that almost every $x\in\left[0,1\right]$ is generic, namely
\[
\forall f\in L^{1}\;:\;\frac{1}{N}\sum_{0}^{N-1}f\circ T^{i}\left(x\right)\to\int_{0}^{1}f\dnuga.
\]
In other words, taking the discrete averages over longer and longer
parts of the $T$-orbit of $x$ gets us closer and closer to the integral. 

\newpage{}

While almost every point is generic, there are many interesting families
of points which are not. It is not hard to show that $x$ is generic
if and only if its coefficients $a_{i}$ in its CFE satisfy a certain
statistics called the Gauss Kuzmin statistics, and in particular every
integer should appear in this sequence. Some interesting example where
this fails:
\begin{enumerate}
\item \uline{The \mbox{$a_{i}$} are bounded}: These numbers are called
\emph{badly approximable numbers} e numbers which do not have ``very
good'' Diophantine approximation. By definition, a number $\alpha$
is badly approximable, if there is some $c>0$ such that for every
rational $\frac{p}{q}$ we have $\frac{c}{q^{2}}<\left|\alpha-\frac{p}{q}\right|$.
\item \uline{The \mbox{$a_{i}$} are eventually periodic:} These correspond
to real algebraic numbers of rank $2$. For example, the number $\alpha=\left[0;1,1,1,....\right]$
satisfy $\alpha=\frac{1}{1+T\left(\alpha\right)}=\frac{1}{1+\alpha}$,
so that $\alpha^{2}+\alpha-1=0$ (and $\alpha>0$), which implies
that $\alpha=\frac{-1+\sqrt{5}}{2}$ (and $\alpha+1$ is the Golden
ratio).
\item \uline{The \mbox{$a_{i}$} is a finite sequence}: These correspond
to rational numbers. In this case we cannot even apply the theorem
since $T^{n}\left(\alpha\right)=0$ for some $n$, and $T$ is not
defined on zero.
\end{enumerate}
While all these families have measure zero by the PET, the family
of badly approximable numbers is very big e indeed, its cardinality
is that of the continuum, it has maximal Hausdorff dimension and is
even Schmidt winning. On the other hand, the other two families are
only infinitely countable, and it turns out that other versions of
the mean and pointwise ergodic theorems hold for them. Note that the
Gauss map do not exactly act on the rationals since their $T$e''orbits''
get stuck once they reach zero (they have ``finite'' $T$-orbit),
but other than this problem, in a sense their behavior is similar
to the numbers with eventually periodic expansion. As we are mainly
interested in the rational case in this paper, we shall concentrate
on them, and the well known analogue for algebraic numbers, which
in essence is Linnik's theorem, can be seen in \cite{einsiedler_distribution_2012}.\\

Trying to find the continued fraction coefficients of a rational number
$\frac{n}{m}$ with $1\leq n\leq m,\;gcd\left(n,m\right)=1$ is basically
the same as running the Euclidean division algorithm. Indeed, writing
$m=a_{1}n+r_{1}$ with $0\leq r_{1}<n$, we obtain the equality 
\[
\frac{n}{m}=\frac{1}{m/n}=\frac{1}{a_{1}+\frac{r_{1}}{n}}.
\]
If $r_{1}=0$, then $\frac{n}{m}=\left[0;a_{1}\right]$ and we are
done. Otherwise, we can divide $n$ by $r_{1}$ to get $\frac{1}{a_{1}+\frac{1}{n/r_{1}}}$
and repeat this process, leading eventually to the continued fraction
expansion $\frac{n}{m}=\left[0;a_{1},a_{2},...,a_{k}\right]$. Note
also that $T\left(\frac{n}{m}\right)=\frac{m\;(mod\;n)}{n}$ where
$m\;\left(mod\;n\right)$ the remainder of dividing $m$ by $n$. 

Let $\len\left(\frac{n}{m}\right)$ to be the length of the $T$e''orbit''
of $\frac{n}{m}$, namely the first index $k$ such that $T^{k}\left(\frac{n}{m}\right)=0$,
or equivalently the number of steps in the Euclidean division algorithm
when dividing $q$ by $p$. We then let 
\[
\nu_{n/m}=\frac{1}{\len\left(\frac{n}{m}\right)}\sum_{0}^{\len\left(\frac{n}{m}\right)-1}\delta_{T^{i}\left(\frac{n}{m}\right)}
\]
be the uniform probability measure on the ``full $T$-orbit'' of
$\frac{n}{m}$. 

\begin{figure}[H]
\begin{centering}
\includegraphics[scale=0.5]{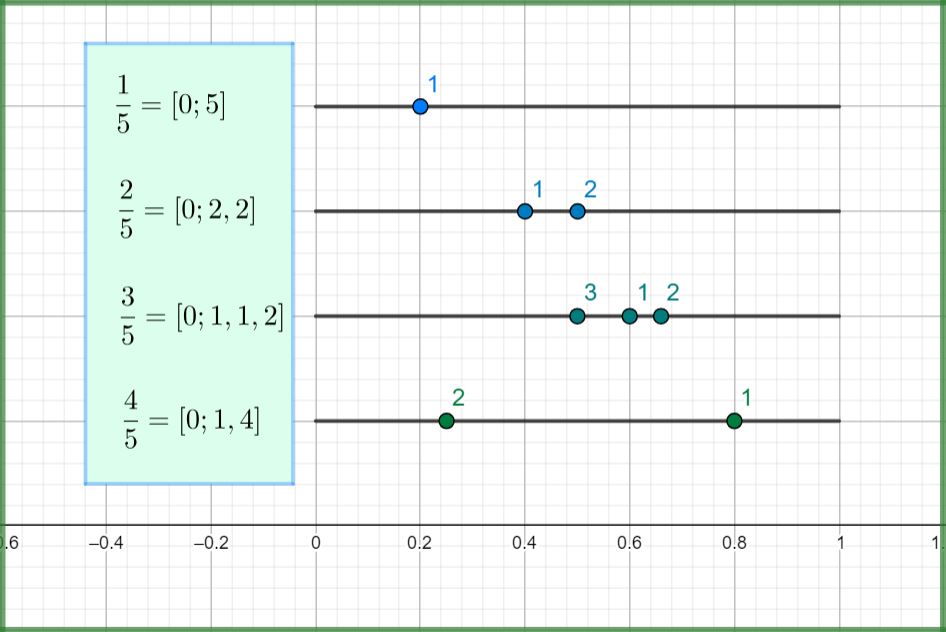}
\par\end{centering}
\caption{\label{fig:Gauss_orbits}The ``finite orbits'' of $\frac{n}{5}$
for $n=1,2,3,4$ (GeoGebra \cite{hohenwarter_geogebra:_2002}).}

\end{figure}
Since a single ``$T$-orbit'' of a rational number cannot converge
to the Gauss measure $\nuga$, we can hope that maybe a sequence of
such orbits converge equidistribute:
\begin{defn}
We say that a sequence $\mu_{i}$ of probability measures on $\left[0,1\right]$
equidistributes if $\mu_{i}\wstar\nuga$.
\end{defn}

\begin{problem}
Find $1\leq n_{m}\leq m$, $\left(n_{m},m\right)=1$ such that $\nu_{n_{m}/m}$
equidistributes, i.e. $\nu_{n_{m}/m}\wstar\nuga$ as $m\to\infty$.
\end{problem}

Clearly, not every sequence $n_{m}$ defines an equidistributing sequence
$\nu_{n_{m}/m}$. For example, as can be seen in \figref{Gauss_orbits},
the measures $\nu_{1/m}=\delta_{1/m}$ are always Dirac measures on
a single point, so that $\nu_{1/m}$ cannot converge to $\nuga$.
Similarly $\nu_{(m-1)/m}$ are supported on 2 points so they cannot
equidistribute. But there are only 2 such ``bad'' measures for any
$q$, and maybe the rest are not so bad.

With this in mind, for $m\in\NN$ fixed we set $\Lambda_{m}=\left\{ n\in\NN\;\mid\;1\leq n\leq m,\;\left(n,m\right)=1\right\} $
and define the averages 
\[
\nu_{m}:=\frac{1}{\left|\Lambda_{m}\right|}{\displaystyle \sum_{n\in\Lambda_{m}}}\nu_{n/m}
\]
where $\left|\Lambda_{m}\right|=\varphi\left(m\right)$ is the Euler
totient function. Thus if the set of ``bad'' orbit measures is very
small, they will not affect this average. In \cite{david_equidistribution_2018}
the author and Uri Shapira proved that this is indeed the case, namely
$\nu_{m}\wstar\nuga$. 

It is interesting to ask what happens if we do not average over all
of $\Lambda_{m}$, but only over, for example, half of it. If we decompose
$\Lambda_{m}=\tilde{\Lambda}_{m}\sqcup\hat{\Lambda}_{m}$ to two halves,
and let $\tilde{\nu}_{m},\hat{\nu}_{m}$ be the corresponding averages,
then $\nu_{m}=\frac{1}{2}\tilde{\nu}_{m}+\frac{1}{2}\hat{\nu}_{m}$.
We can now take the limit of both sides, where we might restrict to
a subsequence to assume that $\tilde{\nu}_{m}$ and $\hat{\nu}_{m}$
converge to $\tilde{\nu}_{\infty}$ and $\hat{\nu}_{\infty}$ respectively
to get the convex combination of 
\[
\nuga=\frac{1}{2}\tilde{\nu}_{\infty}+\frac{1}{2}\hat{\nu}_{\infty}.
\]

One of the defining properties of ergodic measures with respect to
some action $T$ is that they are the extreme point in the space of
$T$-invariant probability measure, namely they cannot be written
as a nontrivial convex combination of $T$-invariant probability measures.
It is not immediately clear that both $\tilde{\nu}_{\infty}$ and
$\hat{\nu}_{\infty}$ are $T$-invariant, but this is true, and it
will be much more obvious once we move on to the language of lattice
and diagonal orbits. In any way, the property mentioned above, shows
that both $\tilde{\nu}_{\infty}$ and $\hat{\nu}_{\infty}$ must be
$\nuga$. The constant $\frac{1}{2}$ was not really important , and
we can actually do it for any $0<\alpha<1$.

This idea can be used to further upgraded the equidistribution result
and show that the intuition about small ``bad'' sets is correct
e there are families $\Lambda_{m}'\subseteq\Lambda_{m}$ with $\frac{\left|\Lambda_{m}'\right|}{\left|\Lambda_{m}\right|}\to1$,
such that for any choice of $n_{m}\in\ff_{m}'$ we have that $\nu_{n_{m}/m}\wstar\nu_{Gauss}$
as $m\to\infty$. Thus in a philosophical sense we have a mean and
a pointwise ergodic theorems for the rational (non generic) points
as well.

Trying to prove this claim, leads to at least two problems that we
must overcome.
\begin{itemize}
\item The standard dynamical method to solve such problems is to show that
the limit measure $\nu_{m}\wstar\nu_{\infty}$ (if it exists) is $T$-invariant,
and then use some classification for $T$-invariant probability measure.
If each of the $\nu_{m}$ were $T$-invariant in themselves, then
clearly $\nu_{\infty}$ would be $T$-invariant as well. However,
in our case not only are the measure not $T$-invariant, since $T\left(0\right)$
is not well defined, $T^{\ell}\left(\nu_{m}\right)$ is not well defined
for any $\ell>1$.
\item In the current formulation, each point the the orbits of $\frac{n}{m}$
for some $\left(n,m\right)=1$ has some positive weight in $\nu_{m}=\frac{1}{\varphi\left(m\right)}{\displaystyle \sum_{n\in\Lambda_{m}}}\nu_{n/m}$.
However, its weight is determined according to which orbit it is in.
For example, the ``orbit'' of $\frac{1}{5}$ contains only one point
so its weight from there is $\frac{1}{1}\cdot\frac{1}{\varphi\left(5\right)}=\frac{1}{4}$.
On the other hand, the ``orbit'' of $\frac{3}{5}$ has three points,
so each point there contributes $\frac{1}{3}\cdot\frac{1}{\varphi\left(5\right)}=\frac{1}{12}$
mass to the measure $\nu_{m}$. If we want each such point to have
the same measure, then we can define instead
\[
\tilde{\nu}_{m}=\frac{1}{\sum_{n\in\Lambda_{m}}\len\left(\frac{n}{m}\right)}\sum_{n\in\Lambda_{m}}\sum_{0}^{\len\left(\frac{n}{m}\right)-1}\delta_{T^{i}\left(\frac{n}{m}\right)},
\]
and similarly ask whether $\tilde{\nu}_{m}\to\nuga$. Thus, in a sense
it is not clear what is the ``right'' normalization. Interestingly,
the upgrade mentioned above shows that both of these normalization
equidistribute.
\end{itemize}
As we shall see, viewing this problem via the diagonal flow in the
space of 2-dimensional unimodular lattices helps us solve both of
these problems.

\newpage{}

\section{\label{subsec:continuous_analogue}The continuous analogue and symmetries}

It is well known that the continued fraction expansion of some $\alpha\in\left(0,1\right)$
can be extracted from a certain $A$-orbit in the space of 2-dimensional
unimodular lattices $\SL_{2}\left(\ZZ\right)\backslash\SL_{2}\left(\RR\right)$
where $A$ is the diagonal subgroup of $\SL_{2}\left(\RR\right)$.
Let us recall the main steps of this process.

For the rest of this section we fix the following notations:
\begin{align*}
X & :=\SL_{2}\left(\ZZ\right)\backslash\SL_{2}\left(\RR\right),\quad\HH:=\PSL_{2}\left(\RR\right)/\PSO_{2}\left(\RR\right)\\
A & :=\left\{ a\left(t\right)=\left(\begin{array}{cc}
e^{-t/2} & 0\\
0 & e^{t/2}
\end{array}\right)\;\mid\;t\in\RR\right\} \\
U & :=\left\{ u_{x}=\left(\begin{array}{cc}
1 & x\\
0 & 1
\end{array}\right)\;\mid\;x\in\RR\right\} 
\end{align*}

An element $\SL_{2}\left(\ZZ\right)\cdot g$ with $g\in\SL_{2}\left(\RR\right)$
correspond to the lattice $\ZZ^{2}\cdot g$. To get some intuition
we will look instead on the hyperbolic upper half plane $\HH$ modulo
the $\SL_{2}\left(\ZZ\right)$-action which we can actually draw.
The fundamental domain for $\SL_{2}\left(\ZZ\right)$ in $\HH$ is
\[
\mathcal{F}=\left\{ z\in\CC\;\mid\;\norm z>1,\;Im\left(z\right)>0,\;\left|Re\left(z\right)\right|<\frac{1}{2}\right\} ,
\]
where a point $z\in\ff$ correspond to the lattice $span_{\ZZ}\left(1,z\right)$.

Recall that $\SL_{2}\left(\RR\right)$ act on the hyperbolic upper
half plane $\HH$ via the M\"{o}bius transformation $\ss{\left(\begin{array}{cc}
a & b\\
c & d
\end{array}\right)}z=\frac{az+b}{cz+d}$. The geodesics in $\HH$ are $gAi,\;g\in\SL_{2}\left(\RR\right)$,
which in $\HH$ look like either half circles with their ends on the
$x$-axis, or vertical lines. Trying to compute the endpoints, namely
the limit when $t\to\pm\infty$ we get that
\[
\left(\begin{array}{cc}
a & b\\
c & d
\end{array}\right)a\left(t\right)i=\frac{ae^{-t}i+b}{ce^{-t}i+d}=\begin{cases}
\frac{a}{c} & t\to-\infty\\
\frac{b}{d} & t\to\infty
\end{cases}.
\]

In particular for $g=u_{\alpha},\;\alpha\in\RR$, the endpoint of
$u_{\alpha}a\left(t\right)i$ in the past is $\frac{1}{0}=\infty$,
while in the future it is $\alpha$, so that the geodesic $u_{\alpha}Ai$
is the line $x=\alpha$. We then consider the projection of this geodesic
to the modular surface $\SL_{2}\left(\ZZ\right)\backslash\HH\cong\PSL_{2}\left(\ZZ\right)\backslash\PSL_{2}\left(\RR\right)/\SO_{2}\left(\RR\right)$,
or equivalently to the standard fundamental domain $\mathcal{F}$. 

\begin{figure}[H]
\begin{centering}
\includegraphics[scale=0.35]{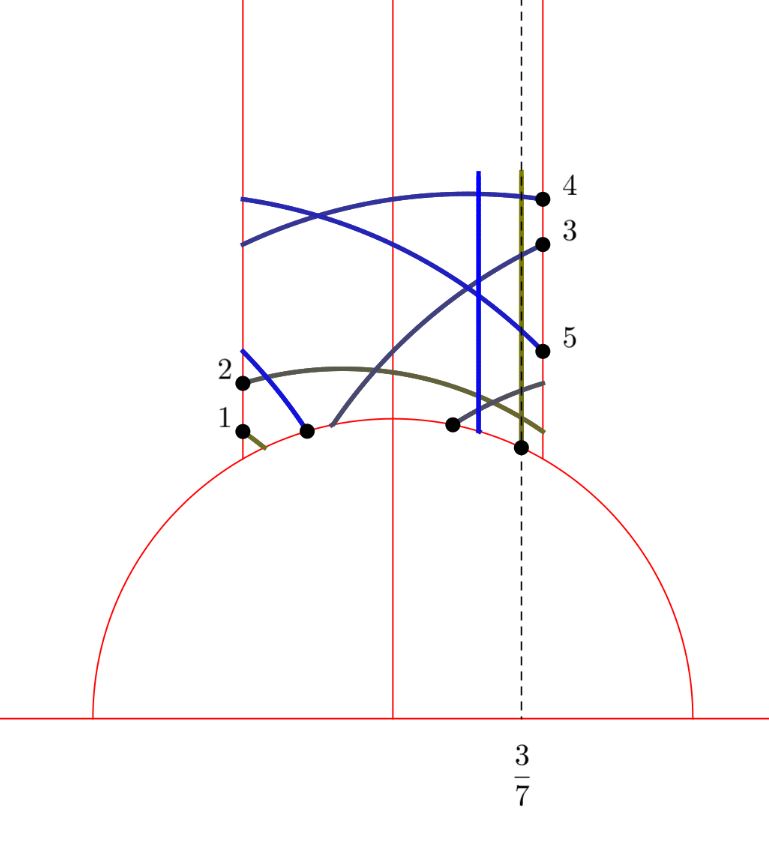}
\par\end{centering}
\caption{\label{fig:CFE_to_GEO}The geodesic from $\infty$ to $\frac{3}{7}$
projected to $\protect\SL_{2}\left(\protect\ZZ\right)\backslash\protect\HH$:
Coming from the cusp, it first hits the bottom, then 2-times the left
boundary, the bottom again, 3-times the right boundary, the bottom,
and then go straight up to the cusp. These boundary hitting can be
seen in the continued fraction expansion of $\frac{3}{7}$ which is
$\frac{3}{7}=\frac{1}{2+\frac{1}{3}}=\left[0;2,3\right]$.}
\end{figure}
Every time the geodesic leaves the fundamental domain $\ff$, we need
to act with a matrix in $\SL_{2}\left(\ZZ\right)$ in order to bring
it back inside. In particular, the matrices corresponding to the left
and right boundary of $\ff$ are $\ss{\left(\begin{array}{cc}
1 & 1\\
0 & 1
\end{array}\right)},\ss{\left(\begin{array}{cc}
1 & -1\\
0 & 1
\end{array}\right)}$, and to the lower boundary we have the matrix $\ss{\left(\begin{array}{cc}
0 & 1\\
-1 & 0
\end{array}\right)}$. These transformations affect the endpoints of the geodesic via the
maps $x\mapsto x\pm1$ and $x\mapsto-\frac{1}{x}$. Note that this
is more or less what we use in the Gauss map $T\left(x\right)=\frac{1}{x}-\flr{\frac{1}{x}}$
e indeed, we first apply $x\mapsto\frac{1}{x}$, and then decrease
by $-1$ exactly $\flr{\frac{1}{x}}$ times. Geometrically, the geodesic
hits the lower boundary and then (more or less) $\flr{\frac{1}{x}}$
times the right boundary before coming back to the lower boundary.
The minus sign in the action $\ss{\left(\begin{array}{cc}
0 & 1\\
-1 & 0
\end{array}\right)}\left(x\right)=-\frac{1}{x}$ above produce an alternation between the left and right boundary.
As can be seen in the example in \figref{CFE_to_GEO}, the geodesic
from $\infty$ to $\frac{3}{7}$ hits two times the left boundary
and then 3 times the right boundary, and the corresponding CFE of
$\frac{3}{7}$ is $\frac{3}{7}=\left[0;2,3\right]$. For the exact
connection between the continued fraction expansion and the diagonal
flow in $\SL_{2}\left(\ZZ\right)\backslash\SL_{2}\left(\RR\right)$,
we refer the reader to section 9.6 in \cite{einsiedler_ergodic_2010}.

The spaces $X=\SL_{2}\left(\ZZ\right)\backslash\SL_{2}\left(\RR\right)$
and $\SL_{2}\left(\ZZ\right)\backslash\HH\cong X_{2}/\SO_{2}\left(\RR\right)$
are quite similar since $\SO_{2}\left(\RR\right)$ is a compact group,
so the picture above gives a very good intuition of what happens in
$X$. For example, an important phenomenon for $A$-orbits of the
form $\SL_{2}\left(\ZZ\right)u_{\alpha}A$ in $X$ with rational $\alpha\in\QQ$,
as can be seen in \figref{CFE_to_GEO}, is that they come and eventually
return to the cusp, and we call such orbits \emph{divergent}.\emph{
}Indeed, this follows from the fact that under the $\SL_{2}\left(\ZZ\right)$-M\"{o}bius
action, all the rational points and $\infty$ are equivalent. This
is the analogue of the fact that continued fraction expansion of rationals
are finite (or equivalently $T^{k}\left(\frac{n}{m}\right)=0$ for
some $k$). We can always define the $A$-invariant measure on the
orbit $\SL_{2}\left(\ZZ\right)gA$ by pushing the standard Lebesgue
measure from $stab\left(\SL_{2}\left(\ZZ\right)g\right)\backslash A$.
Unless the orbit is periodic, or equivalently $stab\left(\SL_{2}\left(\ZZ\right)g\right)$
is a lattice in $A\cong\RR$, then the measure is infinite. However
for divergent orbits this measure is locally finite - the measure
of any compact set is finite (because most of the mass of the measure
is near the cusp). 

Let us denote by $\mu_{n/m}$ this $A$-invariant measure on the orbit
$\SL_{2}\left(\ZZ\right)u_{n/m}A$ and define $\mu_{m}=\frac{1}{\left|\Lambda_{m}\right|}\sum_{n\in\Lambda_{m}}\mu_{n/m}$.
The analogue of the equidistribution $\nu_{m}\wstar\nuga$ in the
continued fraction setting is $\mu_{m}\to\mu_{Haar}$, where $\mu_{Haar}$
is the $\SL_{2}\left(\RR\right)$-invariant measure on $X=\SL_{2}\left(\ZZ\right)\backslash\SL_{2}\left(\RR\right)$.
Note that since $\mu_{m}$ are only locally finite, and not probability
measure, by the limit we mean that $\frac{\mu_{m}\left(f_{1}\right)}{\mu_{m}\left(f_{2}\right)}\overset{q\to\infty}{\to}\frac{\mu_{Haar}\left(f_{1}\right)}{\mu_{Haar}\left(f_{2}\right)}$
for any $f_{1},f_{2}\in C_{c}\left(X\right)$ with $\mu_{Haar}\left(f_{2}\right)\neq0$
(see \subsecref{Locally_finite} for more details about locally finite
measures and their limits).

\begin{figure}[H]
\begin{centering}
\includegraphics[scale=0.35]{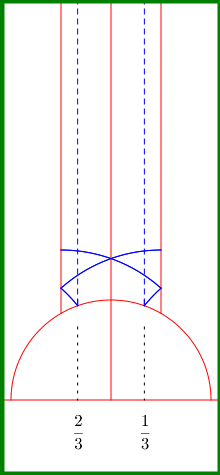}~~\includegraphics[scale=0.35]{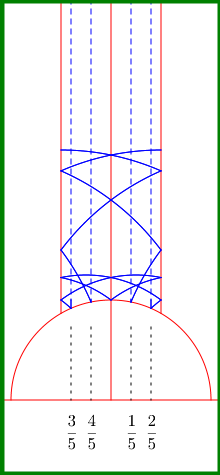}~~\includegraphics[scale=0.35]{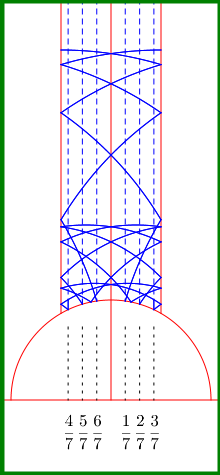}
\par\end{centering}
\caption{\label{fig:all_orbits}From left to right we have the $A$-orbits
on which $\mu_{3},\mu_{5}$ and $\mu_{7}$ are supported.}
\end{figure}
There are three main advantages of working with $X$ instead of with
continued fractions, which can already be seen in the examples in
\figref{all_orbits}:
\begin{enumerate}
\item Unlike the measures $\nu_{m}$ on which we cannot really act with
the Gauss map $T$, the measures $\mu_{m}$ are $A$-invariant. Of
course we didn't get this for free e we now work with infinite locally
finite measures instead of probability measure. Fortunately, most
of the mass of these measures are near the cusp, so as we shall see
later, we can reduce this problem to finite measures. Also, we will
soon see that the time the $A$-orbit-measure $\mu_{n/m}$ spends
in $X$ ``before'' diverging to the cusp (minus their ``vertical''
parts) doesn't depend on $n$. Thus, unlike the continued fraction
result where we had two type of averages in $\nu_{n}$ and $\tilde{\nu}_{n}$,
in this setting we have only one natural way to average.
\item The pictures are symmetric with respect to the $y$-axis. Moreover,
we know that the $A$-orbits leading to rational numbers come and
go to the cusp, but instead of having two vertical lines for each
orbit, to a total of $2\varphi\left(m\right)$, we only have $\varphi\left(m\right)$.
In other words, a geodesic leading to $\frac{n}{m}$, will eventually
go straight to the cusp, but the corresponding vertical line will
be over $\frac{n'}{m}$ for some $n'$. This symmetry can be expressed
in the continued fraction form, but cannot seen as clearly as in \figref{all_orbits}.
\item If we don't fold the orbits into the fundamental domain, and only
consider their vertical geodesics, then all of these are parallel
lines which we can get one from the other by translation in the $x$-coordinate.
This almost correspond to the unipotent flow in $X$ and in general
unipotent flows are much better understood than geodesic flows.
\end{enumerate}
These phenomena are key to proving the equidistribution result.

\newpage{}

\section{\label{subsec:Symmetries_horospheres}Symmetries and hidden horospheres}

Part (3) above is probably one of the most important observations,
since it allow us to use horocycles and not just geodesics.

\begin{figure}[H]
\begin{centering}
\includegraphics[scale=0.4]{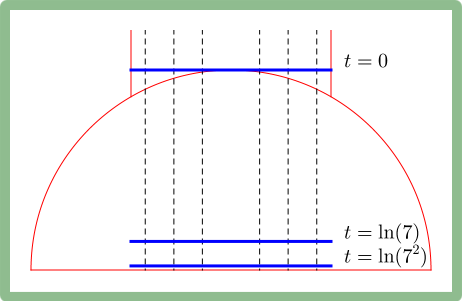}
\par\end{centering}
\caption{\label{fig:horocycles}The dashed lines are the geodesics without
foldings modulo $\protect\SL_{2}\left(\protect\ZZ\right)$. The blue
segments are horocycles e at time $t$ it is in the Euclidean height
$e^{-t}$.}
\end{figure}
For every $h$, the horocycle at height $h$ (blue lines in \figref{horocycles})
contains $\varphi\left(m\right)$ points, one from each geodesic.
Let us assume that as $m\to\infty$, in every height $h$ taking the
uniform average over the $\varphi\left(m\right)$ points is more or
less the uniform measure over the whole horocycle. Our measure is
the average over $\varphi\left(m\right)$ orbits of the diagonal group
$A$, and switching the order of integration we can first take the
discrete average in each height, and then take their averages along
the $A$ direction. With our assumption above, this will be more or
less the same as taking the uniform average in each horocycle, and
then taking the averages of the horocycles.

We now have one of the main results in this space, namely the fact
that expanding horocycles equidistribute: if we take a single horocycle,
e.g. at height $t=0$, and push it by a diagonal element $a\in A$
which expands it (in the picture, this means pushing it down), then
as $a$ increases to infinity the pushed horocycles equidistribute
inside the whole space. More formally:
\begin{defn}[Uniform measure on the standard horocycle]
 Let $\mu_{U}$ be the pushforward of the Lebesgue measure on $\left[0,1\right]$
to $x\mapsto\SL_{2}\left(\ZZ\right)u_{x}$, namely $\mu_{U}\left(f\right):=\int_{0}^{1}f\left(\SL_{2}\left(\ZZ\right)u_{x}\right)\dx$
for $f\in C_{c}\left(X_{2}\right)$.
\end{defn}

\begin{thm}
Expanding horocycles equidistribute e i.e. $a\left(-t\right)_{*}\mu_{U}\wstar\mu_{Haar}$
as $t\to\infty$.
\end{thm}

This result explains why we expect our measures $\mu_{m}$ to equidistribute
in $X$ (or at least their part with $t\geq0$, though as we shall
see, this will be enough). Hence, we only need to justify our assumption
that the discrete average over the $\varphi\left(m\right)$ points
is almost the uniform measure over the corresponding horocycle.

For that, first note that the horocycle at height $1=e^{-0}$ is just
$\left\{ \SL_{2}\left(\ZZ\right)\cdot u_{x}\;\mid\;x\in\RR\right\} $,
and since $u_{1}=\ss{\left(\begin{array}{cc}
1 & 1\\
0 & 1
\end{array}\right)}\in\SL_{2}\left(\ZZ\right)$, the horocycle is homeomorphic to a cycle $\nicefrac{\RR}{\ZZ}$.
Under this identification, the $\varphi\left(m\right)$ points on
it are simply $\frac{1}{m}\Lambda_{m}=\left\{ \frac{n}{m}\;\mid\;1\leq n\leq m,\;\left(n,m\right)=1\right\} $.

\begin{figure}[H]
\begin{centering}
\includegraphics[scale=0.35]{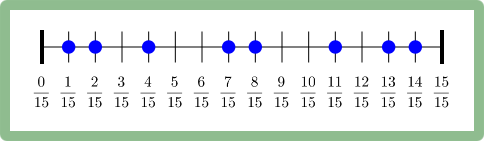}~~~~\includegraphics[scale=0.35]{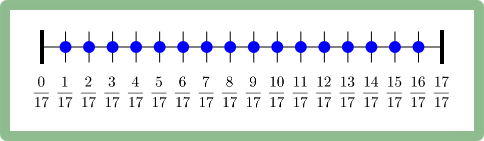}
\par\end{centering}
\caption{\label{fig:coprime} The sets $\frac{1}{15}\Lambda_{15}$ and $\frac{1}{17}\Lambda_{17}$
inside the segment $\left[0,1\right]$. The missing points on the
left correspond to numbers divisable by $3$ and $5$.}
\end{figure}
As can be seen for example in the figure above for $m=17$, for primes
$m$ the sets $\frac{1}{m}\Lambda_{m}$ equidistribute in $\left[0,1\right]$
as $m\to\infty$. If $m$ is not prime, then we get ``holes'' in
$\left[0,1\right]$ for each $n$ not coprime to $m$, and there are
more and more holes the more prime divisors $m$ has. However, as
we shall see later (\lemref{equi_coprime}) using a simple inclusion
exclusion argument, even in these cases the $\varphi\left(m\right)$
points equidistribute in $\left[0,1\right]$.

We can do the same for other horocycles in different heights, but
the argument above is not enough. While all the horocycles are homeomorphic
to cycles, the length of the cycle is not fixed. The lower we are
in \figref{horocycles} the larger the cycle is in the hyperbolic
plane (the distance between the vertical lines increase). In particular,
there is a point where the horocycle is isometric to $\nicefrac{\left[0,m\right]}{0\sim m}$
and the points are simply the integers $n$ with $\left(n,m\right)=1$
so they are at least at distance 1 from each other. The average over
these points and the uniform measure on $\left[0,m\right]$ are quite
far away.

\begin{figure}[H]
\begin{centering}
\includegraphics[scale=0.1]{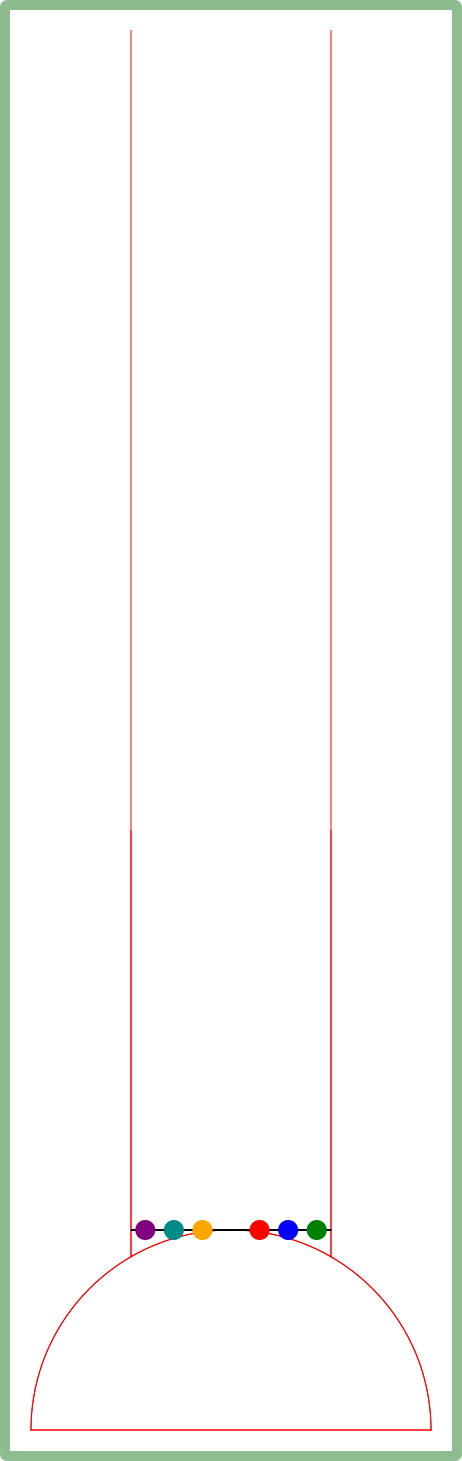}~\includegraphics[scale=0.1]{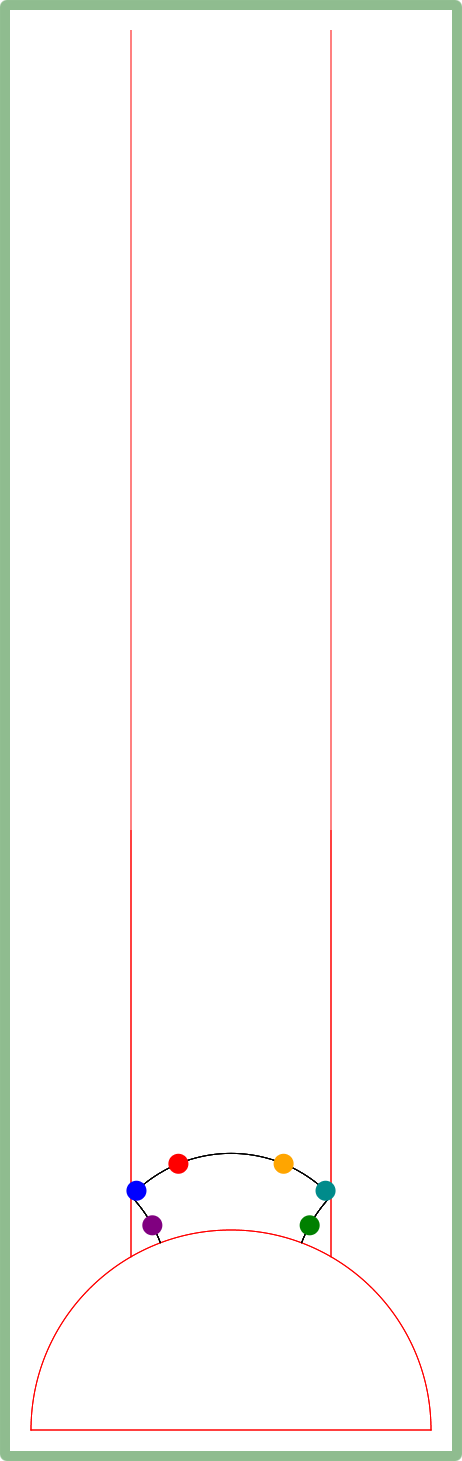}~\includegraphics[scale=0.1]{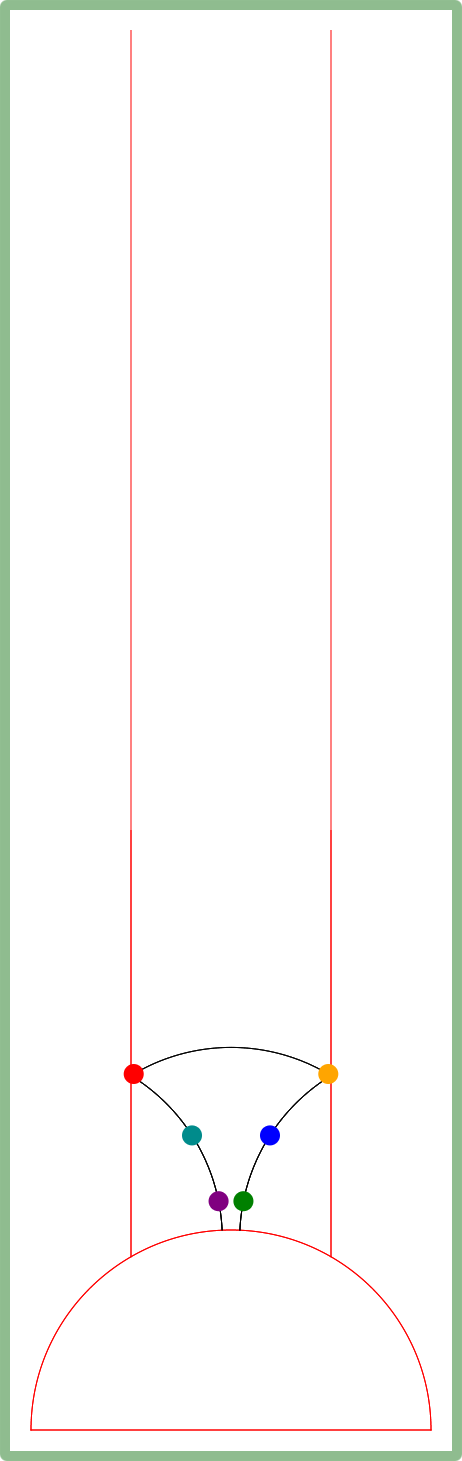}~\includegraphics[scale=0.1]{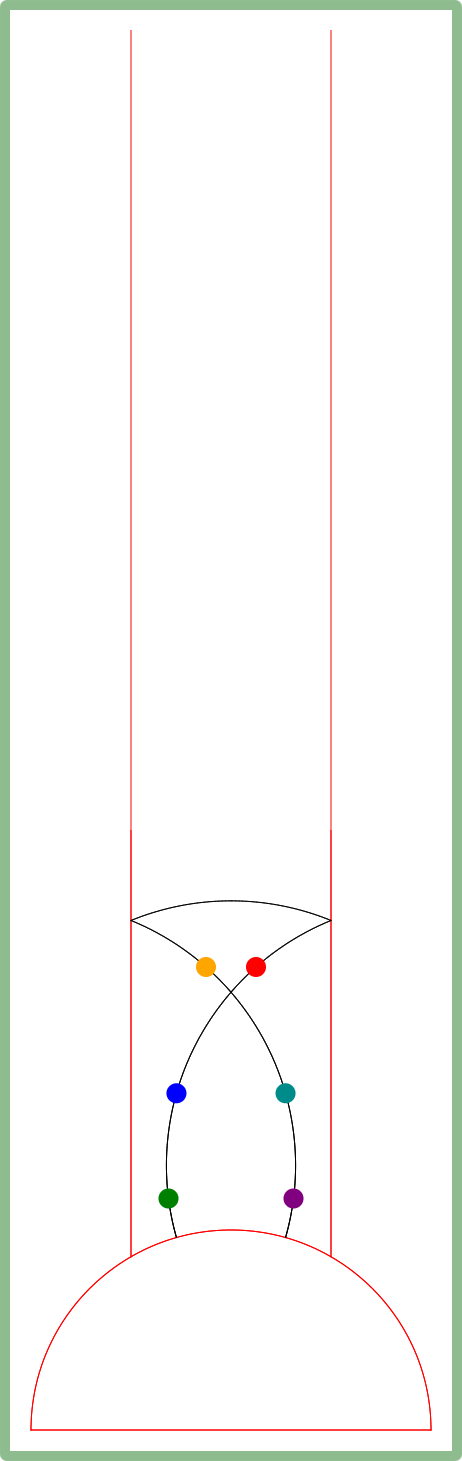}~\includegraphics[scale=0.1]{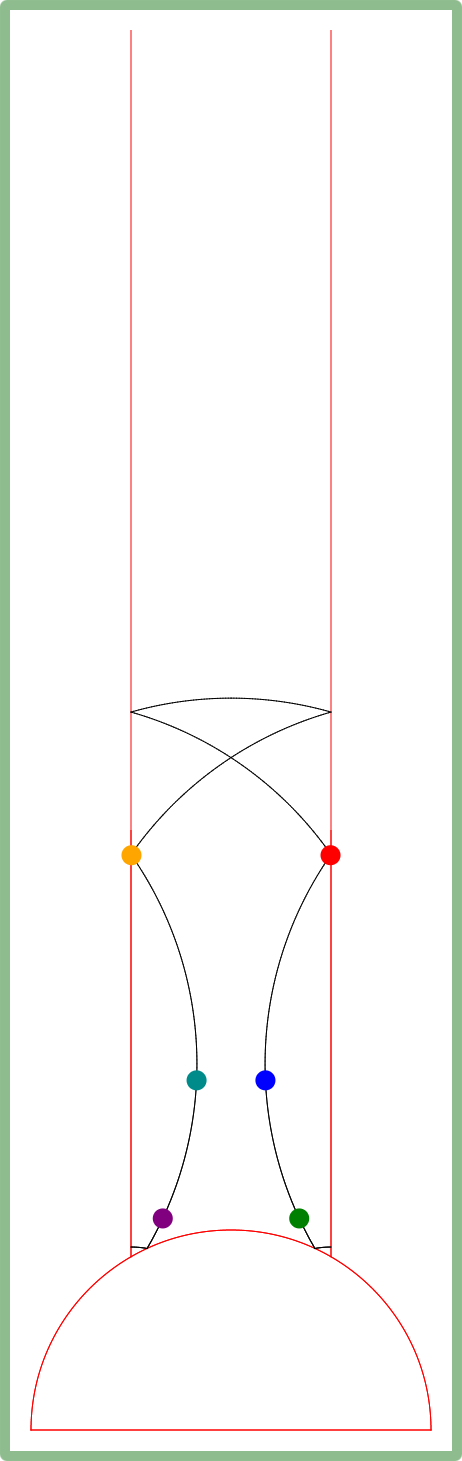}~\includegraphics[scale=0.1]{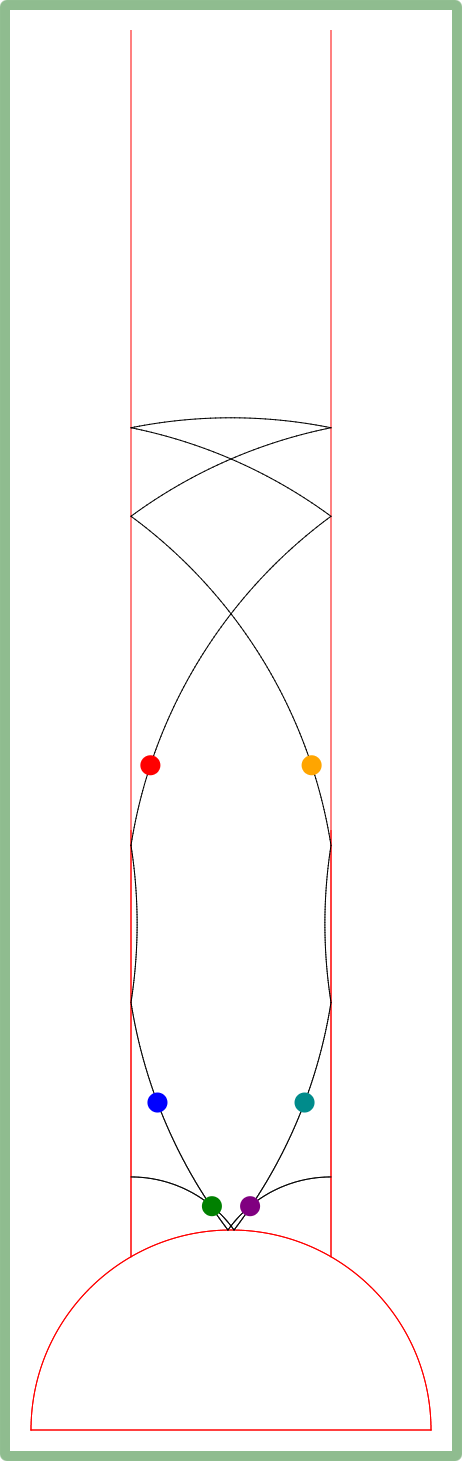}~\includegraphics[scale=0.1]{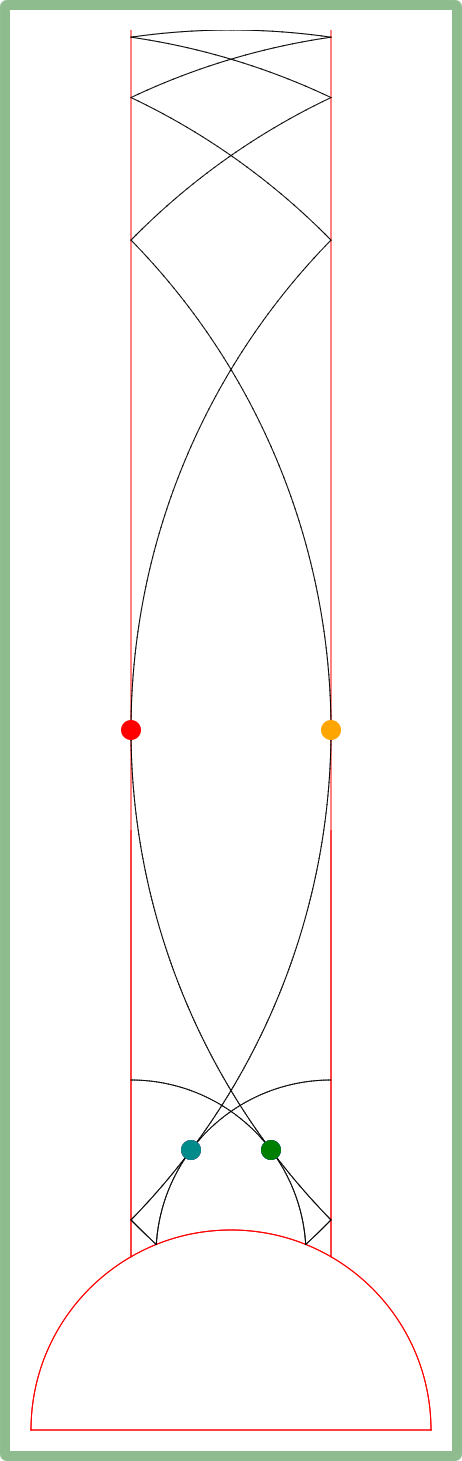}~\includegraphics[scale=0.1]{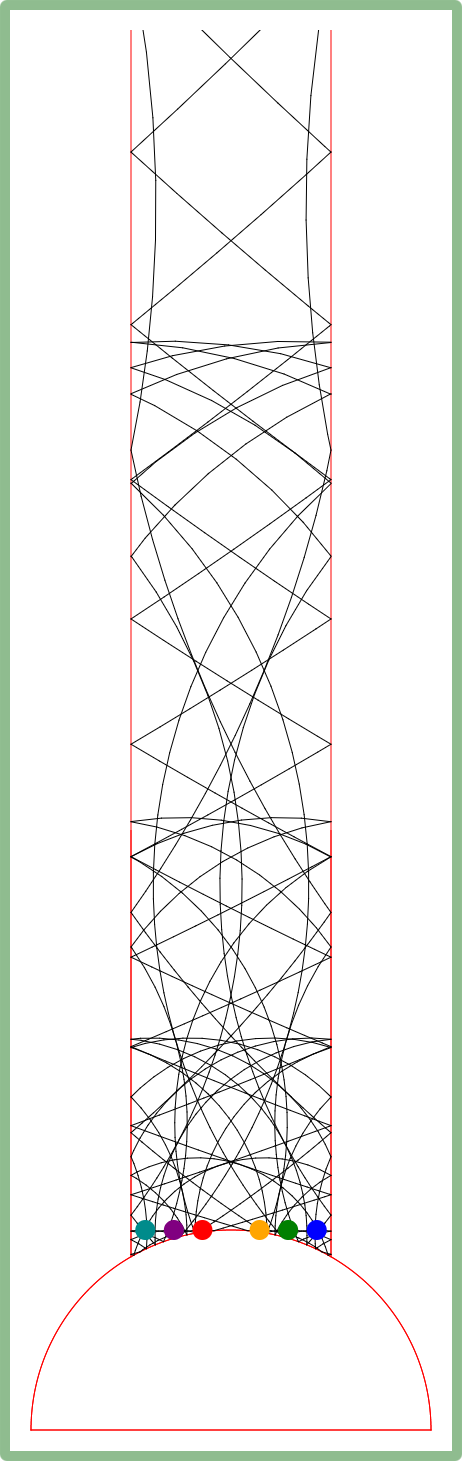}
\par\end{centering}
\caption{\label{fig:expanding_discrete_horocycles}For $m=7$, in the first
7 images from the left we see the $\varphi\left(7\right)$ points
on the corresponding horocycle at times $t=\ln\left(7\right)\cdot\frac{i}{6},\;i=0,1,...,6$.
The right most image is at time $t=2\cdot\ln\left(7\right)$.}
\end{figure}
As can be seen in the image above, at time $0$ our points are very
close to one another along the horocycle. As $t$ increases this distance
becomes longer and longer until at time $t=\ln\left(7\right)$ they
are quite far away. While the distance along the horocycle continues
to grow, when we get to time $2\ln\left(7\right)$ all the points
return to the same horocycle. If we ignore the horocycle itself, then
the pictures at time $t=0$ and $t=2\ln\left(7\right)$ are exactly
the same. This is part of a symmetry that we will use later on which
shows that the time $\left[\ln\left(m\right),2\ln\left(m\right)\right]$
are a mirror image of the times $\left[0,\ln\left(m\right)\right]$.\\

To understand this expanding problem formally we do the following
simple computation:

\[
\SL_{2}\left(\ZZ\right)u_{\left(n+1\right)/m}a\left(t\right)=\SL_{2}\left(\ZZ\right)u_{n/m}a\left(t\right)a\left(-t\right)u_{1/m}a\left(t\right)=\SL_{2}\left(\ZZ\right)u_{n/m}a\left(t\right)\cdot u_{e^{t}/m}.
\]
This tells us that at time $t$, the distance (over the horocycle)
between the points corresponding to $\frac{n}{m}$ and $\frac{n+1}{m}$
is $\frac{e^{t}}{m}$. Hence, as long as $\frac{e^{t}}{m}$ is small,
say $t\leq\left(1-\varepsilon\right)\ln\left(m\right)$ for some fixed
$\varepsilon>0$, the argument that the discrete average is closed
to the uniform average over the horocycle works.

As $t\geq\ln\left(m\right)$ grows, the distance between consecutive
points increases to infinity (along the horocycle), however luckily
for us our orbits have a nice algebraic structure which we can exploit.
There is a symmetry at the time $\leq\ln\left(m\right)$ and at the
times $\geq\ln\left(m\right)$ which we just saw an example in \figref{expanding_discrete_horocycles}

Let us make this notion more precise.

Recall that the point $\SL_{2}\left(\ZZ\right)g\in X$ correspond
to the lattice $\ZZ^{2}\cdot g\subseteq\RR^{2}$, and in particular
we have that $\ZZ^{2}u_{\alpha}=span_{\ZZ}\left\{ \left(0,1\right),\left(1,\alpha\right)\right\} $.
Let $1\leq n\leq m,\;\left(n,m\right)=1$ and consider the lattices
\begin{align*}
L_{n/m} & =\ZZ^{2}u_{\alpha}=span_{\ZZ}\left\{ \left(0,1\right),\left(1,\frac{n}{m}\right)\right\} \\
L_{m} & =\ZZ^{2}\left(\begin{array}{cc}
m & 0\\
0 & 1
\end{array}\right)=span_{\ZZ}\left\{ \left(0,1\right),\left(m,0\right)\right\} .
\end{align*}
Note that $\left(m,0\right)=m\left(1,\frac{n}{m}\right)-n\left(0,1\right)$
is in $L_{n/m}$ so that $L_{m}$ is a sublattice of $L_{n/m}$. Furthermore
$covol\left(L_{m}\right)=m$ and $\nicefrac{L_{n/m}}{L_{m}}\cong\nicefrac{\ZZ}{m\ZZ}$
(because $\left(n,m\right)=1$). As example, see the left picture
in \figref{symmetry_latices} below.

If $L\leq\RR^{2}$ is any lattice which contains $L_{m}$ such that
$\nicefrac{L}{L_{m}}=\nicefrac{\ZZ}{m\ZZ}$ then $L_{m}\leq L\leq\frac{1}{m}L_{m}$
and $\nicefrac{\frac{1}{m}L_{m}}{L}\cong\nicefrac{\ZZ}{m\ZZ}$ as
well. It is now a standard argument to show that $L$ must be $L_{n/m}$
for some $\left(n,m\right)=1$. This already gives us a good starting
point e these $L_{n/m}$ are exactly the representatives of the distinct
$A$-orbits, and this point of view group them naturally together
as certain lattices which contain $L_{m}$.

The lattice $L_{n/m}$ all lie on the horocycle of height $t=0$.
As we want to check what happens at general time $t$, we simply multiply
by $a\left(t\right)$. In particular at time $t=\ln\left(m\right)$,
that we already encountered above, something interesting happens to
$L_{m}$. Indeed, we get
\[
L_{m}a\left(\ln\left(m\right)\right)=\ZZ^{2}\left(\begin{array}{cc}
m & 0\\
0 & 1
\end{array}\right)\left(\begin{array}{cc}
m^{-1/2} & 0\\
0 & m^{1/2}
\end{array}\right)=m^{1/2}\ZZ^{2}
\]
which is just a stretching of the lattice $\ZZ^{2}$, and in particular
it is invariant under the reflection of switching the $x$ and $y$
coordinates. Let us denote this reflection by $\tau$, namely $\tau\left(x,y\right)=\left(y,x\right)$,
so that $\tau\left(L_{m}a\left(\ln\left(m\right)\right)\right)=L_{m}a\left(\ln\left(m\right)\right)$.
It is also easy to check that $L_{m}a\left(2\ln\left(m\right)\right)=\tau\left(L_{m}\right)$.

\begin{figure}[H]
\begin{centering}
\includegraphics[scale=0.25]{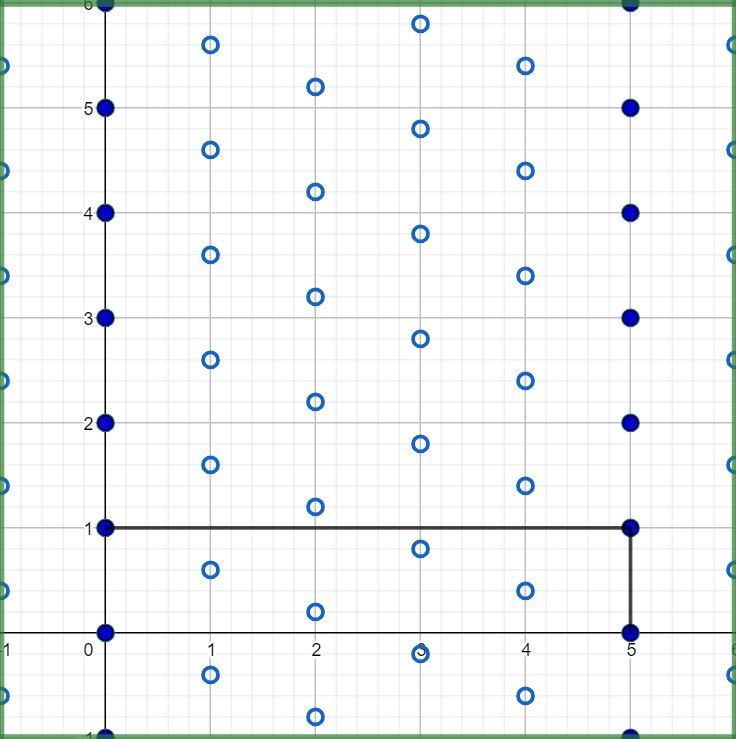}~~\includegraphics[scale=0.25]{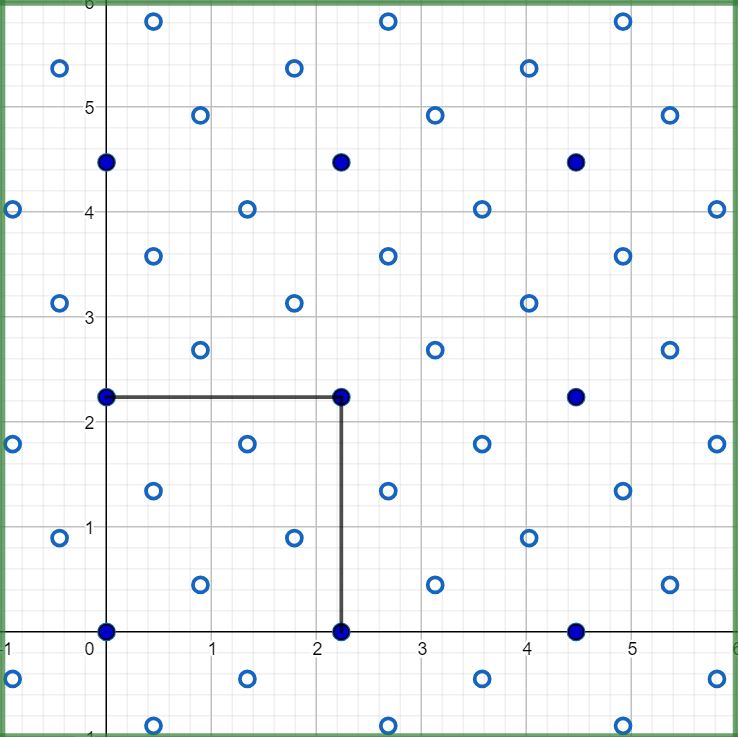}~~\includegraphics[scale=0.25]{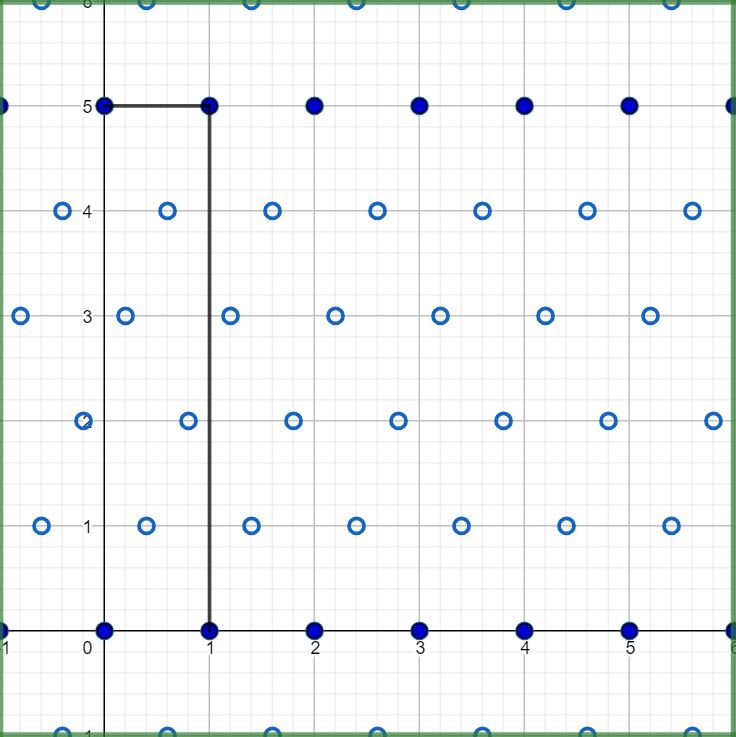}
\par\end{centering}
\caption{\label{fig:symmetry_latices}From left to right are the lattices $L_{3/5}a\left(t\right)$
(circles) and $L_{5}a\left(t\right)$ (full balls) at times $t=i\cdot\ln\left(5\right),\;i=0,1,2$
(GeoGebra \cite{hohenwarter_geogebra:_2002}).}
\end{figure}
As can be seen in \figref{symmetry_latices}, unlike $L_{m}a\left(\ln\left(m\right)\right)$
in the middle picture, the lattice $L_{n/m}a\left(\ln\left(m\right)\right)$
is not invariant under the reflection $x\leftrightarrow y$. Its reflection,
however, looks similar to $L_{n/m}a\left(\ln\left(m\right)\right)$
e both contain $\sqrt{m}\ZZ^{2}$ as a sublattice, and the quotient
is $\nicefrac{\ZZ}{m\ZZ}$. Hence $\tau\left(L_{n/m}a\left(\ln\left(m\right)\right)\right)=L_{n'/m}a\left(\ln\left(m\right)\right)$
for some $n'$ coprime to $m$, and a simple computation shows that
$nn'\equiv_{m}-1$. So while the lattice itself is not invariant under
this reflection, the set $\left\{ L_{n/m}a\left(\ln\left(m\right)\right)\;\mid\;n\in\Lambda_{m}\right\} $
is invariant. We can do the same argument for the left and right images
in \figref{symmetry_latices} which are almost reflections of one
another. So up to changing the $n$ with $n'$, going forward from
time $t=\ln\left(m\right)$ and going backwards are the same up to
this reflection. More formally, we have that 
\[
\tau\left(L_{n/m}a\left(\ln\left(m\right)+t\right)\right)=L_{n'/m}a\left(\ln\left(m\right)-t\right).
\]
Because we want to take the average over all the $n$, we get that
the average at time $\ln\left(m\right)+t$ is the same as the average
at time $\ln\left(m\right)-t$ composed with the reflection. Thus,
if we can show that our measures equidistribute at the times $t\leq\ln\left(m\right)$,
then this reflection implies that they also equidistribute at times
$t\geq\ln\left(m\right)$. With this in mind, we decompose our orbits
to 4 parts (which in \figref{horocycles} are separated by the blue
lines):
\begin{enumerate}
\item $t<0$: The orbit come from the cusp directly to the horocycle $\left\{ \SL_{2}\left(\ZZ\right)u_{x}\;\mid\;0\leq x\leq1\right\} $.
Most of the mass here is near the cusp, so in any way it will not
contribute too much to our measure.
\item $0<t<\ln\left(m\right)$: The orbit flows from the standard horocycle
above up to a $\tau$-invariant set. In these times the discrete points
equidistribute along their corresponding horocycle (up to some small
distance from $\ln\left(m\right)$) and further more, the horocycles
are expanding and therefore they equidistribute in the whole space.
\item $\ln\left(m\right)<t<2\ln\left(m\right)$: The mirror image of $0<t<\ln\left(m\right)$.
\item $2\ln\left(m\right)<t$: The mirror image of $t<0$ e the orbits starts
at the transposed horocycle $\left\{ \SL_{2}\left(\ZZ\right)u_{x}^{T}\;\mid\;0\leq x\leq1\right\} $
and flow directly to the cusp.
\end{enumerate}
Thus, in the end, it is enough to show that our equidistribution argument
works at the times $0<t<\ln\left(m\right)$, and we already have a
good reason for that to hold.

\newpage{}

\section{\label{subsec:The_p_adic_motivation}The $p$-adic motivation}

Using the process and language described in the previous section,
one can already prove full equidistribution. However, this result
becomes much more natural once we begin to use the language of $p$-adic
and adelic numbers. For example, on each horocycle we only have a
finite set of points which only approximate the horocycle. There are
many families of points which can do that, so what natural reason
do we have to choose these exact families, and why do they all fit
together so nicely? If we ever try to generalize this equidistribution
result to higher dimension, which families would we want to choose
then? The language of adelic numbers make this discussion much more
natural and even answer some of these questions.\\

Recall that our measure $\mu_{m}$ as viewed in the hyperbolic plane
looks like $\varphi\left(m\right)$ vertical lines (see \figref{horocycles}),
in particular we can think of it as a single line translated by matrices
of the form $u_{x}$ from the left. However, our whole discussion
is in the modular surface $\SL_{2}\left(\ZZ\right)\backslash\SL_{2}\left(\RR\right)$,
so this left translation is not well defined. In order to make this
process well defined, we need to go to a bigger space which projects
onto the space of unimodular lattices.

At time $t$, our $\varphi\left(m\right)$ points are $\SL_{2}\left(\ZZ\right)u_{n/m}a\left(t\right)$
with $1\leq n\leq m,\;\left(n,m\right)=1$. Thus, in a sense we want
to take $\SL_{2}\left(\ZZ\right)a\left(t\right)$ and ``multiply''
if from the right by $u_{n/m}$ with $n\in\Lambda_{m}$. The matrices
$u_{n/m}$ are all in $\SL_{2}\left(\ZZ\left[\frac{1}{m}\right]\right)$,
which suggests that we might want to work with the space $\SL_{2}\left(\ZZ\left[\frac{1}{m}\right]\right)\backslash\SL_{2}\left(\RR\right)$
instead. However, in this space the points $\SL_{2}\left(\ZZ\left[\frac{1}{m}\right]\right)u_{n/m}a\left(t\right),\;n\in\Lambda_{m}$
are identified to a single point, which is not what we wanted, and
even without this problem, the group $\SL_{2}\left(\ZZ\left[\frac{1}{m}\right]\right)$
is not a lattice in $\SL_{2}\left(\RR\right)$ so we cannot use all
the results about lattices. Fortunately, there is an easy and natural
way to solve it e we look instead on the diagonal embedding $\SL_{2}\left(\ZZ\left[\frac{1}{m}\right]\right)\hookrightarrow\SL_{2}\left(\RR\right)\times\SL_{2}\left(\ZZ\left[\frac{1}{m}\right]\right)$
and the quotient space
\[
\SL_{2}\left(\ZZ\left[\frac{1}{m}\right]\right)\backslash\SL_{2}\left(\RR\right)\times\SL_{2}\left(\ZZ\left[\frac{1}{m}\right]\right).
\]
By definition, every pair $\left(g^{\left(\infty\right)},g^{\left(m\right)}\right)\in\SL_{2}\left(\RR\right)\times\SL_{2}\left(\ZZ\left[\frac{1}{m}\right]\right)$
is equivalent to $\left(\left(g^{\left(m\right)}\right)^{-1}g_{\infty},Id\right)$,
and it is easy to check that the map
\[
\pi_{m}:\SL_{2}\left(\ZZ\left[\frac{1}{m}\right]\right)\left(g^{\left(\infty\right)},g^{\left(m\right)}\right)\mapsto\SL_{2}\left(\ZZ\right)\left(g^{\left(m\right)}\right)^{-1}g_{\infty}\in\SL_{2}\left(\ZZ\right)\backslash\SL_{2}\left(\RR\right),
\]
is a well defined surjective map. It is also not hard to show that
the preimage of $\SL_{2}\left(\ZZ\right)g^{\left(\infty\right)}$
is exactly the orbit $\SL_{2}\left(\ZZ\left[\frac{1}{m}\right]\right)\left(g^{\left(\infty\right)},\SL_{2}\left(\ZZ\right)\right)$.
In particular, the projection $\pi_{m}$ induces an isomorphism 
\[
\SL_{2}\left(\ZZ\left[\frac{1}{m}\right]\right)\backslash\left[\SL_{2}\left(\RR\right)\times\SL_{2}\left(\ZZ\left[\frac{1}{m}\right]\right)\right]/\left[Id\times\SL_{2}\left(\ZZ\right)\right]\cong\SL_{2}\left(\ZZ\right)\backslash\SL_{2}\left(\RR\right).
\]
Philosophically, this new presentation of the space of lattices tell
us that these is a right ``action'' of $\SL_{2}\left(\ZZ\left[\frac{1}{m}\right]\right)/\SL_{2}\left(\ZZ\right)$
on our standard space of unimodular lattices $\SL_{2}\left(\ZZ\right)\backslash\SL_{2}\left(\RR\right)$.\\

For example, our set $\SL_{2}\left(\ZZ\right)u_{n/m}a\left(t\right),\;n\in\Lambda_{m}$
from before can now be seen as the projection of 
\[
\left\{ \SL_{2}\left(\ZZ\left[\frac{1}{m}\right]\right)\left(a\left(t\right),u_{-n/m}\right),\;n\in\Lambda_{m}\right\} =\SL_{2}\left(\ZZ\left[\frac{1}{m}\right]\right)\left(a\left(t\right),Id\right)\cdot\left\{ \left(Id,u_{-n/m}\right),\;n\in\Lambda_{m}\right\} ,
\]
namely it is a projection of an orbit of the set $\left\{ \left(Id,u_{-n/m}\right),\;n\in\Lambda_{m}\right\} $.
Note that the map $x\mapsto u_{x}$ from $\RR$ to $U$ is an isomorphism
so we can identify $\left\{ \left(Id,u_{-n/m}\right),\;n\in\Lambda_{m}\right\} $
with $\frac{1}{m}\Lambda_{m}$. Since we only care about the projection
of this orbit, and the projection is constant on $Id\times\SL_{2}\left(\ZZ\right)$
which contains $\left(Id,u_{1}\right)$, we can actually think of
this set as $\frac{1}{m}\Lambda_{m}\;(mod\;1)$ or $\Lambda_{m}\;(mod\;m)$.

This is already a better presentation, since now we have an actual
orbit. Unfortunately, the matrix multiplication of the $u_{x}$ translate
to addition of $x$, and even when considered mod $m$, the set $\Lambda_{m}$
is not of a group (additively). However, it is a group multiplicatively,
and we want to somehow exploit this fact.

In order to present this as a group action we first write the $u_{n/m}$
as conjugations
\[
u_{n/m}=\left(\begin{array}{cc}
1 & 0\\
0 & n^{-1}
\end{array}\right)u_{-1/m}\left(\begin{array}{cc}
1 & 0\\
0 & n
\end{array}\right).
\]
The set $\left\{ \left(\begin{array}{cc}
1 & 0\\
0 & n^{-1}
\end{array}\right)\;\mid\;n\in\Lambda_{m}\right\} $ is again not a group, but unlike before, when we consider the multiplication
mod $m$ it is a group.

Both the problem that our new set is only a group mod $m$, and that
the matrices $\ss{\left(\begin{array}{cc}
1 & 0\\
0 & n^{-1}
\end{array}\right)}$ are not in $\SL_{2}\left(\ZZ\left[\frac{1}{m}\right]\right)$ (the
determinant is not in $\ZZ\left[\frac{1}{m}\right]^{\times}$) can
be fixed once we move to the groups $\GL_{2}\left(\QQ_{m}\right)$
and $\GL_{2}\left(\ZZ_{m}\right)$ instead of $\SL_{2}\left(\ZZ\left[\frac{1}{m}\right]\right)$
and $\SL_{2}\left(\ZZ\right)$ in the right coordinate of our bigger
space.\\

There are many ways to define the $m$-adic integers $\ZZ_{m}$ and
$m$-adic numbers $\QQ_{m}$, and we refer the reader to \cite{gouvea_p-adic_2003,koblitz_p-adic_1996}
for further details. Probably the most elementary way is as rings
of formal power series where
\begin{align*}
\ZZ_{m} & =\left\{ \sum_{0}^{\infty}a_{j}m^{j}\;\mid\;a_{j}\in\left\{ 0,...,m-1\right\} \right\} .\\
\QQ_{m} & =\left\{ \sum_{N}^{\infty}a_{j}m^{j}\;\mid\;N\in\ZZ,\;a_{j}\in\left\{ 0,...,m-1\right\} \right\} 
\end{align*}

The first important observation is that the map $\rho\left(\sum_{0}^{\infty}a_{j}m^{j}\right)=a_{0}\;(mod\;m)$
from $\ZZ_{m}$ to $\nicefrac{\ZZ}{m\ZZ}$ is a well defined homomorphism
of rings. There are many results which we can get from this homomorphism
into the finite ring $\nicefrac{\ZZ}{m\ZZ}$ (and the similarly defined
homomorphism into the rings $\nicefrac{\ZZ}{m^{k}\ZZ}$). This homomorphisms
are at the core of these $m$-adic numbers and we will use them repeatedly.
As a first example, we use it to study the invertible elements and
the general structure of $\ZZ_{m}$ and $\QQ_{m}$.
\begin{claim}
An element $z=\sum_{0}^{\infty}a_{j}m^{j}\in\ZZ_{m}^{\times}$ is
invertible if and only if $\left(a_{0},m\right)=1$ or equivalently
$\rho\left(z\right)\in\left(\nicefrac{\ZZ}{m\ZZ}\right)^{\times}$. 
\end{claim}

\begin{proof}
Let $\sum_{0}^{\infty}b_{j}m^{j}\in\ZZ_{m}$ and write 
\[
\left(\sum_{0}^{\infty}a_{j}m^{j}\right)\left(\sum_{0}^{\infty}b_{j}m^{j}\right)=\left(\sum_{0}^{\infty}c_{j}m^{j}\right).
\]
The fact that $\rho$ is homomorphism implies in particular that $a_{0}b_{0}\equiv_{m}c_{0}$.
It follows that if $\sum_{0}^{\infty}a_{j}m^{j}\in\ZZ_{m}^{\times}$,
then $a_{0}$ is invertible mod $m$, namely $\left(a_{0},m\right)=1$.
For the other direction, a simple induction argument shows that we
can always choose $b_{j}$ so that $c_{0}=1$ and $c_{j}=0$ for all
$n\geq1$. 

\newpage{}
\end{proof}
\begin{cor}
The following holds:
\begin{enumerate}
\item We can write $\ZZ_{m}^{\times}=\bigcup_{n}\left(n+m\ZZ_{m}\right)$
where $n\in\Lambda_{m}$.
\item We have that $\ZZ_{m}\cap\QQ=\ZZ\left[\frac{1}{p}\;\mid\;\left(p,m\right)=1\right]$.
\item If $p$ is prime, then $\ZZ_{p}=\left\{ 0\right\} \cup\bigsqcup_{0}^{\infty}p^{n}\ZZ_{p}^{\times}$,
$\QQ_{p}=\left\{ 0\right\} \cup\bigsqcup_{-\infty}^{\infty}p^{n}\ZZ_{p}^{\times}$
and $\QQ_{p}$ is a field.
\end{enumerate}
\end{cor}

\begin{proof}
\begin{enumerate}
\item These are exactly the preimages of $\left(\nicefrac{\ZZ}{m\ZZ}\right)^{\times}$
under $\rho$, which by the previous claim form the invertibles in
$\ZZ_{m}$.
\item If $p\in\ZZ$ is coprime to $m$, then by the previous claim $\frac{1}{p}\in\ZZ_{m}$,
implying the $\supseteq$ containment. On the other hand, in order
to show that a rational $\frac{a}{b}\in\ZZ_{m}\cap\QQ$ is in $\ZZ\left[\frac{1}{p}\;\mid\;\left(p,m\right)=1\right]$,
we may assume that the prime divisors of $a$ and $b$ are prime divisors
of $m$. Write $\frac{a}{b}=\prod_{1}^{k}p_{i}^{\ell_{i}}$ with $\ell_{i}\in\ZZ$,
and $p_{i}$ are the distinct primes which divide $m$. Assume without
loss of generality that $\ell_{1}\leq\ell_{2}\leq\cdots\leq\ell_{k}$
so that $\prod_{1}^{k}p_{i}^{\ell_{i}}=m^{\ell_{1}}\prod_{2}^{k}p_{i}^{\ell_{i}-\ell_{1}}$
where $n=\prod_{2}^{k}p_{i}^{\ell_{i}-\ell_{1}}\in\ZZ$ and $n\not\equiv_{m}0$.
Since $n=\sum_{0}^{M}c_{j}m^{j}$ with $c_{0}\neq0$, it follows that
$m^{\ell_{1}}n\in\ZZ_{m}$ if and only if $\ell_{1}\geq0$, which
is equivalent to $\frac{a}{b}\in\ZZ\leq\ZZ\left[\frac{1}{p}\;\mid\;\left(p,q\right)=1\right]$.
\item Every nonzero element in $\ZZ_{p}$ can be written as $p^{N}\sum_{0}^{\infty}a_{j}p^{j}$
with $a_{0}\neq0$ and $N\geq0$, so by part (1) it is in $p^{N}\ZZ_{p}^{\times}$
and a for $\QQ_{p}$ we have the same presentation but $N$ allowed
to be negative as well. Finally, since $p$ is invertible in $\QQ_{p}$,
we conclude that $\QQ_{q}\backslash\left\{ 0\right\} =\bigsqcup_{-\infty}^{\infty}p^{n}\ZZ_{q}^{\times}$
consists of invertible elements, hence $\QQ_{p}$ is a field.
\end{enumerate}
\end{proof}
Part (3) is very important, and shows that every $p$-adic number
can be written as $p^{n}\alpha$ with $\alpha\in\ZZ_{p}$. This idea
allows us to generalize the presentation $\QQ_{p}=\ZZ_{p}\left[\frac{1}{p}\right]$
to the analogue $\GL_{2}\left(\ZZ\left[\frac{1}{p}\right]\right)\GL_{2}\left(\ZZ_{p}\right)=\GL_{2}\left(\QQ_{p}\right)$
for a prime $p$. In turn, we get the analogue for the isomorphism
from the beginning of this section
\[
\GL_{2}\left(\ZZ\left[\frac{1}{p}\right]\right)\backslash\left[\GL_{2}\left(\RR\right)\times\GL_{2}\left(\QQ_{p}\right)\right]/\left[Id\times\GL_{2}\left(\ZZ_{p}\right)\right]\cong\GL_{2}\left(\ZZ\right)\backslash\GL_{2}\left(\RR\right).
\]
For general $m$, one can use the Chinese remainder theorem (and an
equivalent definition of the $p$-adic numbers) to show that $\QQ_{m}\cong\prod\QQ_{p_{i}}$
and $\ZZ_{m}\cong\prod\ZZ_{p_{i}}$ where $p_{i}$ are the distinct
prime divisors of $m$. We can then generalize the isomorphism above
to any natural number $m$.
\begin{rem}
We move from $\SL$ to $\GL$ since it is easier to work with $\GL$
and not to worry about the determinant. Furthermore, the group $\SL$
doesn't act transitively on the space of the generalized lattices
over the adeles. Later on in \subsecref{Adelic_lattices} we will
move to the group $\GL_{2}^{1}$ which is a little bit smaller than
$\GL_{2}$ and is the right generalization of $\SL_{2}\left(\RR\right)$
to the adelic language.
\end{rem}

Returning to our original problem, let us write $u_{-n/m},\;n\in\Lambda_{m}$
as
\[
u_{-n/m}=\left(\begin{array}{cc}
1 & 0\\
0 & n^{-1}
\end{array}\right)u_{-1/m}\left(\begin{array}{cc}
1 & 0\\
0 & n
\end{array}\right)\in\left(\begin{array}{cc}
1 & 0\\
0 & n^{-1}
\end{array}\right)u_{-1/m}\cdot\GL_{2}\left(\ZZ_{m}\right).
\]
It then follows that our points $\GL_{2}\left(\ZZ\right)u_{n/m}a\left(t\right),\;n\in\Lambda_{m}$
are the projection of 
\[
\GL_{2}\left(\ZZ\left[\frac{1}{m}\right]\right)\left\{ \left(a\left(t\right),\left(\begin{array}{cc}
1 & 0\\
0 & n^{-1}
\end{array}\right)\right),\;n\in\Lambda_{m}\right\} \cdot\left(Id,u_{-1/m}\right).
\]
This already looks like a translation with $u_{-1/m}$ of (part of)
the diagonal orbit in $\GL_{2}\left(\ZZ_{m}\right)$. We claim that
the rest of the diagonal orbit can be decomposed to $\varphi\left(m\right)$
cosets, each containing a different $\ss{\left(\begin{array}{cc}
1 & 0\\
0 & n^{-1}
\end{array}\right)}$ with $n\in\Lambda_{m}$, and when projected down to $\GL_{2}\left(\ZZ\right)\backslash\GL_{2}\left(\RR\right)$
each of these cosets are mapped to a single point. Since the cosets
of a single group have the same mass, the projected measure is going
to be uniform (for the full details, see \subsecref{Uniform_prime_invariance})

In other words our measure $\mu_{m}$ which are uniform averages over
$\varphi\left(m\right)$ orbit measure in in the space of Euclidean
lattices are the projection of a single orbit measure 
\[
\GL_{2}\left(\ZZ\left[\frac{1}{m}\right]\right)A_{m}\left(Id,u_{-1/m}\right)
\]
where $A_{m}\leq\GL_{2}\left(\RR\right)\times\GL_{2}\left(\ZZ_{m}\right)$
is the diagonal subgroup.

While we translate by $u_{-1/m}\in\GL_{2}\left(\QQ_{m}\right)$ in
the big new space, when we project it down to $\GL_{2}\left(\ZZ\right)\backslash\GL_{2}\left(\RR\right)$
we only care about the translating element as in $\nicefrac{\GL_{2}\left(\QQ_{m}\right)}{\GL_{2}\left(\ZZ_{m}\right)}$.
This space, just like $\GL_{2}\left(\ZZ\right)\backslash\GL_{2}\left(\RR\right)$,
can be identified as the space of lattices in $\QQ_{m}$ (see \subsecref{Adelic_lattices}).
This space also comes with a geometric interpretation, and in particular
for $m=p$ primes the space $\nicefrac{\PGL_{2}\left(\QQ_{p}\right)}{\PGL_{2}\left(\ZZ_{p}\right)}$
can be viewed as the $p+1$-regular tree. We will not use this interpretation
here, and we refer the interested reader to \cite{lubotzky_discrete_1994}.
However, we do want to show how we can see the symmetry of our orbits
in this new language.

Recall, that our orbits have symmetry around the time $t=\ln\left(m\right)$.
In this new bigger space, the points at this time are
\begin{align*}
\GL_{2}\left(\ZZ\left[\frac{1}{m}\right]\right)\left(a\left(\ln\left(m\right)\right),a^{\left(m\right)}\right)\left(Id,u_{-1/m}\right) & =\GL_{2}\left(\ZZ\left[\frac{1}{m}\right]\right)\left(\left(\begin{array}{cc}
m^{-1} & 0\\
0 & 1
\end{array}\right)m^{1/2},a^{\left(m\right)}\right)\left(Id,u_{-1/m}\right)\\
 & =\GL_{2}\left(\ZZ\left[\frac{1}{m}\right]\right)\left(m^{1/2},a^{\left(m\right)}\right)\left(Id,\left(\begin{array}{cc}
m & 0\\
0 & 1
\end{array}\right)u_{-1/m}\right),
\end{align*}
where we use the fact that $\ss{\left(\begin{array}{cc}
m & 0\\
0 & 1
\end{array}\right)}\in\GL_{2}\left(\ZZ\left[\frac{1}{m}\right]\right)$ and diagonal elements commute. In our space of $m$-adic lattice
$\nicefrac{\GL_{2}\left(\QQ_{m}\right)}{\GL_{2}\left(\ZZ_{m}\right)}$,
the point $\ss{\left(\begin{array}{cc}
m & 0\\
0 & 1
\end{array}\right)}u_{-1/m}$ correspond to
\[
\left(\begin{array}{cc}
m & 0\\
0 & 1
\end{array}\right)u_{-1/m}\ZZ_{m}^{2}=\left(\begin{array}{cc}
m & -1\\
0 & 1
\end{array}\right)\ZZ_{m}^{2}.
\]
For simplicity, considering the lattice $\ss{\left(\begin{array}{cc}
m & -1\\
0 & 1
\end{array}\right)}\ZZ^{2}$ instead, we see that an equivalent definition is 
\[
\left\{ \left(k_{1},k_{2}\right)\in\ZZ^{2}\;\mid\;k_{1}+k_{2}\equiv_{m}0\right\} .
\]
Clearly this lattice is invariant under our symmetry $\tau$ which
switches the $x$ and $y$ coordinates. We actually also see that
$m^{1/2}$ makes an appearance, which is the normalization that appeared
in \figref{symmetry_latices} and the argument after it. Thus, the
better translation choice should be $\ss{\left(\begin{array}{cc}
m & -1\\
0 & 1
\end{array}\right)}$, and we will see it again more formally in \subsecref{Uniform_prime_invariance},
but on the other hand $u_{-1/m}$ has the advantage of being unipotent,
so we will keep it.

\newpage{}

To summarize what we saw so far:
\begin{enumerate}
\item We start with a problem about continued fractions of rational numbers.
\item We saw how to translate this problem to the space of unimodular lattices
and $A$-orbits there. In this language our measure was an average
of $\varphi\left(q\right)$ $A$-orbits.
\item Using Fubiny we rewrote the measure as an $A$-orbit of $\varphi\left(m\right)$
points on a single horocycle, and we explained why the average on
these points is close to the uniform average on the whole cocycle.
This was true for half of the $A$-orbit, and the other half was a
mirror image of the first.
\item We then lifted the problem to the $m$-adic number, where the $\varphi\left(m\right)$
different $A$-orbits are combined to a translation of a single $A_{m}$-orbit. 
\item We are now left to show that as $m\to\infty$, the translated orbit
$\GL_{2}\left(\ZZ\left[\frac{1}{m}\right]\right)A_{m}\left(Id,u_{-1/m}\right)$
``equidistributes''. Note that this is still not well defined, because
for each $m$ this translated orbit lives in a different space $\GL_{2}\left(\ZZ\left[\frac{1}{m}\right]\right)\backslash\left[\GL_{2}\left(\RR\right)\times\GL_{2}\left(\QQ_{m}\right)\right]$.
\end{enumerate}
The last statement should be very familiar to people in homogeneous
dynamics and ergodic theory. Indeed, this is the famous shearing effect.
If we have a measure on diagonal orbit and we shear it (translate)
by a unipotent element, then the resulting measure will be close to
an average over expanding horocycles. Since expanding horocycles equidistribute,
then we should expect our measures to equidistribute. So far our translation
was only in the $m$-adic coordinate, but of course we can do it in
the real coordinate as well. If we restrict to only translation in
the real coordinate, then this was done in \cite{oh_limits_2014}.
We will give some of the intuition in below in \subsecref{Shearing}
where we will see the analogues of some of the results that we have
seen so far for the translation in the $m$-adic place. Finally, what
we will want to do is to combine translations in the real and in the
$m$-adic place. There are two main issues when doing this combination.
First, while the real and $m$-adic places behave similarly in many
ways, there are still some differences, and there is some technicalities
when trying to combine them. Second, when $m\to\infty$, the space
$\QQ_{m}$ can change (recall that it only depends on the primes which
divide $m$). For that we will define the adelic numbers in \subsecref{shearing_Adelic_intuition}
which contain all the $m$-adic numbers. Finally, we will need to
show how to lift equidistribution results from $\RR$ and $m$-adic
spaces to the whole adelic space.

\newpage{}

\section{\label{subsec:Shearing}Shearing and equidistribution}

In this section we give some of the intuition and ideas for the equidistribution
resulting from shearing, namely a translation of a fixed diagonal
orbit by a unipotent matrix. This is true in quite a general setting,
though for our example, we will concentrate on $\SL_{2}\left(\ZZ\right)\backslash\SL_{2}\left(\RR\right)$
and the standard $A$-orbit through the origin.

As usual, to visualize this space, we look instead on the hyperbolic
plane, where the $A$-orbit $\SL_{2}\left(\ZZ\right)A$ there is simply
the $y$-axis
\[
A\left(i\right)=\left\{ a\left(t\right)i\;\mid\;t\in\RR\right\} =\left\{ \frac{i}{e^{t}}\;\mid\;t\in\RR\right\} .
\]
The translated orbit if $\SL_{2}\left(\ZZ\right)Au_{x}$ for $x\in\RR$,
which on the hyperbolic space is
\[
Au_{x}\left(i\right)=A\left(i+x\right)=\left\{ \frac{i+x}{e^{t}}\;\mid\;t\in\RR\right\} .
\]

\begin{rem}
Note that up until now we had $\SL_{2}\left(\ZZ\right)u_{x}A$ because
the translation was in the $m$-adic place. Now the translation is
in the real place so we need to switch between $u_{x}$ and $A$.
\end{rem}

This set $Au_{x}\left(i\right)$ is again a line which goes through
the origin, and the bigger $x$ is, the smaller the slope is. Folding
this line modulo $\left\langle u_{1}\right\rangle \leq\SL_{2}\left(\ZZ\right)$
we already get this picture of union of almost horizontal lines:

\begin{figure}[H]
\begin{centering}
\includegraphics[scale=0.2]{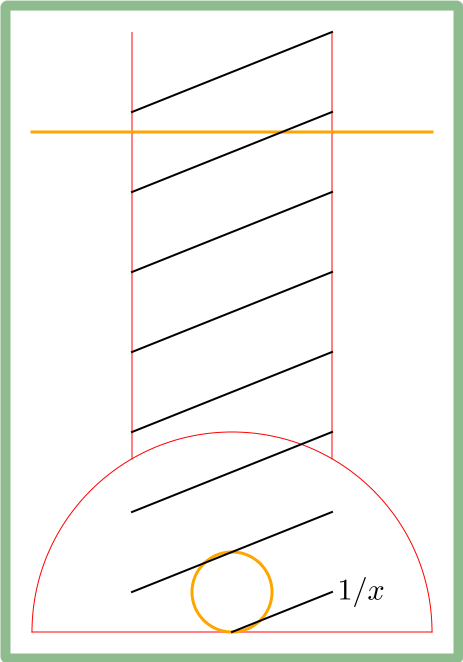}~~~~\includegraphics[scale=0.2]{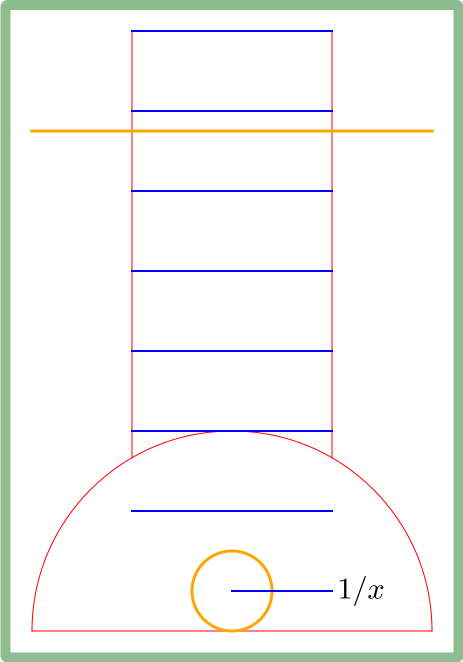}
\par\end{centering}
\caption{\label{fig:shearing_decomposition-1}The black lines are part of the
curve $\protect\SL_{2}\left(\protect\ZZ\right)Au_{x}$ as viewed in
the hyperbolic plane modulo $u_{1}$. On the right, the blue lines
are ``approximations'' of the black lines which are horocycles.
The orange circle through the origin is the image of the orange line
after applying the M\"{o}bius action $z\protect\mapsto-\frac{1}{z}$,
so that the area above the line and inside the circle is a neighborhood
of the cusp.}
\end{figure}
Next, we try to approximate each black line in the image above with
the corresponding blue horocycle line. In order to do that we write
\[
\SL_{2}\left(\ZZ\right)a\left(T+t\right)u_{x}=\SL_{2}\left(\ZZ\right)u_{xe^{-T}e^{-t}}a\left(T\right)a\left(t\right)
\]
for the integration at times $\left[T,T+\delta\right]$.

\newpage{}

There are three parts to this expression:
\begin{enumerate}
\item $\SL_{2}\left(\ZZ\right)u_{xe^{-T}e^{-t}}$ : This part is on the
standard horocycle $\SL_{2}\left(\ZZ\right)U$. However the integral
over $t$ is not the Haar measure there, since $t\mapsto xe^{-T}e^{-t}$
is not linear.
\item Multiplication by $a\left(T\right)$ : This will take the horocycle
from above and expand it (push it down).
\item Multiplication by $a\left(t\right)$ : When $t$ is very small, this
will be negligible.
\end{enumerate}
Part (3) is the simplest one. If $f$ is any compactly supported continues
function and $\varepsilon>0$, then by uniform continuity there is
some $\delta>0$ such that if $\norm g<\delta$ for some $g\in\SL_{2}\left(\RR\right)$,
then $\norm{g\left(f\right)-f}_{\infty}<\varepsilon$. In particular,
if we assume that $t\in\left[0,\delta\right]$, then we can remove
the multiplication by $a\left(t\right)$ up to some $\varepsilon$
error which will be as small as we want. Thus, let us ignore this
$a\left(t\right)$ from now on.

For part (1), in order to get the Haar measure on the standard horocycle,
and not an exponential movement, let us set $s=xe^{-T}e^{-t}$. We
then get that 
\[
\int_{0}^{\delta}f\left(\SL_{2}\left(\ZZ\right)u_{xe^{-T}e^{-t}}a\left(T\right)\right)\dt=\int_{xe^{-T}e^{-\delta}}^{xe^{-T}}f\left(\SL_{2}\left(\ZZ\right)u_{s}a\left(T\right)\right)\frac{1}{s}\ds.
\]
The last integral is almost the Haar measure on the standard horocycle,
where the only problem is the $\frac{1}{s}$. Fortunately, it is inside
$\left[e^{-\delta}\frac{x}{e^{T}},\frac{x}{e^{T}}\right]$ and $e^{-\delta}$
is very close to 1, so $\frac{1}{s}$ is almost the constant $\frac{e^{T}}{x}$. 

Once we changed $\frac{1}{s}$ to a constant, we integrate over the
whole horocycle $\flr{xe^{-T}\left(1-e^{-\delta}\right)}$ times plus
an extra $xe^{-T}\left(1-e^{-\delta}\right)-\flr{xe^{-T}\left(1-e^{-\delta}\right)}$.
Since $\delta$ is now fixed, if $\frac{x}{e^{T}}$ is large, then
this extra integration will be very small. Hence, for example we can
assume that $T\leq\ln\left(x\right)\left(1-\varepsilon\right)$. 

Putting everything together, we get that up to some small error we
have that {\small{}
\[
\frac{1}{\delta}\int_{0}^{\delta}f\left(\SL_{2}\left(\ZZ\right)a\left(T+t\right)u_{x}\right)\dt\sim\frac{1}{\delta}\frac{e^{T}}{x}\flr{xe^{-T}\left(1-e^{-\delta}\right)}\int_{0}^{1}f\left(\SL_{2}\left(\ZZ\right)u_{s}a\left(T\right)\right)\ds\sim\int_{0}^{1}f\left(\SL_{2}\left(\ZZ\right)u_{s}a\left(T\right)\right)\ds.
\]
}If we also assume that $T$ is not too small, say $T\geq\varepsilon\ln\left(x\right)$,
then by the equidistribution of expanding horocycles, the last term
is more or less $\mu_{Haar}\left(f\right)$.

Decomposing the segment $\left[0,\ln\left(x\right)\right]$ to $\delta=\delta\left(T,x,\varepsilon\right)$
intervals, we get that up to an error as small as we want (depending
on $\varepsilon$ and $\norm f_{\infty}$) we have that 
\[
\frac{1}{\ln\left(x\right)}\int_{0}^{\ln\left(x\right)}f\left(\SL_{2}\left(\ZZ\right)a\left(t\right)u_{x}\right)\dx\sim\mu_{Haar}\left(f\right).
\]

This takes care of a big part of our translated orbit. For times $t<0$,
the translated orbit goes to the cusp, so most of it doesn't contribute
to the integral (above the orange line in \figref{shearing_decomposition-1}),
and the time before it leaves the support of $f$ is uniformly bounded,
so all together it contributes a constant (which depends only on $\norm f$
and $supp\left(f\right)$) e see \lemref{restricted_measures} for
the exact details. This constant is of course negligible with respect
to our normalization by $\frac{1}{\ln\left(x\right)}$ as $x\to\infty$.

As in our study of continued fractions of rational numbers, our argument
fail when $T\geq\ln\left(x\right)$ is too large, but here too there
is a symmetry which comes to help us. However, for this argument we
need to perturb a bit our measure as follows. Instead of translating
by the unipotent matrix $u_{x}$, we will do it using 
\[
h\left(y\right)=\left(\begin{array}{cc}
\cosh\left(y/2\right) & \sinh\left(y/2\right)\\
\sinh\left(y/2\right) & \cosh\left(y/2\right)
\end{array}\right).
\]
The first important observation is that 
\begin{align*}
h\left(y\right) & =\left(\begin{array}{cc}
\cosh^{-1/2}\left(y\right) & 0\\
0 & \cosh^{1/2}\left(y\right)
\end{array}\right)\left(\begin{array}{cc}
1 & \sinh\left(y\right)\\
0 & 1
\end{array}\right)\left[\overbrace{\left(\cosh\left(y\right)\right)^{-1/2}\left(\begin{array}{cc}
\cosh\left(y/2\right) & -\sinh\left(y/2\right)\\
\sinh\left(y/2\right) & \cosh\left(y/2\right)
\end{array}\right)}^{k\left(y\right)}\right]\\
\Rightarrow h\left(y\right) & \in a\left(\ln\left(\cosh\left(y\right)\right)\right)\cdot u_{\sinh\left(y\right)}\cdot k\left(y\right),\qquad k\left(y\right)\in\SO_{2}\left(\RR\right).
\end{align*}
This means that 
\[
\SL_{2}\left(\ZZ\right)Ah\left(y\right)=\SL_{2}\left(\ZZ\right)Au_{\sinh\left(y\right)}k\left(y\right).
\]
Since $k\left(y\right)$ is in a compact set, the limit as $y\to\infty$
equidistribute if and only if $\SL_{2}\left(\ZZ\right)Au_{\sinh\left(y\right)}$
equidistribute, and this is exactly the shearing that we discussed
before. Also, since \\
$\left|\sinh\left(y\right)-\cosh\left(y\right)\right|=e^{-y}\to0$
as $y\to\infty$, for $x=\sinh\left(y\right)$ very big we get that
\[
a\left(\ln\left(\cosh\left(y\right)\right)\right)\cdot u_{\sinh\left(y\right)}\sim a\left(\ln\left(x\right)\right)u_{x},
\]
so this translation already has inside it our center of symmetry,
which is at time $\ln\left(x\right)$ (this is the analogue to translation
by $\ss{\left(\begin{array}{cc}
m & 0\\
0 & 1
\end{array}\right)}u_{-1/m}$ instead of $u_{-1/m}$ that we saw in the $m$-adic translation).

Recall that in our discussion of the continued fraction of rational
numbers we denoted by $\tau$ the reflection $\tau=\ss{\left(\begin{array}{cc}
0 & 1\\
1 & 0
\end{array}\right)}$. It is now easy to see that $\tau h\left(y\right)=h\left(y\right)\tau$
and $a\left(t\right)\tau=\tau a\left(-t\right)$. We then get that
\[
\GL_{2}\left(\ZZ\right)a\left(t\right)h\left(y\right)\tau=\GL_{2}\left(\ZZ\right)\tau a\left(-t\right)h\left(y\right)=\GL_{2}\left(\ZZ\right)a\left(-t\right)h\left(y\right),
\]
so that flowing forward and backward in time is the same up to this
reflection. In particular, in order to show equidistribution when
$y\to\infty$, it is enough to prove it for $t\in[0,\infty)$.

Going back to the unipotent translation by $u_{x}$, we get that the
times $[\ln\left(x\right),\infty)$ are the mirror image of $(-\infty,\ln\left(x\right)]$,
so it is enough to show equidistribution for $t\leq\ln\left(x\right)$
which we already have shown.

Now that we have both the equidistribution for the $m$-adic translations,
and the real translations, we can combine both of these ideas to get
equidistribution for combined translations. This will be done in \secref{real_case},
and other than it being more technical, it contains no new ideas.
\begin{rem}
It is interesting to understand the symmetry mentioned above in the
hyperbolic plane. The curve $h\left(x\right)\left(i\right),\;x\in\RR$
is simply the upper half of the unit circle $y=\sqrt{1-x^{2}}$. This
means that the mid point in this visualization is on that half circle.
Our two cutoffs right after our measure ``comes'' from the cusp
and before it ``returns'' to the cusp are when it intersect the
standard horocycle and its reflection via $z\mapsto-\frac{1}{z}$.

\newpage{}
\end{rem}

\section{\label{subsec:shearing_Adelic_intuition}Shearing over the adeles}

Up until now we saw two types of equidistribution phenomena. The first
started with measures coming from continued fractions of rational
numbers, which we reinterpreted as measures on the space of unimodular
lattices $\SL_{2}\left(\ZZ\right)\backslash\SL_{2}\left(\RR\right)$,
where each one was a finite average over $A$-orbits. We then saw
a more natural point of view for each such measure, namely as a projection
of a translation of a single diagonal orbit in $\GL_{2}\left(\ZZ\left[\frac{1}{m}\right]\right)\backslash\GL_{2}\left(\RR\right)\times\GL_{2}\left(\QQ_{m}\right)$
where the translation is done in the $\QQ_{m}$-coordinate by the
matrix $u_{-1/m}$.

The phenomena of translating a diagonal orbit with a unipotent matrix
is called shearing. To get some intuition, we saw how the shearing
process in $\SL_{2}\left(\ZZ\right)\backslash\SL_{2}\left(\RR\right)$
leads to expanding horocycles, which we can use to prove equidistribution.
Actually, the measures in this case can also be seen as projections
of a translated diagonal orbit from $\PGL_{2}\left(\ZZ\left[\frac{1}{m}\right]\right)\backslash\PGL_{2}\left(\RR\right)\times\PGL_{2}\left(\QQ_{m}\right)$,
where the translation is done in the $\RR$-coordinate by the matrix
$u_{x}$. 

The same shearing argument almost works for our original equidistribution
result. The problem is that as $m_{i}\to\infty$, the set $\QQ_{m_{i}}$
can change. However, if all the $m_{i}$ are divisible only by primes
from a finite set $\left\{ p_{1},...,p_{k}\right\} $, then $\QQ_{m_{i}}\leq\QQ_{\prod p_{i}}=\prod\QQ_{p_{i}}$
in which case a similar shearing argument will work. However, if this
is not the case, we need to take the product over all the primes,
which lead to the definition of the adeles.

The main result of this notes is to combine the two equidistribution
results that we saw so far, and to show that they still hold when
we need infinitely many primes. As we have already seen, the steps
in both of these results are quite similar e finding the important
part of the measure, where the rest is ``near'' the cusp, use a
symmetry argument to cut the important part to two so we can then
approximate the measure with expanding horocycles, and finally use
the result about expanding horocycles. However, there are some differences,
mainly where we approximate the horocycles using the shearing effect
in the $\RR$-coordinate, or approximating using averages on discrete
points so that the shearing is in the $\QQ_{m}$-coordinate.

In order to continue our investigation, we first need a better understanding
of the adele ring, and we begin first with some topological properties
of the $p$-adic numbers for primes $p$.
\begin{defn}
Let $z=\sum_{N}^{\infty}a_{j}m^{j}\in\QQ_{m}$ with $a_{N}\neq0$.
Define the $m$-adic \emph{valuation} and \emph{norm} to be
\begin{align*}
val_{m}\left(z\right) & =N\\
\left|z\right|_{m} & =m^{-N}=m^{-val_{m}\left(z\right)}.
\end{align*}
For $0$ we define $val_{m}\left(0\right)=\infty$ and $\left|0\right|_{m}=0$.
\end{defn}

\begin{claim}
Let $p$ be a prime number. Then the following holds.
\begin{enumerate}
\item The function $\left|\cdot\right|_{p}$ satisfies:
\begin{enumerate}
\item For all $q\in\QQ_{p}$ we have $\left|q\right|_{p}\geq0$ with equality
if and only if $q=0$.
\item (strong triangle inequality) For every $z,w\in\QQ_{p}$ we have that
$\left|z+w\right|_{p}\leq\max\left(\left|z\right|_{p},\left|w\right|_{p}\right)$
with equality if $\left|z\right|_{p}\neq\left|w\right|_{p}$.
\item (multiplicative) For every $z,w\in\QQ_{p}$ we have that $\left|zw\right|_{p}=\left|z\right|_{p}\left|w\right|_{p}$.
\end{enumerate}
\item $\ZZ_{p}$ is a compact ring inside $\QQ_{p}$.
\item $\QQ_{p}$ is a complete field with respect to the $p$-adic norm.
\end{enumerate}
\newpage{}
\end{claim}

\begin{proof}
\begin{enumerate}
\item This is elementary and is left to the reader.
\item Note that if $z=\sum_{N}^{\infty}a_{n}p^{n},w=\sum_{N}^{\infty}b_{n}p^{n}\in\QQ_{p}$,
then $\left|z-w\right|_{p}\leq p^{-M}$ is exactly the same as $a_{i}=b_{i}$
for all $i<M$. \\
If $z_{i}=\sum_{0}^{\infty}a_{i,j}p^{j}\in\ZZ_{p}$ is any sequence,
then we can use a diagonal argument to find a subsequence where for
each $j$ the sequence $a_{i,j}$ stabilizes to some $a_{j}$ (recall
that $a_{i,j}\in\left\{ 0,...,p-1\right\} $). From the remark above,
it is now easy to check that $\sum_{0}^{\infty}a_{j}p^{j}\in\ZZ_{p}$
is the limit of this subsequence. Thus, $\ZZ_{p}$ is sequentially
compact and therefore compact.
\item By definition, the closed (and open) balls in $\QQ_{p}$ are exactly
$p^{n}\ZZ_{p}$. If $z_{i}$ is any Cauchy sequence, it must be in
one of thee balls, which by part (2) is compact, so the limit of $z_{i}$
exists and must be in the same ball and therefore in $\QQ_{p}$. We
conclude that $\QQ_{p}$ is complete.
\end{enumerate}
For uniformity of notation, we will use $\QQ_{\infty}$ to denote
$\RR$. We set $\PP$ to be the set of prime numbers in $\NN$ and
$\PP_{\infty}=\PP\cup\left\{ \infty\right\} $. 
\end{proof}
\begin{rem}
Note that $\QQ_{p}$ is a complete field, just like $\RR$, and algebraically
speaking $\ZZ_{p}$ behaves similar to $\ZZ$ e for example both are
generated as a topological ring by $1$ in the corresponding norms.
However, while in $\RR$ the subring $\ZZ$ is discrete and has finite
covolume, the ring $\ZZ_{p}$ in $\QQ_{p}$ is compact and has infinite
covolume. Hence, with this point of view $\ZZ_{p}$ behaves more like
$\left[0,1\right]$ in $\RR$. This two opposite points of view e
algebraic and topological e are quite common when dealing with $p$-adic
numbers, namely that sometimes we think of $\ZZ_{p}$ as $\ZZ$ and
sometimes as $\left[0,1\right]$.
\end{rem}

We can now define the adele ring.
\begin{defn}
For a finite set $\infty\in S\subseteq\PP_{\infty}$, let 
\begin{align*}
\QQ_{S} & :=\prod_{\nu\in S}\QQ_{\nu}\\
\QQ^{\left(S\right)} & :=\prod_{\nu\in S}\QQ_{\nu}\times\prod_{p\notin S}\ZZ_{p},
\end{align*}
both with the product topology. We define the \emph{ring of adeles}
$\AA=\QQ_{\PP_{\infty}}$ to be the union $\bigcup_{S}\QQ^{\left(S\right)}$
where $S$ runs over all the finite subsets of $\PP_{\infty}$ containing
$\infty$. The topology on $\AA$ is the induced topology, namely
$U$ is open if $U\cap\QQ_{S}$ is open in $\QQ_{S}$ for any $S$
(or equivalently, it is generated by the open sets in $\QQ^{\left(S\right)}$).
This is called the \emph{restricted product} $\AA:=\RR\times\prod_{p}'\QQ_{p}$
with respect to $\ZZ_{p}$, namely sequence $\left(g^{\left(\infty\right)},g^{\left(2\right)},g^{\left(3\right)},...\right)\in\RR\times\prod_{p}\QQ_{p}$
where $g^{\left(p\right)}\in\ZZ_{p}$ for almost every $p$.

For each $S\subseteq\PP_{\infty}$ set 
\[
\ZZ\left[S^{-1}\right]:=\ZZ\left[\frac{1}{p}\;\mid\;p\in S\backslash\left\{ \infty\right\} \right]\leq\QQ,
\]
and embed it diagonally in $\QQ_{S}$ for $S$ finite and $S=\PP_{\infty}$. 
\end{defn}

\begin{lem}
For $\infty\in S\subseteq\PP_{\infty}$ the group $\ZZ\left[S^{-1}\right]$
is a cocompact lattice in $\QQ_{S}$.
\end{lem}

\begin{proof}
We leave it as an exercise to show that $\left(-\frac{1}{2},\frac{1}{2}\right)\times\prod_{p\in S\backslash\left\{ \infty\right\} }\ZZ_{p}$
which is an open set in $\QQ_{S}$ interests $\ZZ\left[S^{-1}\right]$
only in $\left\{ 0\right\} $, implying that it is discrete in $\QQ_{S}$.
Moreover, using the restricted product structure, and the Chinese
remainder theorem we get that 
\[
\ZZ\left[S^{-1}\right]+\left[-\frac{1}{2},\frac{1}{2}\right]\times\prod_{p\in S\backslash\left\{ \infty\right\} }\ZZ_{p}=\QQ_{S},
\]
so that $\ZZ\left[S^{-1}\right]$ is cocompact.
\end{proof}
Note in particular that for $S=\left\{ \infty\right\} $, we just
get the well known fact that $\ZZ$ is a lattice in $\RR$. As with
the $m$-adic numbers, there is a natural projection $\pi_{S}:\QQ\backslash\AA\to\ZZ\left[S^{-1}\right]\backslash\QQ_{S}$.
First identify $\QQ\backslash\AA$ with the fundamental domain $[0,1)\times\prod_{p}\ZZ_{p}$
of $\QQ$ in $\AA$ and then project to the coordinates in $S$. 

An important observations about these projections is that the preimage
of every point an orbit of $\prod_{p\notin S}\ZZ_{p}$ which is compact
by Tychonoff's theorem. It follows that the preimage of any compact
set is compact, i.e. these projections are proper.

Trying to understand measures over $\QQ\backslash\AA$, we first need
to understand compactly supported continuous functions on $\QQ\backslash\AA$.
Since $\pi_{S}$ is proper for any finite $S\subseteq\PP$ we have
the induced homomorphism 
\[
C_{c}\left(\ZZ\left[S^{-1}\right]\backslash\QQ_{S}\right)\to C_{c}\left(\QQ\backslash\AA\right).
\]
The functions in the image of this map are exactly those which are
invariant under the action of $\prod_{p\notin S}\ZZ_{p}$. A simple
application of the Stone Weierstrass theorem shows that the union
of these sets of functions, as we run over the finite $S$, is dense
in $C_{c}\left(\QQ\backslash\AA\right)$. This implies that if we
want to prove an equidistribution $\mu_{i}\wstar\mu_{Haar}$ on $\AA$,
it is enough to prove that $\mu_{i}\left(f\right)\to\mu_{Haar}\left(f\right)$
for functions in these images. Alternatively, we need to show the
pushforward of the measures $\mu_{i}$ to each one of the $\ZZ\left[S^{-1}\right]\backslash\QQ_{S}$
equidistributes. 

There is a similar structure when we work with groups over the adeles,
and in particular with the group $\GL_{2}^{1}$ which is the main
focus of these notes. There is however one main difference where $\GL_{2}\left(\ZZ\left[S^{-1}\right]\right)$
is a noncompact lattice in $\GL_{2}^{1}\left(\QQ_{S}\right)$ (since
$\SL_{2}\left(\ZZ\right)\backslash\SL_{2}\left(\RR\right)$ is noncompact). 

We already talked in \subsecref{The_p_adic_motivation} about our
translated orbit for the $S$ finite cases. However, note that this
is not the same problem, because we first do the translations in $\AA$
and then project down to $\QQ_{S}$. This is the point where a single
orbit can decompose to several diagonal orbits, and in particular,
as we saw, a single translated orbit by $u_{1/m}$ is split up to
$\varphi\left(m\right)$ orbits over $\RR$.

One of the main parts of our proof will be to show that if our measures
equidistribute when pushed to only the real place via $\pi_{\left\{ \infty\right\} }$,
then we can lift this equidistribution to any $S$. We already saw
that if we know that this is true for any finite $S$, than it is
true for $S=\PP_{\infty}$, so we are left with this lifting problem
for finite $S$. We will do this in \secref{Adelic_lifting} and we
leave the details to that section, but let us just mention the main
idea which is interesting in itself.

One of the main problem when working over the adeles, is that there
are infinitely many prime place that we can translate in. If the translation
was in only finitely many places, then we can use an already known
shearing result. To work around this problem we look only on the translation
in the real place, which we may assume to be either trivial, or $u_{x_{i}}$
with $x_{i}\to\infty$.

\textbf{\uline{Case 1}}: There is no translation in the real place.

In this case, all of our measures will be invariant under the same
diagonal matrices $A$ in the real place. For such measures, we can
use an invariant called the \emph{entropy }of the measure with respect
to $A$. This entropy measures in a sense how close is the measure
to being $\SL_{2}\left(\RR\right)$-invariant. In particular, the
entropy is bounded from above, and it achieves the maximum entropy
if and only if it is $\SL_{2}\left(\RR\right)$-invariant. 

The trick now is to use the fact that when projecting down the entropy
can only decrease. Hence, if the projected measures equidistribute,
their entropy will converge to the maximal entropy, so the entropy
of our original measure must converge to the maximal entropy also.
It follows that the limit will be $\SL_{2}\left(\RR\right)$-invariant
as well. This is of course a much smaller group that $\GL_{2}^{1}\left(\AA\right)$,
however, because we look on the quotient space $\GL_{2}\left(\QQ\right)\backslash\GL_{2}^{1}\left(\AA\right)$,
the group $\GL_{2}\left(\QQ\right)$ ``mixes'' the space together,
so invariance under the small group $\SL_{2}\left(\RR\right)$ will
automatically imply a larger invariance which is almost the whole
group.

\newpage{}

\textbf{\uline{Case 2}}: The translations in the real place go
to infinity.

In this case we can no longer use the entropy argument, because our
measure are not $A$-invariant. However, we already saw that unipotent
translations of $A$-orbits become more and more like unipotent orbits.
More specifically, if a measure $\mu$ is $A$ invariant, and we consider
translations $\mu_{i}=u_{x_{i}}\left(\mu\right)$ as $x_{i}\to\infty$,
then $\mu_{i}$ will be invariant under 
\[
u_{x_{i}}a\left(t\right)u_{-x_{i}}=a\left(t\right)u_{\left(e^{t}-1\right)x_{i}}.
\]
Fixing some constant $c$, we can choose $t_{i}$ such that $\left(e^{t_{i}}-1\right)x_{i}=c$
and note that since $x_{i}\to\infty$ we have that $t_{i}\to0$. It
follows that $\mu_{i}$ is $a\left(t_{i}\right)u_{c}$ invariant and
$a\left(t_{i}\right)\to a\left(0\right)=Id$, so any limit measure
will be invariant under $u_{c}$. As $c$ was arbitrary, the limit
measure will be invariant under the unipotent group $U$. 

This is very helpful, since we can now use Ratner's classification
theorem of unipotent invariant measures to conclude that our limit
measure is algebraic e it is supported on an orbit of some unimodular
subgroup $L\leq\GL_{2}^{1}\left(\AA\right)$ which contain $U$. At
this point we will show that if the projection to $\SL_{2}\left(\ZZ\right)\backslash\SL_{2}\left(\RR\right)$
is the Haar measure, then $L$ must contain all of $\SL_{2}\left(\RR\right)$.
We can now continue like in case 1 and conclude that our measure must
be $\GL_{2}^{1}\left(\AA\right)$-invariant.\\

These are all the main steps for the full equidistribution over the
adeles, namely first prove equidistribution only in the projection
to the real place, and then use either $A$-invariance and entropy
or $U$-invariance and Ratner's theorem to lift the equidistribution
to all of the adeles.

\newpage{}

\part{\label{part:Proofs_Details}Proofs and Details}

Now that we have seen all of the main ideas and steps leading from
our problem about continued fractions of rational to shearing over
the adeles, we turn to give the details behind these ideas. We begin
in \secref{Adelic_lifting} where we give the definitions for the
space of lattices over the adeles. This space has the standard space
of Euclidean lattices as its quotient, and in particular the Haar
measure on this space is pushed forward to the Haar measure on the
Euclidean lattices. In that subsection we will give some natural conditions
which implies that the converse holds as well, namely we can lift
the equidistribution in the real place to an equidistribution over
all the adeles.

Once we have these notations and the lifting result, we define in
\secref{Adelic_translations-1} the diagonal orbit through the origin,
the locally finite, diagonal invariant measure it supports and its
translations. Using the Iwasawa decomposition, we will see that for
equidistribution results for general translations, it is enough to
prove it for unipotent translation, or in other words, we need to
prove that the shearing process holds over the adeles. In particular
we will show that these orbit translation measures satisfy automatically
the extra conditions needed for the lifting result from \secref{Adelic_lifting}.
To simplify the notations, we restrict the discussion in this section
to dimension 2, though the most of results hold for a general dimension.

Finally, in \secref{real_case} we prove that in dimension 2, when
pushed to the real place our orbit translation measures equidistribute.
This equidistribution result together with the conditions that we
prove in \secref{Adelic_translations-1} allow us to use the lifting
result from \secref{Adelic_lifting} and to get the full equidistribution
result over the adeles. This result can be proved for some specific
translations in higher dimension (see for example \cite{david_equidistribution_nodate}),
however since we don't know if it holds in general dimension (and
even what is the right formulation), we stay only in dimension 2.

\section{\label{sec:Adelic_lifting}Adelic Lifting}

\global\long\def\norm#1{\left\Vert #1\right\Vert }%
\global\long\def\AA{\mathbb{A}}%
\global\long\def\QQ{\mathbb{Q}}%
\global\long\def\PP{\mathbb{P}}%
\global\long\def\CC{\mathbb{C}}%
\global\long\def\HH{\mathbb{H}}%
\global\long\def\ZZ{\mathbb{Z}}%
\global\long\def\NN{\mathbb{N}}%
\global\long\def\KK{\mathbb{K}}%
\global\long\def\RR{\mathbb{R}}%
\global\long\def\FF{\mathbb{F}}%
\global\long\def\oo{\mathcal{O}}%
\global\long\def\aa{\mathcal{A}}%
\global\long\def\bb{\mathcal{B}}%
\global\long\def\ff{\mathcal{F}}%
\global\long\def\mm{\mathcal{M}}%
\global\long\def\limfi#1#2{{\displaystyle \lim_{#1\to#2}}}%
\global\long\def\pp{\mathcal{P}}%
\global\long\def\qq{\mathcal{Q}}%
\global\long\def\da{\mathrm{da}}%
\global\long\def\dt{\mathrm{dt}}%
\global\long\def\dg{\mathrm{dg}}%
\global\long\def\ds{\mathrm{ds}}%
\global\long\def\dm{\mathrm{dm}}%
\global\long\def\dmu{\mathrm{d\mu}}%
\global\long\def\dx{\mathrm{dx}}%
\global\long\def\dy{\mathrm{dy}}%
\global\long\def\dz{\mathrm{dz}}%
\global\long\def\dnu{\mathrm{d\nu}}%
\global\long\def\flr#1{\left\lfloor #1\right\rfloor }%
\global\long\def\nuga{\nu_{\mathrm{Gauss}}}%
\global\long\def\diag#1{\mathrm{diag}\left(#1\right)}%
\global\long\def\bR{\mathbb{R}}%
\global\long\def\Ga{\Gamma}%
\global\long\def\PGL{\mathrm{PGL}}%
\global\long\def\GL{\mathrm{GL}}%
\global\long\def\PO{\mathrm{PO}}%
\global\long\def\SL{\mathrm{SL}}%
\global\long\def\PSL{\mathrm{PSL}}%
\global\long\def\SO{\mathrm{SO}}%
\global\long\def\mb#1{\mathrm{#1}}%
\global\long\def\wstar{\overset{w^{*}}{\longrightarrow}}%
\global\long\def\vphi{\varphi}%
\global\long\def\av#1{\left|#1\right|}%
\global\long\def\inv#1{\left(\mathbb{Z}/#1\mathbb{Z}\right)^{\times}}%
\global\long\def\cH{\mathcal{H}}%
\global\long\def\cM{\mathcal{M}}%
\global\long\def\bZ{\mathbb{Z}}%
\global\long\def\bA{\mathbb{A}}%
\global\long\def\bQ{\mathbb{Q}}%
\global\long\def\bP{\mathbb{P}}%
\global\long\def\eps{\epsilon}%
\global\long\def\on#1{\mathrm{#1}}%
\global\long\def\nuga{\nu_{\mathrm{Gauss}}}%
\global\long\def\set#1{\left\{  #1\right\}  }%
\global\long\def\smallmat#1{\begin{smallmatrix}#1\end{smallmatrix}}%
\global\long\def\len{\mathrm{len}}%
\global\long\def\idealeq{\trianglelefteqslant}%

The main goal of this section is to provide some natural conditions
on a measure on the space of adelic lattices, such that it will be
the Haar measure there if and only if its pushforward to the space
of standard Euclidean lattices is the Haar measure there. We start
by fixing our notations for working with the adeles.

\subsection{\label{subsec:Adelic_lattices}Adelic lattices - notations}

It is well known that the space of unimodular lattice in $\RR^{d}$
can be identified with $\SL_{d}\left(\ZZ\right)\backslash\SL_{d}\left(\RR\right)$.
The main goal of this subsection is to extend this presentation to
the adelic setting and to fix the notation for the rest of these notes.

Let $\PP$ be all the primes in $\NN$, and let $\PP_{\infty}=\PP\cup\left\{ \infty\right\} $
be the set of all places over $\QQ$. Unless stated otherwise, the
sets $S\subseteq\PP_{\infty}$ that we work with will always contain
$\infty$. For a subset $S\subseteq\PP_{\infty}$ (possibly infinite)
we let
\[
\QQ_{S}=\prod_{p\in S}'\QQ_{p},\quad\;\ZZ\left[S^{-1}\right]:=\ZZ\left[\frac{1}{p}:p\in S\backslash\left\{ \infty\right\} \right],
\]
where $\prod'$ is the restricted product with respect to $\ZZ_{p}\leq\QQ_{p}$.
We consider $\ZZ\left[S^{-1}\right]$ as embedded diagonally in $\QQ_{S}$
and it is well known that under this embedding $\ZZ\left[S^{-1}\right]$
is a lattice in $\QQ_{S}$. We shall usually write elements $x\in\QQ_{S}$
(and in other such products) as $x=\left(x^{\left(\infty\right)},x^{\left(p_{1}\right)},x^{\left(p_{2}\right)},...\right)$
where $p_{i}\in S$ are the primes. We denote by $x^{\left(f\right)}$
the element $x^{\left(f\right)}=\left(x^{\left(p_{1}\right)},x^{\left(p_{2}\right)},...\right)\in\prod_{p}'\QQ_{p}$,
and using the diagonal embedding, if $x\in\ZZ\left[S^{-1}\right]$,
then we also write $x^{\left(f\right)}=\left(x,x,x,...\right)\in\prod_{p}'\QQ_{p}$.
We will mainly be interested with $S=\PP_{\infty}$ and $S$ finite
(and in particular $S=\left\{ \infty\right\} $).

We similarly extend these notation to $\QQ_{S}^{d},\;\ZZ\left[S^{-1}\right]^{d}$
for any dimension $d\geq1$ and later on to groups over $\QQ_{S}$.
Since $\QQ_{S}$ is generally not a field (unless $S=\left\{ \infty\right\} $),
the space $\QQ_{S}^{d}$ is not a vector space, but it is a $\QQ_{S}$-module,
namely we can multiply by elements from $\QQ_{S}$. As in vector spaces,
modules over commutative rings always have a basis, and many of the
results for vector spaces hold here as well.

Note that for $S=\left\{ \infty\right\} $, the notation above is
just $\QQ_{S}=\RR,\;\ZZ\left[S^{-1}\right]=\ZZ$ which is the original
example of a lattice. In general we have two definitions for Euclidean
lattices in $\RR^{d}$ - the first is a discrete, finite covolume
subgroup of $\RR^{d}$ and the second is the $\ZZ$-span of a basis
of $\RR^{d}$. We now extend this notion to general $S$.
\begin{defn}
Fix some $d\geq1$, $S\subseteq\PP_{\infty}$ and let $L\leq\QQ_{S}^{d}$.
\begin{enumerate}
\item A \emph{$\ZZ\left[S^{-1}\right]$-module} in $\QQ_{S}^{d}$ is a subgroup
$L\leq\QQ_{S}^{d}$ closed under multiplication by $\ZZ\left[S^{-1}\right]$.
\item A \emph{lattice }in $\QQ_{S}^{d}$ is a $\ZZ\left[S^{-1}\right]$-module
which is discrete and cocompact.
\item We say that a lattice is unimodular, if it has covolume 1 (with the
standard Haar measure on $\QQ_{S}^{d}$).
\end{enumerate}
\end{defn}

As in the real case, one can show that $L\leq\QQ_{S}^{d}$ is a lattice,
if and only if it is spanned over $\ZZ\left[S^{-1}\right]$ by a $\QQ_{S}$
basis of $\QQ_{S}^{d}$. 
\begin{example}
$\ZZ\left[S^{-1}\right]$ is a unimodular lattice in $\QQ_{S}$. For
discreteness, since $\ZZ\left[S^{-1}\right]$ is a group it is enough
to show that $0$ is separated from all the other point, and indeed
it is the unique point in $\left(-1,1\right)\times{\displaystyle \prod_{p\in S\backslash\left\{ \infty\right\} }}\ZZ_{p}$.
In addition the compact set $\left[0,1\right]\times{\displaystyle \prod_{p\in S\backslash\left\{ \infty\right\} }}\ZZ_{p}$
is a fundamental domain, so that $\ZZ\left[S^{-1}\right]$ is a lattice.
Similarly $\ZZ\left[S^{-1}\right]^{d}$ is a unimodular lattice in
$\QQ_{S}^{d}$.
\end{example}

Next, we set $\GL_{d}\left(\QQ_{S}\right):=\prod_{\nu\in S}'\GL_{d}\left(\QQ_{\nu}\right)$
where the restricted product is with respect to $\GL_{d}\left(\ZZ_{p}\right)$.
The group $\GL_{d}\left(\QQ_{S}\right)$ acts transitively on the
space of $d$-dimensional lattices in $\QQ_{S}^{d}$, and the stabilizer
of $\ZZ\left[S^{-1}\right]^{d}$ is $\GL_{d}\left(\ZZ\left[S^{-1}\right]\right)$
(embedded diagonally). Thus, just like in the real case, we can identify
this space of lattices in $\QQ_{S}^{d}$ with $\GL_{d}\left(\ZZ\left[S^{-1}\right]\right)\backslash\GL_{d}\left(\QQ_{S}\right)$.

If we want to restrict our attention to unimodular lattice, we need
to know how an element in $\GL_{d}\left(\QQ_{S}\right)$ changes the
measure in $\QQ_{S}^{d}$. As in the real case, this change can be
measured by the determinant of the matrix.
\begin{defn}
Fix some $d\geq1$ and $\infty\in S\subseteq\PP_{\infty}$.
\begin{enumerate}
\item For $x=\left(x^{\left(\nu\right)}\right)_{\nu\in S}\in\QQ_{S}$, we
define $\left|x\right|=\left|x\right|_{S}:=\prod_{\nu\in S}\left|x^{\left(\nu\right)}\right|_{\nu}\in\RR_{\geq0}$,
where $\left|\cdot\right|_{\nu}$ is the standard norm on $\QQ_{\nu}$.
\item We define $\det=\det_{S}:\GL_{d}\left(\QQ_{S}\right)\to\QQ_{S}$ by
applying determinant in each place. We further write $\left|\det\right|$
to be the composition $\left|\det\right|:\GL_{d}\left(\QQ_{S}\right)\overset{\det}{\to}\QQ_{S}\overset{\left|\cdot\right|_{S}}{\to}\RR$.\\
\end{enumerate}
\end{defn}

Note that by the definition of restricted product, if $x=\left(x^{\left(\nu\right)}\right)_{\nu\in S}\in\QQ_{S}$,
then $x^{\left(p\right)}\in\ZZ_{p}$ for almost every prime $p\in S$
and therefore $\left|x^{\left(p\right)}\right|_{p}\leq1$. It follows
that $\left|x\right|=\prod_{\nu\in S}\left|x^{\left(\nu\right)}\right|_{\nu}$
is well defined, though it can be zero even if $x$ doesn't have any
zero entries. However, if all the entries are nonzero and $x^{\left(p\right)}\in\ZZ_{p}^{\times}$
for almost every $p$, or equivalently $\left|x^{\left(p\right)}\right|_{p}=1$,
then we get that $\left|x\right|>0$. In particular we see that for
$g\in\GL_{d}\left(\QQ_{S}\right)$ we have that $\left|\det\right|\left(g\right)>0$.
Furthermore, for $x\in\ZZ\left[S^{-1}\right]^{\times}$, by the product
formula we have that $\left|x\right|_{S}=1$, implying that $\left|\det\right|\left(g\right)=1$
for $g\in\GL_{d}\left(\ZZ\left[S^{-1}\right]\right)$.

It can now be shown that for $\Omega\subseteq\QQ_{S}^{d}$ and $g\in\GL_{d}\left(\QQ_{S}\right)$,
the measure of $g\left(\Omega\right)$ is the measure of $\Omega$
times $\left|\det\right|\left(g\right)$. With this in mind we define
\begin{align*}
G_{S} & =\GL_{d}^{1}\left(\QQ_{S}\right):=\left\{ g\in\GL_{n}\left(\QQ_{S}\right)\;\mid\;\left|\det\right|\left(g\right)=1\right\} ,\\
\Gamma_{S} & =\GL_{d}\left(\ZZ\left[S^{-1}\right]\right),\\
X_{S} & =\Gamma_{S}\backslash G_{S},
\end{align*}
so that $X_{S}$ can be identified with the space of unimodular lattices
in $\QQ_{S}^{d}$. For $S=\left\{ \infty\right\} $ and $S=\PP_{\infty}$,
we will also use $G_{\RR}:=G_{\left\{ \infty\right\} }$, $G_{\AA}=G_{\PP_{\infty}}$
and similarly for $\Gamma_{S}$ and $X_{S}$. 

The space $X_{S}$ is locally compact, second countable Hausdorff
spaces and has the natural $G_{S}$-action from the right. Moreover,
the group $G_{S}$ is unimodular  and $\Gamma_{S}\leq G_{S}$ is
a lattice, so $X_{S}$ supports a $G_{S}$-invariant probability measure
which we denote by $\mu_{Haar,S}$.

Finally, as a sanity check, if $S=\left\{ \infty\right\} $, then
$X_{S}$ is simply $\GL_{d}\left(\ZZ\right)\backslash\GL_{d}^{1}\left(\RR\right)$.
Both of the groups $\GL_{d}\left(\ZZ\right),\GL_{d}^{1}\left(\RR\right)$
have the index two subgroup $\SL_{d}\left(\ZZ\right)$ and $\SL_{d}\left(\RR\right)$
respectively, so that $X_{S}\cong\SL_{d}\left(\ZZ\right)\backslash\SL_{d}\left(\RR\right)$
is the standard space of $d$-dimensional unimodular lattices in $\RR^{d}$.
\begin{rem}
The groups $\GL_{d}^{1}$ and $\PGL_{d}$ are not that far off from
each other, and one can actually prove all of the results here for
$\PGL_{d}$ instead. However, we choose to work with $\GL_{d}^{1}$
since it simplifies many of the notation, and in particular we work
with matrices and not equivalence classes modulo the center. This
allows us, for example, to have the generalization of Mahler criterion
that we prove in \secref{Mahler_criterion}.
\end{rem}

The next step is to connect between the spaces $X_{S}$ for different
$S\subseteq\PP_{\infty}$. For any $\tilde{S}\subseteq S$ the standard
projection $\GL_{d}\left(\QQ_{S}\right)\to\GL_{d}\left(\QQ_{\tilde{S}}\right)$
doesn't induce a well defined projection for the quotient spaces $X_{S}\to X_{\tilde{S}}$.
However, there is such a natural projection $\pi_{\tilde{S}}^{S}:X_{S}\to X_{\tilde{S}}$
which is defined as follows. Consider first the natural open embedding:
\[
H_{S}:=\SL_{d}\left(\RR\right)\times\prod_{p\in S}\GL_{d}\left(\ZZ_{p}\right)\hookrightarrow G_{S}.
\]
Note that while the elements $g\in G_{S}$ are such that the product
of $\left|\det\left(g^{\left(\nu\right)}\right)\right|_{\nu}$ is
1, the elements $h\in H_{S}$ satisfy $\left|\det\left(h^{\left(\nu\right)}\right)\right|_{\nu}=1$
for all $\nu$. 
\begin{claim}
The map $H_{S}\hookrightarrow G_{S}$ induces a homeomorphism $\SL_{d}\left(\ZZ\right)\backslash H_{S}\cong\Gamma_{S}\backslash G_{S}$.
\end{claim}

\begin{proof}
We claim that the $H_{S}$ acts transitively on $X_{S}$ and since
$H_{S}\cap\Gamma_{S}=\SL_{d}\left(\ZZ\right)$, the claim will follow.
Let $g\in G_{S}$, and let $q\in\Gamma_{S}$ be the identity matrix
with 
\[
sign\left(g^{\left(\infty\right)}\right)\prod_{p\in S\backslash\left\{ \infty\right\} }\left|\det\left(g^{\left(p\right)}\right)\right|_{p}
\]
 in the $\left(1,1\right)$-coordinate. It then follows that for every
prime $p\in S\backslash\left\{ \infty\right\} $ we have that 
\[
\left|\det\left(qg^{\left(p\right)}\right)\right|_{p}=\left|\det\left(q\right)\right|_{p}\left|\det\left(g^{\left(p\right)}\right)\right|_{p}=\left|\det\left(g^{\left(p\right)}\right)\right|_{p}^{-1}\left|\det\left(g^{\left(p\right)}\right)\right|_{p}=1.
\]
Moreover, since $\left|\det_{S}\left(qg\right)\right|=1$ and $\det\left(qg^{\left(\infty\right)}\right)>0$,
we also get that $\det\left(qg^{\left(\infty\right)}\right)=1$. In
other words we have shown that $qg\in H_{S}$ which proves the transitivity.
\end{proof}
\newpage{}

In this new presentation the lattice $\SL_{d}\left(\ZZ\right)$ is
fixed, so that given $\infty\in\tilde{S}\subseteq S\subseteq\PP_{\infty}$,
the standard projection $H_{S}\to H_{\tilde{S}}$ induces the projection
\[
\pi_{\tilde{S}}^{S}:X_{S}\cong\SL_{d}\left(\ZZ\right)\backslash H_{S}\to\SL_{d}\left(\ZZ\right)\backslash H_{\tilde{S}}\cong X_{\tilde{S}}.
\]
The preimage of every point is then an orbit of $\prod_{p\in S\backslash\tilde{S}}\GL_{d}\left(\ZZ_{p}\right)$
which is compact, implying that $\pi_{\tilde{S}}^{S}$ is proper.

In general, the presentation with $H_{S}$ is much more convenient
to work with, because it let us connect between the different $X_{S}$.
On the other hand, we want to act with the larger group $G_{S}$,
so throughout these notes we will need to move back and forth between
these two presentations.

Just like the space of Euclidean lattices, we have a generalized Mahler
criterion for the space of $S$-adic lattices. We will prove this
criterion in \secref{Mahler_criterion}.
\begin{defn}
For $h\in H_{S}$ define 
\[
ht_{S}\left(\ZZ\left[S^{-1}\right]^{d}h\right)=ht_{\infty}\left(\ZZ^{d}h^{\left(\infty\right)}\right):=\left(\min_{0\neq v\in\ZZ^{d}}\norm{vh^{\left(\infty\right)}}\right)^{-1}.
\]
\end{defn}

\begin{lem}[Generalized Mahler's criterion]
\label{lem:generalized_mahler} A set $\Omega\subseteq X_{S}$ is
bounded if and only if $ht_{S}\left(\Omega\right)$ is bounded.
\end{lem}

We can identify $G_{\tilde{S}}$ as a subgroup of $G_{S}$ as the
elements which are the identity in all of the entries in $S\backslash\tilde{S}$.
It is now not hard to check that $\pi_{\tilde{S}}^{S}$ is $G_{\tilde{S}}$-equivariant.
In particular a $G_{S}$-invariant probability measure on $X_{S}$
will be pushed down to a $G_{\tilde{S}}$. 

For the converse direction, suppose now that $\mu_{S}$ is a probability
measure on $X_{S}$ such that it pushforward $\pi_{\RR}^{S}\left(\mu_{S}\right)$
to the ``smallest'' possible space $X_{\RR}$ is the $\SL_{d}\left(\RR\right)$-invariant
measure. Trying to lift this invariance back to $\mu_{S}$ we encounter
two problems: 
\begin{enumerate}
\item Show that $\mu_{S}$ itself is $\SL_{d}\left(\RR\right)$-invariant.
\item Show that $\mu_{S}$ is invariant under $G_{S\backslash\left\{ \infty\right\} }$
as well.
\end{enumerate}
These two conditions will require us to show some invariance condition
of $\mu_{S}$. In order to get (1) we will need an extra invariance
condition that $\mu_{S}$ is invariant under the diagonal or unipotent
flow in $\SL_{d}\left(\RR\right)$ (which is done in \subsecref{Lifting_R_invariance}).
Once we have this $\SL_{d}\left(\RR\right)$-invariance, we automatically
get in \subsecref{automatice_invariance} invariance under a larger
group - this is because $\mu_{S}$ is a measure on $\Gamma_{S}\backslash G_{S}$
and $\Gamma_{S}$ ``mixes'' the real coordinate with the coordinates
in $S\backslash\left\{ \infty\right\} $. However, this will not provide
a full $G_{S}$-invariance, but if $\left|S\right|<\infty$, and we
have some extra uniformity condition over the primes in $S\backslash\left\{ \infty\right\} $
, then we will get this invariance. Finally, For $\left|S\right|$
infinite, by the structure of restricted products, it will suffice
to prove $G_{S_{0}}$-invariance for every $\infty\in S_{0}\subseteq S$
with $\left|S_{0}\right|$ finite. This final part will be done in
\subsecref{to_full_invariance}, where we will also prove the main
lifting result in \thmref{main_lifting}.

\newpage{}

\subsection{\label{subsec:Lifting_R_invariance}Lifting the $\protect\SL_{d}\left(\protect\RR\right)$-invariance}

Let us fix $S\subseteq\PP_{\infty}$ finite and let $\mu_{S}$ be
a probability measure on $X_{S}$ such that $\pi_{\RR}^{S}\left(\mu_{S}\right)=\mu_{Haar,\RR}$.
We begin with the proof of lifting the $\SL_{d}\left(\RR\right)$-invariance
of $\pi_{\RR}^{S}\left(\mu_{S}\right)$ to the $\SL_{d}\left(\RR\right)$-invariance
of $\mu_{S}$ itself. The idea is to use either diagonal or unipotent
invariance, and the main tools to study these are maximal entropy
for the diagonal case and Ratner's classification theorem for the
unipotent case. However, both of these theorems are usually formulated
for spaces of the form $\mathbb{G}\left(\ZZ\left[S^{-1}\right]\right)\backslash\mathbb{G}\left(\QQ_{S}\right)$
for some finite $S\subseteq\PP_{\infty}$, and our space $X_{S}=\GL_{d}\left(\ZZ\left[S^{-1}\right]\right)\backslash\GL_{d}^{1}\left(\QQ_{S}\right)$
is not exactly like this. So instead, we will first prove the claim
for measures over 
\[
Y_{S}=\SL_{d}\left(\ZZ\left[S^{-1}\right]\right)\backslash\SL_{d}\left(\QQ_{S}\right),
\]
which can be viewed as subspaces of $X_{S}$ via the embedding $\SL_{d}\left(\QQ_{S}\right)\hookrightarrow\GL_{d}^{1}\left(\QQ_{S}\right)$,
and in the end we will show how to extend it to $X_{S}$.

Let us first recall the required results, starting with maximal entropy. 

We give here the basic definitions for entropy, though we will not
really use them, and only use the result about the maximal entropy.
For more details about entropy in homogeneous spaces, see \cite{einsiedler_diagonal_nodate,einsiedler_entropy_nodate}.
\begin{defn}
Let $\left(X,\mu,T\right)$ be a measure-preserving system. For a
finite measurable partition $\pp$ of $X$ and $n\in\NN$ we write
$\pp_{n}=\bigvee_{0}^{n-1}T^{-i}\pp$ where $\vee$ is the joint refinement
operation.
\begin{enumerate}
\item For a finite measurable partition $\pp$ of $X$ we write $H_{\mu}\left(\pp\right)=-{\displaystyle \sum_{P\in\pp}}\mu\left(P\right)\ln\left(P\right)$
where \\
$0\ln\left(0\right)=0$ and set $H_{\mu}\left(T,\pp\right)={\displaystyle \lim_{n\to\infty}}\frac{1}{n}H_{\mu}\left(\pp^{n}\right)$
(and this limit always exists).
\item The entropy of $\mu$ with respect to $T$ is defined to be $h_{\mu}\left(T\right)=\sup_{\pp}H_{\mu}\left(\pp,T\right)$
where the supremum is over finite measurable partitions of $X$.
\end{enumerate}
\end{defn}

On each of the lattice spaces $X_{S}$ we have the action of $\SL_{d}\left(\RR\right)$
and in particular of its positive diagonal subgroup $A$. Recall that
we can identify this subgroup with $\RR_{0}^{d}:=\left\{ \left(t_{1},...,t_{d}\right)\;\mid\;\sum t_{i}=0\right\} $
via $\bar{t}\mapsto a\left(\bar{t}\right):=diag\left(e^{t_{1}},...,e^{t_{d}}\right)$.
When the action $T$ is a multiplication by some certain elements
from $A$, the maximal possible entropy can be achieved only with
the $\SL_{d}\left(\RR\right)$-invariant measures. Let us make this
statement more precise.
\begin{defn}
For the spaces $X_{S},S\subseteq\PP$ finite and $\bar{t}\in\RR_{0}^{d}$,
we shall denote by $T_{\bar{t}}:X_{S}\to X_{S}$ the right multiplication
$T_{a}\left(x\right)=xa$ where $a=a\left(\bar{t}\right):=\diag{e^{t_{1}},...,e^{t_{d}}}\in\SL_{d}\left(\RR\right)$.
The stable horosphere subgroup of $a$ is defined to be
\begin{align*}
U_{a} & :=\left\{ g\in\SL_{d}\left(\RR\right)\;\mid\;a^{n}ga^{-n}\to e\;as\;n\to\infty\right\} \\
 & =\left\{ I+\sum_{t_{i}<t_{j}}\alpha_{i,j}e_{i,j}\in\SL_{d}\left(\RR\right)\;\mid\;\alpha_{i,j}\in\RR\right\} .
\end{align*}
We will further use the notation $U_{i,j}=\left\{ I+ue_{i,j}\in\SL_{d}\left(\RR\right)\;\mid\;u\in\RR\right\} $
for $i\neq j$ and note that $U_{a}=\left\langle U_{i,j}\;\mid\;t_{i}<t_{j}\right\rangle $.
\end{defn}

In particular if $t_{1}\leq t_{2}\leq\cdots\leq t_{d}$, then $U_{a}$
is a subgroup of the unipotent upper triangular matrices, and it equals
this group if all of the $t_{i}$ are distinct. The matrix $a$ acts
by conjugation on the Lie algebra $\mathfrak{U}_{a}=span_{\RR}\left\{ e_{i,j}\;\mid\;t_{i}<t_{j}\right\} $
of $U_{a}$, where each $e_{i,j}$ is an eigenvector with eigenvalue
$e^{t_{i}-t_{j}}$. An important constant that we will use is $\Psi_{a}:=-\ln\left|\det\left(Ad_{a}\mid_{\mathfrak{U}_{a}}\right)\right|$
which measures how much conjugation by $a$ ``stretches'' $U_{a}$.
\begin{example}
\begin{enumerate}
\item For the matrix $a=\left(\begin{array}{cc}
e^{-t/2} & 0\\
0 & e^{t/2}
\end{array}\right)$, the stable horospherical subgroup is $U_{a}=\left\{ \left(\begin{array}{cc}
1 & u\\
0 & 1
\end{array}\right)\;\mid\;u\in\RR\right\} $ and the Lie algebra is $\mathfrak{U}_{a}=\left\{ ue_{1,2}=\left(\begin{array}{cc}
0 & u\\
0 & 0
\end{array}\right)\;\mid\;u\in\RR\right\} $ with a single eigenvalue $\frac{1}{e}$. Hence $\Psi_{a}=-\ln\left|\det\left(Ad_{a}\mid_{\mathfrak{U}_{a}}\right)\right|=-\ln\left(\frac{1}{e}\right)=1$.
\item In higher dimension, for $a=\diag{e^{-\frac{(d-1)}{d}},e^{\frac{1}{d}},...,e^{\frac{1}{d}}}$
we have $\mathfrak{U}_{a}=\left\{ \sum_{2}^{d}u_{i}e_{1,i}\;\mid\;u_{i}\in\RR\right\} $
and the eigenvalue $\frac{1}{e}$ has multiplicity $d-1$. Hence $\Psi_{a}=-\ln\left|\det\left(Ad_{a}\mid_{\mathfrak{U}_{a}}\right)\right|=d-1$.
\item If $t_{1}\leq t_{2}\leq\cdots\leq t_{d}$, then $\Psi_{a}=-\ln\left|\det\left(Ad_{a}\mid_{\mathfrak{U}_{a}}\right)\right|=\sum_{i<j}\left(t_{j}-t_{i}\right)$.
\item For any $a\in A$ we have that $\Psi_{a}=\Psi_{a^{-1}}$.\\
\end{enumerate}
\end{example}

We can now formulate the maximal entropy result.
\begin{thm}[see Theorems 7.6 and 7.9 in \cite{einsiedler_diagonal_nodate}]
\label{thm:max_entropy} Fix some finite set $S\subseteq\PP_{\infty}$
, $a=a\left(\bar{t}\right)\in A$ for some $\bar{t}\in\RR_{0}^{d}$
and let $\mu_{S}$ be a $T_{a}$-invariant probability measure on
$Y_{S}$. Then $h_{\mu_{S}}\left(T_{a}\right)\leq\Psi_{a}$ with equality
if and only if $\mu_{S}$ is $U_{a}$-invariant. Similarly $h_{\mu_{S}}\left(T_{a}^{-1}\right)\leq\Psi_{a}$
with equality if and only if $\mu_{S}$ is $U_{a^{-1}}=U_{a}^{tr}$-invariant. 
\end{thm}

In case that $T$ is invertible, like with $T_{a}$ above, we have
that $h_{\mu}\left(T\right)=h_{\mu}\left(T^{-1}\right)$. Thus, an
immediate corollary of the theorem above is that if $\mu_{S}$ is
a $T_{a}$-invariant probability measure on $Y_{S}$ which has maximal
entropy $\Psi_{a}$ with respect to $T_{a}$, then it is $\left\langle U_{a},U_{a}^{tr}\right\rangle =\SL_{d}\left(\RR\right)$
invariant. \\

The second result we need deals with unipotent-invariant measures,
in which case we use Ratner's theorem.
\begin{thm}
\label{thm:Ratner}(See \cite{tomanov_orbits_2000}) Fix some finite
$S\subseteq\PP_{\infty}$ finite and let $\mu_{S}$ be an ergodic
$U$-invariant probability on $Y_{S}$ for some unipotent subgroup
$U$ of $\SL_{d}\left(\QQ_{S}\right)$. Then there exists a subgroup
$H\leq L\leq\SL_{d}\left(\QQ_{S}\right)$, such that $\mu_{S}$ is
an $L$-invariant probability measure on a closed $L$-orbit in $Y_{S}$.
\end{thm}

For such algebraic measures that we get from Ratner's theorem, we
have the following lifting result.
\begin{lem}
\label{lem:unipotent_lifting}Let $S\subseteq\PP_{\infty}$ be a finite
set and write $\pi=\pi_{\RR}^{S}$. Let $\mu_{S}$ be a probability
measure on $Y_{S}$ such that:
\begin{enumerate}
\item $\mu_{S}$ is an $L$-invariant probability measure on $xL$ where
$x\in Y_{S}$ and $L\leq\SL_{d}\left(\QQ_{S}\right)$.
\item $L$ contains at least one element $g\in\SL_{d}\left(\RR\right)$
which is not $\pm Id$.
\item $supp\left(\pi\left(\mu_{S}\right)\right)=Y_{\RR}$.
\end{enumerate}
Then $\mu_{S}$ is $\SL_{d}\left(\RR\right)$-invariant.

\end{lem}

\begin{proof}
Define $L^{\infty}$ and $L_{\infty}$ be the intersection and projection
of $L$ to the real place, namely
\begin{align*}
L^{\infty} & =\left\{ g\in\SL_{d}\left(\RR\right)\;\mid\;\left(g,Id,...,Id\right)\in L\right\} ,\\
L_{\infty} & =\left\{ g\in\GL_{d}\left(\RR\right)\;\mid\;\exists g^{\left(f\right)}\in\prod_{p\in S\backslash\left\{ \infty\right\} }\SL_{d}\left(\QQ_{p}\right),\;s.t.\;\left(g,g^{\left(p\right)}\right)\in L\right\} .
\end{align*}
Since $L$ is closed as the stabilizer of $\mu_{S}$, and $\SL_{d}\left(\RR\right)$
is closed in $\SL_{d}\left(\QQ_{S}\right)$, we see that $L^{\infty}=L\cap\SL_{d}\left(\RR\right)$
is closed and it is also easy to see that it is normal in $L_{\infty}$.
We shall soon see that condition (3) above implies that $L_{\infty}\cap\SL_{d}\left(\RR\right)$
is dense in $\SL_{d}\left(\RR\right)$, so that $L^{\infty}\leq\SL_{d}\left(\RR\right)$
is actually normal in $\SL_{d}\left(\RR\right)$. Using part (2) and
the simplicity of $\PSL_{d}\left(\RR\right)$, we conclude that $L^{\infty}$
must be all of $\SL_{d}\left(\RR\right)$, which is what we wanted
to prove. 

Thus, we are left to show that the condition $supp\left(\pi\left(\mu_{S}\right)\right)=Y_{\RR}$
implies that $L_{\infty}$ is dense in $\SL_{d}\left(\RR\right)$.

First, it is easy to check that the pushforward satisfies $\pi\left(supp\left(\mu_{S}\right)\right)\subseteq supp\left(\pi\left(\mu_{S}\right)\right)$,
but the converse is true as well. Indeed, fix some $y\in supp\left(\pi\left(\mu_{S}\right)\right)$
and an open neighborhood $y\in V_{0}$ with compact closure. For any
open subset $y\in V\subseteq V_{0}$ we have that $\mu_{S}\left(\pi^{-1}\left(V\right)\right)=\left(\pi\mu_{S}\right)\left(V\right)>0$,
so we can find $x_{V}\in supp\left(\mu_{S}\right)\cap\pi^{-1}\left(V\right)\subseteq\pi^{-1}\left(\overline{V}_{0}\right)$.
Since the last set is compact (using the fact that the map is proper),
we conclude that the net $V\mapsto x_{V}$ has a convergent subnet
to some $x_{\infty}\in\pi^{-1}\left(\overline{V}_{0}\right)$. Furthermore,
since $V\mapsto\pi\left(x_{V}\right)\in V$ converges to $y$, we
obtain that $x_{\infty}\in\pi^{-1}\left(y\right)$. Finally, since
$supp\left(\mu_{S}\right)$ is closed it follows that $x_{\infty}\in supp\left(\mu_{S}\right)$
so that $y\in\pi\left(supp\left(\mu_{S}\right)\right)$. 

By the assumption that $supp\left(\pi\left(\mu_{S}\right)\right)=Y_{\RR}$,
and since $supp\left(\mu_{S}\right)=xL$, we get that $\pi\left(xL\right)=Y_{\RR}$,
so we may choose $x=\SL_{d}\left(\ZZ\left[S^{-1}\right]\right)\cdot h$
for some $h=\left(Id,h^{\left(f\right)}\right)$ where $h^{\left(f\right)}\in{\displaystyle \prod_{p\in S\backslash\left\{ \infty\right\} }}\SL_{d}\left(\ZZ_{p}\right)$. 

Letting $\tilde{\Gamma}_{S}=\SL_{d}\left(\ZZ\left[S^{-1}\right]\right)$,
for any $g=\left(g^{\left(\infty\right)},g^{\left(f\right)}\right)\in L$
we can write 
\[
xg=\tilde{\Gamma}_{S}\left(g^{\left(\infty\right)},h^{\left(f\right)}g^{\left(f\right)}\right)=\tilde{\Gamma}_{S}\left(\gamma g^{\left(\infty\right)},\gamma h^{\left(f\right)}g^{\left(f\right)}\right),
\]
where $\gamma\in\tilde{\Gamma}_{S}$ and $\gamma h^{\left(f\right)}g^{\left(f\right)}\in{\displaystyle \prod_{p\in S\backslash\left\{ \infty\right\} }}\SL_{d}\left(\ZZ_{p}\right)$,
implying that $\pi\left(xg\right)=\tilde{\Gamma}_{\RR}\gamma g^{\left(\infty\right)}$.
We conclude that 
\begin{align*}
Y_{\RR} & =\pi\left(xL\right)\subseteq\tilde{\Gamma}_{\RR}\cdot\left(\tilde{\Gamma}_{S}L_{\infty}\right)\subseteq\tilde{\Gamma}_{S}L_{\infty},
\end{align*}
and therefore $\SL_{d}\left(\RR\right)\subseteq\tilde{\Gamma}_{S}L_{\infty}$.
Let us show that the fact that $\tilde{\Gamma}_{S}$ is countable
implies that $L_{\infty}\cap\SL_{d}\left(\RR\right)$ is dense in
$\SL_{d}\left(\RR\right)$. 

Fix some $M\in\mathfrak{sl}_{d}\left(\RR\right)$ and let $L_{M}=\left\{ t\in\RR\;\mid\;\exp\left(tM\right)\in L_{\infty}\right\} $
which is a subgroup of $\RR$. If we can show that $L_{M}$ is dense
in $\RR$, then in particular $\exp\left(M\right)\in\overline{L_{\infty}}$.
If we can show this for any $M$, then we will get that $L_{\infty}\cap\SL_{d}\left(\RR\right)$
is dense in $\SL_{d}\left(\RR\right)$. 

Fix some $\varepsilon>0$, and for every $0<t<\varepsilon$ write
$\exp\left(tM\right)=\gamma_{t}g_{t}$ where $\gamma_{t}\in\tilde{\Gamma}_{S}$
and $g_{t}\in L_{\infty}$. Since there are uncountable such $t$
and $\tilde{\Gamma}_{S}$ is countable, there are $0<t_{1}<t_{2}<\varepsilon$
such that $\gamma_{t_{1}}=\gamma_{t_{2}}$. It follows that $\exp\left(\left(t_{2}-t_{1}\right)M\right)=g_{t_{1}}^{-1}g_{t_{2}}\in L_{\infty}$,
so that $t_{2}-t_{1}\in L_{M}\cap\left(0,\varepsilon\right)$. Since
$\varepsilon$ was arbitrary, we conclude that $L_{M}$ must be dense
in $\RR$ and therefore $\overline{L_{\infty}\cap\SL_{d}\left(\RR\right)}=\SL_{d}\left(\RR\right)$
which was the last result that we needed to complete the proof.

\end{proof}
We can now put everything together to get our $\SL_{d}\left(\RR\right)$-invariance
lifting on $Y_{S}$.
\begin{lem}
\label{lem:increase_partial_invariance}Let $S\subseteq\PP_{\infty}$
be a finite set and $\mu_{S}$ a probability measure on $Y_{S}$ such
that $\mu_{Haar,\RR}=\pi_{\RR}^{S}\left(\mu_{S}\right)$. Then if
$\mu_{S}$ is invariant under a one parameter unipotent subgroup $\left\{ u^{t}\;\mid\;t\in\RR\right\} \in\SL_{d}\left(\RR\right)$
or a diagonal element $Id\neq a\in\SL_{d}\left(\RR\right)$, then
it is $\SL_{d}\left(\RR\right)$-invariant. 
\end{lem}

\begin{proof}
We claim that we may assume that $\mu_{S}$ is $a$ (resp. $u$) ergodic.
Indeed, if $\mu_{S}=\int\mu_{S,\alpha}\mathrm{d\alpha}$ is the ergodic
decomposition to $a$ (resp. $u$) ergodic measures, then $\mu_{Haar,\RR}=\pi_{\RR}^{S}\left(\mu_{S}\right)=\int\pi_{\RR}^{S}\left(\mu_{S,\alpha}\right)\mathrm{d\alpha}$
is also a decomposition. Since $\mu_{Haar,\RR}$ is both $u$ and
$A$-ergodic, then it is an extreme point in the space of invariant
probability measures, and therefore this decomposition is trivial
- outside of a zero measure set, we have that $\pi_{\RR}^{S}\left(\mu_{S,\alpha}\right)=\mu_{Haar,\RR}$.
Thus, it is enough to prove the lemma for the $A$ (resp. $u$)-invariant
and ergodic measures $\mu_{S,\alpha}$.

Assume first that $\mu_{S}$ is $A$-invariant. As the entropy can
only decrease in factors, using \thmref{max_entropy} we obtain that
\[
\Psi_{a}=h_{\mu_{Haar,\RR}}\left(T_{a}\right)\leq h_{\mu_{S}}\left(T_{a}\right)\leq\Psi_{a},
\]
hence $h_{\mu_{S}}\left(T_{a}\right)=\Psi_{a}$ and similarly $h_{\mu_{S}}\left(T_{a}^{-1}\right)=\Psi_{a}$.
Using \thmref{max_entropy} once again we conclude that $T$ is $\left\langle U_{a},U_{a}^{tr}\right\rangle =\SL_{d}\left(\RR\right)$-invariant.\\

If $\mu_{S}$ is $u$-invariant and ergodic under some unipotent matrix
in $\SL_{d}\left(\RR\right)$, then we can apply Ratner's theorem
which provides condition (1) in \lemref{unipotent_lifting}. Since
$u$ is not central, we get condition (2) of that lemma. Finally,
we try to lift the Haar measure $\pi_{\RR}^{S}\left(\mu_{S}\right)=\mu_{Haar,\RR}$,
so that condition (3) is satisfied as well. Hence by this lemma we
get that $\mu_{S}$ is $\SL_{d}\left(\RR\right)$-invariant.
\end{proof}
Finally, we want to move from $Y_{S}$ to $X_{S}$. The difference
between these two spaces is that in $Y_{S}$ we require all the elements
to be of determinant 1, while in $X_{S}$ the product of the norms
of the determinant is $1$. To help us move from one space to the
other we use the following definitions.
\begin{defn}
\label{def:determinant}For $S\subseteq\PP_{\infty}$, we define the
determinant
\[
\det_{S\backslash\left\{ \infty\right\} }:G_{S}\overset{\det}{\longrightarrow}\RR^{\times}\times{\displaystyle \prod_{p\in S\backslash\left\{ \infty\right\} }}\QQ_{p}^{\times}\longrightarrow{\displaystyle \prod_{p\in S\backslash\left\{ \infty\right\} }}\QQ_{p}^{\times}.
\]

We identify the elements from ${\displaystyle \prod_{p\in S\backslash\left\{ \infty\right\} }}\QQ_{p}^{\times}$
asdiagonal matrices in $\GL_{d}\left(\QQ_{S}\right)$ via
\[
\bar{\alpha}=\left(\alpha^{\left(p_{1}\right)},...,\alpha^{\left(p_{k}\right)}\right)\mapsto g_{\bar{\alpha}}=\left(Id,diag\left(\alpha^{\left(p_{1}\right)},1,...,1\right),...,diag\left(\alpha^{\left(p_{k}\right)},1,...,1\right)\right).
\]
\end{defn}

\begin{thm}
\label{thm:increase_invariance}\lemref{increase_partial_invariance}
holds for the space $X_{S}$ as well.
\end{thm}

\begin{proof}
Let $K^{1}=\left\{ \bar{\alpha}\in{\displaystyle \prod_{p\in S\backslash\left\{ \infty\right\} }}\QQ_{p}^{\times}\;\mid\;\prod\left|\alpha^{\left(p_{i}\right)}\right|_{p_{i}}=1\right\} $.
Viewing $Y_{S}$ as a subspace of $X_{S}$ via embedding $\SL_{d}\left(\QQ_{S}\right)\hookrightarrow\GL_{d}^{1}\left(\QQ_{S}\right)$
we get decompose $X_{S}$ as 

\[
X_{S}=\bigsqcup_{\bar{\alpha}\in K^{1}}Y_{S}g_{\bar{\alpha}}.
\]
This defines the map $X_{S}\to K^{1}$ which sends elements in $Y_{S}g_{\bar{\alpha}}$
to $\bar{\alpha}$. Given a probability measure on $\mu$ on $X_{S}$,
we can use disintegration of measures to obtain 
\[
\mu=\int_{K^{1}}\mu_{\bar{\alpha}}\circ g_{\bar{\alpha}}\mathrm{d\alpha}
\]
where for almost every $\bar{\alpha}$, the measure $\mu_{\bar{\alpha}}$
is supported on $Y_{S}$ and $\mathrm{d\alpha}$ is the pushforward
of $\mu$ to $K^{1}$. Since $\SL_{d}\left(\RR\right)$ acts $Y_{S}$
and commutes with the elements from $K^{1}$, if $\mu$ is $A$ (resp.
$U$)- invariant, then we may assume that the $\mu_{\bar{\alpha}}$
are also $A$ (resp. $U$)-invariant for almost every $\bar{\alpha}$.
Like in \lemref{increase_partial_invariance} above, projecting this
decomposition to $X_{\RR}$, we obtain a convex decomposition of the
Haar measure, so that \lemref{increase_partial_invariance} implies
that $\mu_{\bar{\alpha}}$ is $\SL_{d}\left(\RR\right)$-invariant
for almost every $\bar{\alpha}$. Finally, this in turn implies that
$\mu$ is $\SL_{d}\left(\RR\right)$-invariant which is what we wanted
to show.

\newpage{}
\end{proof}

\subsection{\label{subsec:automatice_invariance}From $\protect\SL_{d}\left(\protect\RR\right)$-invariance
to $\protect\SL_{d}\left(\protect\RR\right)\times{\displaystyle \prod_{p\in S\backslash\left\{ \infty\right\} }}\protect\SL_{d}\left(\protect\ZZ_{p}\right)$-invariance}

Recall that our measure is on the space $X_{S}\cong\SL_{d}\left(\ZZ\right)\backslash H_{S}$
where $\SL_{d}\left(\ZZ\right)$ is embedded diagonally in \\
$H_{S}=\SL_{d}\left(\RR\right)\times{\displaystyle \prod_{p\in S\backslash\left\{ \infty\right\} }}\GL_{d}\left(\ZZ_{p}\right)$.
In the previous section we showed how to lift Haar measures on $X_{\RR}$
to right $\SL_{d}\left(\RR\right)$-invariance on $X_{S}$ for some
finite $S\subseteq\PP_{\infty}$. Now we show how to extend it to
invariance under 
\[
W_{S}:=\SL_{d}\left(\RR\right)\times{\displaystyle \prod_{p\in S\backslash\left\{ \infty\right\} }}\SL_{d}\left(\ZZ_{p}\right),
\]
where the main trick is that in $X_{S}$ we mod out from the left
with $\SL_{d}\left(\ZZ\right)$ which ``mixes'' the coordinates
of the real and prime places..

Two important details for this step is that (from the right) $\SL_{d}\left(\RR\right)$
is a unimodular, cocompact normal subgroup of $H_{S}$ and (from the
left) we have the weak approximation, namely $\ZZ$ is dense in $\prod_{p\in S\backslash\left\{ \infty\right\} }\ZZ_{p}$.
This will help us to move between right and left invariance in $H_{S}$
and to obtain a bigger invariance under 
\[
W_{S}=\overline{\left\langle \SL_{d}\left(\RR\right),\SL_{d}\left(\ZZ\right)\right\rangle _{H_{S}}}.
\]
Note that this bigger group is exactly the kernel of $\det_{S\backslash\left\{ \infty\right\} }$
given in \defref{determinant}, when restricted to $H_{S}$, namely
\[
\det_{S\backslash\left\{ \infty\right\} }:H_{S}\overset{\det}{\longrightarrow}\RR^{\times}\times{\displaystyle \prod_{p\in S\backslash\left\{ \infty\right\} }}\ZZ_{p}^{\times}\longrightarrow{\displaystyle \prod_{p\in S\backslash\left\{ \infty\right\} }}\ZZ_{p}^{\times}.
\]
In particular, like $\SL_{d}\left(\RR\right)$, the group $W_{S}$
is a cocompact, unimodular and normal subgroup of $H_{S}$ as well. 

Actually, both of these groups satisfy a stronger condition - in the
first case $H_{S}$ can be written as a direct product of $\SL_{d}\left(\RR\right)$
with another (compact, unimodular) group, and in the second case,
the identification of $\bar{\alpha}\in{\displaystyle \prod_{p\in S\backslash\left\{ \infty\right\} }}\ZZ_{p}^{\times}\mapsto g_{\bar{\alpha}}$
from \defref{determinant} shows that $H_{S}=W_{S}\cdot\left({\displaystyle \prod_{p\in S\backslash\left\{ \infty\right\} }}\ZZ_{p}^{\times}\right)$
and $W_{S}\cap\left({\displaystyle \prod_{p\in S\backslash\left\{ \infty\right\} }}\ZZ_{p}^{\times}\right)=\left\{ Id\right\} $.

With this in mind, we have the following result about disintegration
of (locally finite) measures.
\begin{thm}
\label{thm:disintegration_locally_finite}Let $H$ be a unimodular
group, $W\leq H$ a unimodular normal subgroup, and $K\leq H$ a compact
subgroup such that $H=W\cdot K$ and $W\cap K=\left\{ e\right\} $.
Denote by $\pi:H\to W\backslash H\cong K$ the natural projection
and by $\mu_{W}$ the $W$-invariant measure on $W$. If $\mu$ is
a left $W$-invariant locally finite measure on $H$, then there exist
$r_{k}\geq0$ for $k\in K$ such that 
\[
\int_{H}f\left(g\right)\dmu\left(g\right)=\int_{K}\left(\int_{W}f\left(hk\right)\dmu_{W}\left(h\right)\right)r_{k}\dnu\left(k\right).
\]

A similar claim holds for right $W$-invariant measures.
\end{thm}

The proof of \ref{thm:disintegration_locally_finite} uses the standard
arguments for disintegration of measures. For completeness, we add
its proof in \ref{sec:Disintegration-of-measures}.

\newpage{}
\begin{cor}
\label{cor:left_to_right}Let $H,W,K$ be as in \ref{thm:disintegration_locally_finite}.
Then a right locally finite measure on $H$ is left $W$-invariant
if and only if it is right $W$-invariant.
\end{cor}

\begin{proof}
Let $\mu$ be a left $W$-invariant measure and fix some $h_{0}\in W$.
Then we have that
\[
\mu\left(R_{h_{0}}\left(f\right)\right)=\int_{K}\left(\int_{W}f\left(hkh_{0}\right)\dmu_{W}\left(h\right)\right)r_{k}\dnu\left(k\right)=\int_{K}\left(\int_{W}f\left(h\left(kh_{0}k^{-1}\right)k\right)\dmu_{W}\left(h\right)\right)r_{k}\dnu\left(k\right)
\]
Since $W$ is normal we have that $kh_{0}k^{-1}\in W$, and because
$W$ is unimodular, its left Haar measure is also right Haar, so that
$\int_{H}f\left(h\left(kh_{0}k^{-1}\right)k\right)\dmu_{W}\left(h\right)=\int_{H}f\left(hk\right)\dmu_{W}\left(h\right)$.
It follows that $\mu\left(R_{h_{0}}\left(f\right)\right)=\mu\left(f\right)$,
so that $\mu$ is also right $W$-invariant. The same argument show
that right implies left $W$-invariance which complete the proof. 
\end{proof}
We can now show how to extend the $\SL_{d}\left(\RR\right)$-invariance
to the $W_{S}$-invariance.
\begin{lem}[Unique Ergodicity]
\label{lem:Unique_Ergodicity}Let $S\subseteq\PP_{\infty}$ be finite
and let $\mu_{S}$ be a $\SL_{d}\left(\RR\right)$-invariant probability
measure on $X_{S}$. Then $\mu_{S}$ must be $W_{S}$-invariant.
\end{lem}

\begin{proof}
Let $\tilde{\mu}_{S}$ be the lift of $\mu_{S}$ from $\SL_{d}\left(\ZZ\right)\backslash H_{S}$
to $H_{S}$, i.e. for sets $F$ inside the fundamental domain we set
$\tilde{\mu}_{S}\left(F\right)=\mu_{S}\left(\SL_{d}\left(\ZZ\right)F\right)$,
and extend this to a left $\SL_{d}\left(\ZZ\right)$-invariant measure
on $G_{S}$. The measure $\tilde{\mu}_{S}$ is left $\SL_{d}\left(\ZZ\right)$
(diagonally) and right $\SL_{d}\left(\RR\right)$-invariant measure,
so by \corref{left_to_right} it is also $\SL_{d}\left(\RR\right)$-left
invariant. Using the weak approximation of $\ZZ$ in $\prod_{p\in S\backslash\left\{ \infty\right\} }\ZZ_{p}$
we get that $\tilde{\mu}_{S}$ is $W_{S}:=\overline{\left\langle \SL_{d}\left(\ZZ\right),\SL_{d}\left(\RR\right)\right\rangle }=\SL_{d}\left(\RR\right)\times\prod_{p\in S\backslash\left\{ \infty\right\} }\SL_{d}\left(\ZZ_{p}\right)$-invariant.
Applying \corref{left_to_right} again, we obtain that $\tilde{\mu}_{S}$
and therefore $\mu_{S}$ is right $W_{S}$-invariant.
\end{proof}

\subsubsection{\label{subsec:to_full_invariance}From $W_{S}$ to $G_{S}$-invariance}

Finally, we want to extend the $W_{S}$-invariance from the previous
section to the full $G_{S}$-invariance for $S\subseteq\PP_{\infty}$
finite. The first observation is that it is enough to show $H_{S}$-invariance.
This is because there is a unique $H_{S}$-invariant measure on $X_{S}=\SL_{d}\left(\ZZ\right)\backslash H_{S}$
(up to normalization) and the $G_{S}$-invariant measure is in particular
$H_{S}$-invariant, so it must be this unique measure.

In order to show the $H_{S}$-invariance, we consider again the map
${\displaystyle \det_{S\backslash\left\{ \infty\right\} }}:H_{S}\to{\displaystyle \prod_{p\in S\backslash\left\{ \infty\right\} }}\ZZ_{p}^{\times}$
defined in the previous section. This map is also well defined on
$X_{S}=\SL_{d}\left(\ZZ\right)\backslash H_{S}$ and by abuse of notation
we will denote it also with ${\displaystyle \det_{S\backslash\left\{ \infty\right\} }}$.
Thus, the last ingredient that we need, is that the pushforward of
the measure to ${\displaystyle \prod_{p\in S\backslash\left\{ \infty\right\} }}\ZZ_{p}^{\times}$
will also be the Haar measure.
\begin{lem}
\label{lem:finite_index_invariance} Let $H,W$ and $K$ be as in
\ref{thm:disintegration_locally_finite} and let $\Gamma\leq H$ be
a lattice which is also contained in $W$. Then a probability measure
$\mu$ on $\Gamma\backslash H$ is $H$-invariant if it is $W$-invariant
and it projection to $W\backslash H$ via $\Gamma\backslash H\to W\backslash H\cong K$
is $K$-invariant.
\end{lem}

\begin{proof}
Using the standard disintegration of measures (see for example section
5.3 in \cite{einsiedler_ergodic_2010}) for the map $\Gamma\backslash H\to W\backslash H\cong K$,
we can write $\mu$ as
\[
\mu\left(f\right)=\int_{K}\left(\int_{\Gamma\backslash W}f\left(\Gamma wk\right)\dmu_{k}\right)\dnu
\]
where $\dnu$ is the pushforward of the measure $\mu$ to $K$ and
$\dmu_{k}$ are supported on $\Gamma\backslash W$. Moreover, the
measures $\dmu_{k}$ are uniquely defined for $\nu$ almost every
$k$. Note that since $W\idealeq H$ and $\mu$ is right $W$-invariant,
for any $w_{0}\in W$ we have that 
\begin{align*}
\mu\left(f\right) & =\mu\left(R_{w_{0}}\left(f\right)\right)=\int_{K}\left(\int_{\Gamma\backslash W}f\left(\Gamma w\left(kw_{0}k^{-1}\right)k\right)\dmu_{k}\right)\dnu.
\end{align*}
But the $\mu_{k}$ are uniquely defined (almost everywhere), so they
must also be $kw_{0}k^{-1}$-invariant. Doing this for a countable
dense set of $W_{0}\subseteq W$ we conclude for $\nu$ almost every
$k$ the measure $\mu_{k}$ is $W_{0}$-invariant, and therefore $W=\overline{W_{0}}$-invariant.
Since there is a unique such probability measure, these are all the
same measure $\mu_{\Gamma\backslash W}$, and therefore 
\[
\mu\left(f\right)=\int_{K}\left(\int_{\Gamma\backslash W}f\left(\Gamma wk\right)\dmu_{\Gamma\backslash W}\right)\dnu.
\]
It now follows that $\mu$ is also right $K$-invariant and therefore
$W\cdot K=H$-invariant.
\end{proof}
We are now ready to put all the results together.
\begin{thm}
\label{thm:main_lifting}Let $S\subseteq\PP_{\infty}$ (may be infinite)
and $\mu_{S}$ a probability measure on $X_{S}$. Denote by $\mu_{\RR}$
the projection $\pi_{\RR}^{S}\left(\mu_{S}\right)$. Suppose that:
\begin{enumerate}
\item ($\RR$-uniformity) $\mu_{\RR}$ is the $\SL_{d}\left(\RR\right)$-invariant
measure on $X_{\RR}$,
\item ($\RR$-invariance) $\mu_{S}$ is invariant under some $id\neq a\in A$
or under some one parameter unipotent subgroup in $\SL_{d}\left(\RR\right)$,
and 
\item (prime-uniformity) for any $S_{0}\subseteq S$ finite, the pushforward
${\displaystyle \det_{S_{0}\backslash\left\{ \infty\right\} }}\left(\pi_{S_{0}}^{S}\left(\mu_{S}\right)\right)$
to ${\displaystyle \prod_{p\in S_{0}\backslash\left\{ \infty\right\} }}\ZZ_{p}^{\times}$
is the Haar measure.
\end{enumerate}
Then $\mu_{S}$ is the $G_{S}$-invariant probability.
\end{thm}

\begin{proof}
We begin with the proof for $S\subseteq\PP_{\infty}$ finite. In this
case, conditions (1) and (2) with \thmref{increase_invariance} imply
that $\mu_{S}$ is $\SL_{d}\left(\RR\right)$-invariant. Then using
\lemref{Unique_Ergodicity} we get that it is $\SL_{d}\left(\RR\right)\times\prod_{p\in S}\SL_{d}\left(\ZZ_{p}\right)$-invariant.
Finally, condition (3) together with \lemref{finite_index_invariance}
imply that $\mu_{S}$ is $H_{S}$-invariant.

Assume now that $S$ is infinite. For any $S_{0}\subseteq S$ finite
we can pull back the functions in $C_{c}\left(X_{S_{0}}\right)$ to
$C_{c}\left(X_{S}\right)$ and using the Stone-Weierstrass theorem
we get that the union of these sets over these $S_{0}$ spans a dense
subset of $C_{c}(X_{S})$. Hence, it is enough to prove that for any
such set $S_{0}$, $f\in C_{c}\left(X_{S_{0}}\right)$ and $g\in H_{S}$
we have that $\mu_{S}\left(g\left(f\circ\pi_{S_{0}}^{S}\right)\right)=\mu_{S}\left(f\circ\pi_{S_{0}}^{S}\right)$.
The function $f\circ\pi_{S_{0}}^{S}$ is already invariant under $g\in H_{S}$
which is the identity in the $S_{0}$ places (because $f$ is invariant
there), so it is enough to prove this for $g\in H_{S_{0}}$, which
then satisfies 
\[
\mu_{S}\left(g\left(f\circ\pi_{S_{0}}^{S}\right)\right)=\mu_{S}\left(g\left(f\right)\circ\pi_{S_{0}}^{S}\right)=\pi_{S_{0}}^{S}\left(\mu_{S}\right)\left(g\left(f\right)\right).
\]
The measure $\pi_{S_{0}}^{S}\left(\mu_{S}\right)$ on $X_{S_{0}}$
also satisfies all the condition of this theorem and $S_{0}$ is finite,
so that $\pi_{S_{0}}^{S}\left(\mu_{S}\right)$ is $H_{S_{0}}$-invariant.
It follows that the expression above equals to $\pi_{S_{0}}^{S}\left(\mu_{S}\right)\left(f\right)=\mu_{S}\left(f\circ\pi_{S_{0}}^{S}\right)$
which is what we wanted to show.
\end{proof}

\newpage{}

\section{\label{sec:Adelic_translations-1}Adelic translations}

\global\long\def\norm#1{\left\Vert #1\right\Vert }%
\global\long\def\AA{\mathbb{A}}%
\global\long\def\QQ{\mathbb{Q}}%
\global\long\def\PP{\mathbb{P}}%
\global\long\def\CC{\mathbb{C}}%
\global\long\def\HH{\mathbb{H}}%
\global\long\def\ZZ{\mathbb{Z}}%
\global\long\def\NN{\mathbb{N}}%
\global\long\def\KK{\mathbb{K}}%
\global\long\def\RR{\mathbb{R}}%
\global\long\def\FF{\mathbb{F}}%
\global\long\def\oo{\mathcal{O}}%
\global\long\def\aa{\mathcal{A}}%
\global\long\def\bb{\mathcal{B}}%
\global\long\def\ff{\mathcal{F}}%
\global\long\def\mm{\mathcal{M}}%
\global\long\def\limfi#1#2{{\displaystyle \lim_{#1\to#2}}}%
\global\long\def\pp{\mathcal{P}}%
\global\long\def\qq{\mathcal{Q}}%
\global\long\def\da{\mathrm{da}}%
\global\long\def\dt{\mathrm{dt}}%
\global\long\def\dg{\mathrm{dg}}%
\global\long\def\ds{\mathrm{ds}}%
\global\long\def\dm{\mathrm{dm}}%
\global\long\def\dmu{\mathrm{d\mu}}%
\global\long\def\dx{\mathrm{dx}}%
\global\long\def\dy{\mathrm{dy}}%
\global\long\def\dz{\mathrm{dz}}%
\global\long\def\dnu{\mathrm{d\nu}}%
\global\long\def\flr#1{\left\lfloor #1\right\rfloor }%
\global\long\def\nuga{\nu_{\mathrm{Gauss}}}%
\global\long\def\diag#1{\mathrm{diag}\left(#1\right)}%
\global\long\def\bR{\mathbb{R}}%
\global\long\def\Ga{\Gamma}%
\global\long\def\PGL{\mathrm{PGL}}%
\global\long\def\GL{\mathrm{GL}}%
\global\long\def\PO{\mathrm{PO}}%
\global\long\def\SL{\mathrm{SL}}%
\global\long\def\PSL{\mathrm{PSL}}%
\global\long\def\SO{\mathrm{SO}}%
\global\long\def\mb#1{\mathrm{#1}}%
\global\long\def\wstar{\overset{w^{*}}{\longrightarrow}}%
\global\long\def\vphi{\varphi}%
\global\long\def\av#1{\left|#1\right|}%
\global\long\def\inv#1{\left(\mathbb{Z}/#1\mathbb{Z}\right)^{\times}}%
\global\long\def\cH{\mathcal{H}}%
\global\long\def\cM{\mathcal{M}}%
\global\long\def\bZ{\mathbb{Z}}%
\global\long\def\bA{\mathbb{A}}%
\global\long\def\bQ{\mathbb{Q}}%
\global\long\def\bP{\mathbb{P}}%
\global\long\def\eps{\epsilon}%
\global\long\def\on#1{\mathrm{#1}}%
\global\long\def\nuga{\nu_{\mathrm{Gauss}}}%
\global\long\def\set#1{\left\{  #1\right\}  }%
\global\long\def\smallmat#1{\begin{smallmatrix}#1\end{smallmatrix}}%
\global\long\def\len{\mathrm{len}}%
\global\long\def\idealeq{\trianglelefteqslant}%

In this section we consider translations of orbit measures over the
adeles, where the end goal is to show that the limit is the uniform
Haar measure. In \thmref{main_lifting} we gave some conditions that
imply that a probability measure is the Haar measure. However, our
orbit translations are only locally finite and not finite, so we begin
this section with the definition and some basic results about such
measures. 

In \subsecref{Orbit_measures} we define what are orbit measures and
their translation, and using an Iwasawa decomposition over the adeles,
we show that in our translation result we only need to consider very
special type of unipotent matrices. In particular this new presentation
will allow us to show that the limit measure (if it exists) will be
either $A$ or $U$-invariant, which is the $\RR$-invariance condition
in \thmref{main_lifting}.

In \subsecref{Uniform_prime_invariance} we prove that any limit of
our translated orbits will satisfy the prime invariance from \thmref{main_lifting}.
In order to do that we need first to show that we can restrict our
infinite measure on the translated divergent orbit to a finite part,
by removing the parts ``close'' to the cusp. Also, we will utilize
some symmetry to cut even this finite part in half. This will help
us later on in \secref{real_case} when we use these measures to approximate
expanding horocycles which will give us the last $\RR$-uniformity
condition that we need for \thmref{main_lifting}.

\subsection{\label{subsec:Locally_finite}Locally finite measures}

So far, all of our spaces $X_{S}$ and the groups are locally compact,
second countable Hausdorff spaces. We will now give the definitions
for locally finite measures on such spaces.
\begin{defn}
\label{def:locally_finite}Let $Z$ be a locally compact second countable
space and denote by $\mm\left(Z\right)$ the set of all locally finite
measures on $Z$, namely measures $\mu$ such that $\mu\left(K\right)<\infty$
for any $K\subseteq Z$ compact. Since locally finite measures don't
have a natural normalization, we define $\mathcal{P}\mm\left(Z\right)$
to be homothety classes of nonzero measures in $\mm\left(Z\right)$
and for $0\neq\mu\in\mm\left(Z\right)$ we denote its class by $\left[\mu\right]\in\mathcal{P}\mm\left(Z\right)$.
In other words $\left[\mu_{1}\right]=\left[\mu_{2}\right]$ if there
is some $c>0$ such that $\mu_{1}=c\mu_{2}$.
\begin{itemize}
\item For $\mu_{i},\mu_{\infty}\in\mm\left(Z\right)$ we say that $\mu_{i}\wstar\mu_{\infty}$
if $\mu_{i}\left(f\right)\to\mu_{\infty}\left(f\right)$ for every
$f\in C_{c}\left(Z\right)$. 
\item If $\mu_{i},\mu_{\infty}$ are nonzero, we will write $\left[\mu_{i}\right]\to\left[\mu_{\infty}\right]$
if $\exists d_{i}>0$ such that $d_{i}\mu_{i}\wstar\mu_{\infty}$.
\end{itemize}
\end{defn}

It is not hard to check that the convergence in $\pp\mm\left(Z\right)$
is equivalent to the following definitions (see for example \cite{shapira_limiting_nodate}):
\begin{enumerate}
\item There exist positive scalars $c_{i}>0$ such that $c_{i}\mu_{i}\left(f\right)\to\mu\left(f\right)$
for any $f\in C_{c}\left(Z\right)$.
\item There exist positive scalars $c_{i}>0$ such that $c_{i}\mu_{i}\mid_{K}\wstar\mu\mid_{K}$
for any compact subset $K\subseteq Z$.
\item For any two $f_{1},f_{2}\in C_{c}\left(Z\right)$ with $\mu\left(f_{2}\right)\neq0$
we have that $\frac{\mu_{i}\left(f_{1}\right)}{\mu_{i}\left(f_{2}\right)}\to\frac{\mu\left(f_{1}\right)}{\mu\left(f_{2}\right)}$.
\end{enumerate}
The last definition let us define a topology on $\pp\mm\left(Z\right)$.
If $\left[\mu\right]\in\pp\mm\left(Z\right)$, then the basic open
sets containing $\left[\mu\right]$ are of the form
\[
V_{\left(\mu,f_{1},f_{2},\varepsilon\right)}:=\left\{ \nu\;\mid\;\left|\frac{\nu\left(f_{1}\right)}{\nu\left(f_{2}\right)}-\frac{\mu\left(f_{1}\right)}{\mu\left(f_{2}\right)}\right|<\varepsilon\right\} ,
\]
where $f_{1},f_{2}\in C_{c}\left(Z\right)$, $\mu\left(f_{2}\right)\neq0$
and $\varepsilon>0$.\\

Note that if $\psi:Z_{1}\to Z_{2}$ is proper, i.e. the preimage of
a compact set is compact, then for $\mu\in\mm\left(Z_{1}\right)$
we have that $\mu\circ\psi^{-1}\in\mm\left(Z_{2}\right)$. Abusing
our notations, we shall also denote by $\psi$ the induced maps $\mm\left(Z_{1}\right)\to\mm\left(Z_{2}\right)$
and $\mathcal{P}\mm\left(Z_{1}\right)\to\mathcal{P}\mm\left(Z_{2}\right)$.\\

Our spaces will usually have some group action on them (mainly $G_{S}$
and $H_{S}$), and the next lemma shows that the induced action on
the locally finite measures is continuous, if the action of $G$ is
continuous.
\begin{lem}
\label{lem:uniform_continuity}Let $G$ act strongly on the space
$Z$ (the map $\left(g,z\right)\mapsto gz$ is continuous). Then any
\\
$f\in C_{c}\left(Z\right)$ is uniformly continuous, namely for every
$\varepsilon>0$ there is some open neighborhood \\
$e\in U\subseteq G$, such that for all $g\in U$ we have that $\norm{f-f\circ g}_{\infty}<\varepsilon$.
\end{lem}

\begin{proof}
Let $\varepsilon>0$. Choose some symmetric open neighborhood $V$
of $e\in G$ with compact closure so that $K=\overline{supp\left(f\right)}\cdot\overline{V}$
is compact. It follows that $f$ is zero on $K^{c}\cdot V$, so that
for any $g\in V$ we have that $\norm{f-f\circ g}_{K^{c},\infty}=0$,
we we only need to worry about what happens inside the set $K$.

Suppose that for every $W\subseteq V$ there exists $x_{W}\in X$
and $w\in W$ such that $\left|f\left(wx_{W}\right)-f\left(x\right)\right|\geq\varepsilon$,
so in particular $x_{W}\in K$. The net $W\mapsto x_{W}$ has its
image in a compact set, and therefore has a convergent subnet to some
$x\in K$, and we restrict ourselves to this subnet. The composition
$X\times G\to X\overset{f}{\to}\RR$ is continuous, hence we can find
$x\in N_{1}\subseteq X$ and $e\in W_{1}\subseteq G$ open such that
$f\left(W_{1}\cdot N_{1}\right)\subseteq B_{\varepsilon/2}\left(f\left(x\right)\right)$.
By the convergence of $x_{W}$ to $x$, we can find $W\subseteq W_{1}$
such that $x_{W}\in N_{1}$, but then 
\[
f\left(Wx_{W}\right)\subseteq f\left(W_{1}N_{1}\right)\subseteq B_{\varepsilon/2}\left(f\left(x\right)\right)\quad\Rightarrow\quad f\left(Wx_{W}\right)\subseteq B_{\varepsilon}\left(f\left(x_{W}\right)\right)
\]
in contradiction to the choice of $x_{W}$. Thus, we proved that there
exists $W_{f,\varepsilon}\subseteq V$ (which we may assume to be
symmetric) such that for all $g\in W_{f,\varepsilon}$ we have that
$\norm{f-f\circ g}_{\infty}<\varepsilon$. 
\end{proof}
\begin{lem}
\label{lem:measure_action_cont}Let $G$ be a locally compact group
acting strongly on a locally compact, second countable Hausdorff space
$Z$. Then the action map $G\times\pp\mm\left(Z\right)\to\pp\mm\left(Z\right)$
defined by $\left(g,\mu\right)\mapsto g\mu$ is continuous.
\end{lem}

\begin{proof}
We want to show that given $\left(g,\left[\mu\right]\right)\in G\times\pp\mm\left(Z\right)$
and any $\varepsilon>0$, $f_{1},f_{2}\in C_{c}\left(Z\right)$ such
that $\left(g\mu\right)\left(f_{2}\right)\neq0$ we have that 
\[
\left|\frac{h\nu\left(f_{1}\right)}{h\nu\left(f_{2}\right)}-\frac{g\mu\left(f_{1}\right)}{g\mu\left(f_{2}\right)}\right|<\varepsilon
\]
for every $\left(h,\left[\nu\right]\right)\in G\times\pp\mm\left(Z\right)$
is a small enough neighborhoods of $g$ and $\left[\mu\right]$ respectively.
Changing $f_{i}$ to $g^{-1}\left(f_{i}\right)$ for $i=1,2$, we
may assume that $g=e$. 

The triangle inequality implies that 
\[
\left|\frac{h\nu\left(f_{1}\right)}{h\nu\left(f_{2}\right)}-\frac{\mu\left(f_{1}\right)}{\mu\left(f_{2}\right)}\right|\leq\left|\frac{h\nu\left(f_{1}\right)}{h\nu\left(f_{2}\right)}-\frac{\nu\left(f_{1}\right)}{\nu\left(f_{2}\right)}\right|+\left|\frac{\nu\left(f_{1}\right)}{\nu\left(f_{2}\right)}-\frac{\mu\left(f_{1}\right)}{\mu\left(f_{2}\right)}\right|
\]
so if $\nu$ is close enough to $\mu$ we may assume that the second
summand is $<\frac{\varepsilon}{2}$.

For the first summand, use \lemref{uniform_continuity} to find for
any $\varepsilon'>0$ a symmetric open set $e\in U_{\varepsilon'}\subseteq G$
with compact closure so that $\norm{f_{1}-f_{1}\circ h},\norm{f_{2}-f_{2}\circ h}<\varepsilon'$
for all $h\in U_{\varepsilon'}$. Then for $i=1,2$ we get that 
\[
\left|h\nu\left(f_{i}\right)-\nu\left(f_{i}\right)\right|=\left|\nu\left(f_{i}\circ h-f_{i}\right)\right|\leq\nu\left(\overline{U}\cdot supp\left(f_{i}\right)\right)\varepsilon',
\]
so that $h\to h\nu\left(f_{i}\right)$ is continuous at $h=e$. Thus,
for $h$ small enough we get that $\left|\frac{h\nu\left(f_{1}\right)}{h\nu\left(f_{2}\right)}-\frac{\nu\left(f_{1}\right)}{\nu\left(f_{2}\right)}\right|<\frac{\varepsilon}{2}$
which completes the proof.
\end{proof}
The result above is well known for probability measures, and it has
three immediate corollaries which we will use. 

\newpage{}
\begin{cor}
\label{cor:measure_action}Let $G$ and $Z$ be as in \lemref{measure_action_cont}.
\begin{enumerate}
\item If $\left[\mu\right]\in\pp\mm\left(Z\right)$, then $stab_{G}\left(\left[\mu\right]\right)$
is closed in $G$.
\item If $\left[\mu_{i}\right]\to\left[\mu\right]$ , $\left[\mu_{i}\right]$
is $g_{i}$-invariant and $g_{i}\to g$ in $G$, then $\left[\mu\right]$
is $g$-invariant.
\item If $\left[\mu_{i}\right]\to\left[\mu\right]$ and $K\subseteq stab_{G}\left(\left[\mu\right]\right)$
is some compact set, then for any $k_{i}\in K$ we also have that
$\left[k_{i}\mu_{i}\right]\to\left[\mu\right]$.
\end{enumerate}
\end{cor}

The last corollary above is very useful, since if $\mu$ is a $G$-invariant
measure, then we can take $K$ to be any compact subset of $G$. Thus,
when speaking about translations, we can always shift the translations
by some elements from a compact set.

\subsection{\label{subsec:Orbit_measures}Orbit measures, translations and the
$\protect\RR$-invariance}

In this section we begin to study the orbit measures. We start with
a general definition of an orbit measure, which we will later use
mainly for the diagonal group and its subgroups.
\begin{defn}
Fix some $L\leq G_{S}$ and $x\in X_{S}$ such that $stab_{L}\left(x\right)\backslash L$
supports an $L$-invariant measure and the map $stab_{L}\left(x\right)g\mapsto xg$
is proper. Then the orbit $xL$ supports an $L$-invariant measure
which is locally finite. We call this measure the \emph{orbit measure}
of $L$ and denote it by $\delta_{xL}$.
\end{defn}

If the $L$-invariant measure on $stab_{L}\left(x\right)\backslash L$
is finite, then we may normalize $\delta_{xL}$ to be a probability
measure. In any case the homothety class of $\delta_{xL}$ will always
be well defined regardless of the normalization, and if the measure
is finite or not.

Recall that we use the following notation for (real) diagonal and
unipotent matrices{\footnotesize{}
\begin{align*}
U & =\left\{ u_{\alpha}=\left(\begin{array}{cc}
1 & \alpha\\
0 & 1
\end{array}\right)\;\mid\;\alpha\in\RR\right\} \\
A & =\left\{ a\left(t\right)=\left(\begin{array}{cc}
e^{-t/2} & 0\\
0 & e^{t/2}
\end{array}\right)\;\mid\;t\in\RR\right\} ,
\end{align*}
}which we always consider as subgroups of $G_{S}$ (in the real coordinate).
\begin{example}
\begin{enumerate}
\item The orbit measure $\delta_{x_{\RR}U}$ is just the Lebesgue measure
on $S^{1}=\nicefrac{\RR}{\ZZ}$ pushed to the horocycle $\Gamma_{\RR}U$.
This is because $U\cong\RR$, while $stab_{U}\left(\Gamma_{\RR}\right)\cong\ZZ\leq\RR$. 
\item The orbit measure $\delta_{\Gamma_{\RR}A}$ is a locally finite measure,
but not a probability. On the other hand, for almost every $x\in\SL_{2}\left(\ZZ\right)\backslash\SL_{2}\left(\RR\right)$,
the orbit $xA$ is dense (and the map $a\mapsto xa$ is not proper),
so that $\delta_{xA}$ is not locally finite.
\end{enumerate}
The second diagonal example above will be the main orbit measure that
we work with, though we will see the unipotent example too. As we
said before, from now on we will restrict our attention to dimension
2, though almost everything in this section can still be generalized
to higher dimension with the right formulation.
\end{example}

\begin{defn}[Diagonal subgroups]
 For $\nu\in\PP_{\infty}$ be a place and let $A_{\nu}$ be the diagonal
matrices in $\GL_{2}\left(\QQ_{\nu}\right)$. We additionally set
\begin{align*}
A_{p}^{+} & =A_{p}\cap\GL_{d}\left(\ZZ_{p}\right)=\left\{ diag\left(\alpha,\beta\right)\;\mid\;\alpha,\beta\in\ZZ_{p}^{\times}\right\} ,\quad p\;is\;prime\\
A_{\infty}^{+} & =A=\left\{ diag\left(e^{-t},e^{t}\right)\;|\;t\in\RR\right\} .
\end{align*}
For general $S\subseteq\PP_{\infty}$ (possibly doesn't contain $\infty$)
we set $A_{S}$ to be the diagonal subgroup in $G_{S}$, i.e. the
restricted product $G_{S}\cap\prod_{\nu\in S}'A_{\nu}$ with respect
to $A_{\nu}^{+}$, and set $A_{S}^{+}=\prod_{\nu\in S}A_{\nu}^{+}$.
\end{defn}

\begin{rem}
While in the prime places $A_{p}^{+}\cong\left(\ZZ_{p}^{\times}\right)^{2}$
is two dimensional, in the real place $A_{\infty}^{+}\cong\RR$ is
one dimensional. The reason for that is that in $\GL_{2}^{1}\left(\AA\right)$
we have the extra condition that \\
$\left|\det\right|\left(a\right)=\prod_{\nu}\left|\det\left(a^{\left(\nu\right)}\right)\right|_{\nu}=1$,
so we lose one dimension. In the standard diagonal subgroup we instead
simply intersect with $G_{S}$.
\end{rem}

\begin{defn}
We denote by $x_{S}=\Gamma_{S}\cdot Id\in X_{S}$ the origin in $X_{S}$. 
\end{defn}

Our main interest will be the orbit measures $\delta_{x_{S}A_{S}}$
and their translations. We begin with the simple observation that
$\QQ_{p}^{\times}\cong p^{\ZZ}\ZZ_{p}^{\times}$, so that $A_{p}\cong\left(p^{\ZZ}\right)^{2}\times A_{p}^{+}$
and use it to give a simpler presentation of $x_{S}A_{S}$. 
\begin{lem}
For every $\infty\in S\subseteq\PP_{\infty}$, the map $A_{S}^{+}\to X_{S}:a\mapsto x_{S}a$
is a proper and bijective map onto the orbit $x_{S}A_{S}$. In particular
it follows that $\delta_{x_{S}A_{S}}=\delta_{x_{S}A_{S}^{+}}$ is
the pushforward of the $A_{S}^{+}$-Haar measure.
\end{lem}

\begin{proof}
Note first that the group $A_{S}^{+}$ is open inside $A_{S}$, so
that the Haar measure on $A_{S}^{+}$ is just the restriction of the
Haar measure from $A_{S}$. 

We claim that $A_{S}=stab_{A_{S}}\left(x_{S}\right)A_{S}^{+}$ and
$stab_{A_{S}}\left(x_{S}\right)\cap A_{S}^{+}=\left\{ Id\right\} $
which implies that \\
$\delta_{x_{S}A_{S}}=\delta_{x_{S}A_{S}^{+}}$. Indeed, if $\left(a_{\nu}\right)\in\prod_{\nu\in S}\QQ_{\nu}^{\times}$
then 
\[
b:=sign\left(a^{\left(\infty\right)}\right)\cdot\prod_{p\in S\backslash\left\{ \infty\right\} }\left|a^{\left(p\right)}\right|_{p}\in\ZZ\left[S^{-1}\right]^{\times},
\]
and $ba\in\RR_{>0}\times\prod_{p}\ZZ_{p}^{\times}$. Extending this
to the diagonal matrices we get that \\
$\left(\Gamma_{S}\cap A_{S}\right)\cdot A_{S}^{+}=A_{S}$ where $\Gamma_{S}\cap A_{S}=stab_{A_{S}}\left(x_{S}\right)$.
Since $\ZZ\left[S^{-1}\right]^{\times}\cap\RR_{>0}\cap{\displaystyle \bigcap_{p\in S\backslash\left\{ \infty\right\} }}\ZZ_{p}^{\times}=\left\{ 1\right\} $,
we obtain that $stab_{A_{S}}\left(x_{S}\right)\cap A_{S}^{+}=\left\{ Id\right\} $. 

For the properness, we use the generalized Mahler's criterion from
\lemref{generalized_mahler} which shows that for $a\in A_{S}^{+}\leq H_{S}$,
the height function is simply $ht_{S}\left(x_{S}a\right)=ht_{\infty}\left(\ZZ^{2}a^{\left(\infty\right)}\right)$.
Thus, being in a compact set means bounding the $a^{\left(\infty\right)}$,
and since in the prime places $\prod A_{p}^{+}$ the group is already
compact, we see that the inverse of a compact set is compact.
\end{proof}
For different $S\subseteq\PP_{\infty}$, the measures $\delta_{x_{S}A_{S}}$
live in different spaces. However, if $S\subseteq S'\subseteq\PP_{\infty}$,
then it is easy to check that $\pi_{S}^{S'}\left(x_{S'}A_{S'}\right)=x_{S}A_{S}$,
so when we choose the normalization for these locally finite measures
we do so that $\pi_{S}^{S'}\left(\mu_{S'}\right)=\mu_{S}$. To be
more precise, we start by fixing an $A_{S}^{+}$-Haar measure $\eta_{S}$
on $A_{S}^{+}$ for each $S\subseteq\PP_{\infty}$. Note that for
$\Omega\subseteq A=A_{\infty}^{+}$, the map $\Omega\mapsto\eta_{S,\RR}\left(\Omega\right):=\eta_{S}\left(\Omega\times\prod_{p\in S}A_{p}^{+}\right)$
is an $A$-invariant measure on $A$. Hence, we can choose normalizations
on the $\eta_{S}$, and therefore $\delta_{x_{S}A_{S}^{+}}$, such
that $\eta_{S,\RR}$ are the standard Lebesgue measure on $A\cong\RR$.\\

The measures we deal with in this paper are translations of the form
$g_{i}\left(\delta_{x_{\AA}A_{\AA}}\right)$ where $g_{i}\in\GL_{2}^{1}\left(\AA\right)$,
and we find conditions on the $g_{i}$ which imply equidistribution.

Before we continue, we note that the measure $g_{i}\left(\delta_{x_{\AA}A_{\AA}}\right)$
is supported on $x_{\AA}A_{\AA}g_{i}^{-1}$. This problematic ``left
to inverse right'' notation is confusing, so instead we will always
translate with inverses. So for example $\left(ag\right)^{-1}\left(\delta_{x_{\AA}A_{\AA}}\right)$
is supported on $x_{\AA}A_{\AA}ag=x_{\AA}A_{\AA}g$ for $a\in A_{\AA}$.
In particular, this $A_{\AA}$-invariance of $\delta_{x_{\AA}A_{\AA}}$
implies that multiplying the $g_{i}$ from the left by elements from
$A_{\AA}$ doesn't change the limit. Multiplying $g_{i}$ from the
right by a sequence from a compact set can be taken care of by using
\corref{measure_action} which leads to the following.
\begin{lem}
Let $g_{i}\in G_{\AA}$, $a_{i}\in A_{\AA}$ and $k_{i}\in K_{\AA}\subseteq A_{\AA}$
where $K_{\AA}$ is a fixed compact set. The sequence $g_{i}^{-1}\left[\delta_{x_{\AA}A_{\AA}}\right]$
equidistributes if and only if $\left(a_{i}g_{i}k_{i}\right)^{-1}\left[\delta_{x_{\AA}A_{\AA}}\right]$
equidistributes.
\end{lem}

The first immediate observation, is that if $g_{i}\in A_{\AA}K_{\AA}$
for some fixed compact set $K_{\AA}$, then $g_{i}^{-1}\left(\delta_{x_{\AA}A_{\AA}}\right)$
cannot equidistribute. Hence a necessary condition for equidistribution
is that $A_{\AA}g_{i}$ diverges in $A_{\AA}\backslash G_{\AA}$.

In general, the last lemma suggest that we should use the Iwasawa
decomposition $ANK$, and in dimension 2 we have a very simple decomposition
for $\GL_{2}\left(\AA\right)$ based on the Chinese remainder theorem.

Recall that for $q\in\Gamma_{\AA}=\GL_{2}\left(\QQ\right)$ we write
$q^{\left(f\right)}=\left(q,q,q,...\right)\in{\displaystyle \prod_{p\in\PP}}'\GL_{2}\left(\QQ_{p}\right)$.
\begin{lem}
\label{lem:Iwasawa_reduction}Let $K$ be the compact set 
\[
K=\left(O_{2}\left(\RR\right)\cdot\left\{ u_{t}\;\mid\;\left|t\right|\leq1\right\} \right)\times\prod_{p\in\PP}\GL_{2}\left(\ZZ_{p}\right)\subseteq\GL_{2}^{1}\left(\AA\right)
\]
Then for every $g\in G_{\AA}$ there are some $m\in\NN_{\geq1}$ and
$n\in\NN_{\geq0}$ such that $\left(u_{n},u_{1/m}^{\left(f\right)}\right)\in A_{\AA}\cdot g\cdot K$.
\end{lem}

\begin{proof}
We already have the standard Iwasawa decompositions $\GL_{2}\left(\RR\right)=A_{\infty}U\mathrm{O}_{2}\left(\RR\right)$
and \\
$\GL_{2}\left(\QQ_{p}\right)=A_{p}U_{p}\GL_{2}\left(\ZZ_{p}\right)$
where $U_{p}$ are upper triangular unipotent in $\GL_{2}\left(\QQ_{p}\right)$.
This means that if $g\in\GL_{2}^{1}\left(\AA\right)$, then we can
always multiply it from the left and right with elements from $A_{\AA}$
and $K$ respectively so that we are left with upper triangular unipotent
matrices $u_{\alpha_{p}},\;\alpha_{p}\in\QQ_{p}$ in each prime place
and $u_{x},\;x\in\RR$ in the real place. Since $diag\left(1,-1\right)\in A_{\infty}\cap O_{2}\left(\RR\right)$,
then we can conjugate $u_{x}$ by it to get $u_{-x}$, so we may assume
that $x\geq0$. Moreover, by multiplying further by $u_{\flr x-x}$
we may assume that $x=n$ is a nonnegative integer.

As for the prime places, by defintion in most prime places $g^{\left(p\right)}\in\GL_{2}\left(\ZZ_{p}\right)$,
so after the decomposition above we may assume that $\left|\alpha_{p}\right|>1$
for finitely many $p$, and for the rest $\alpha_{p}=1$.

With this assumption, $m={\displaystyle \prod_{p}}\left|\alpha_{p}\right|_{p}$
is wel defined. Moreover $\frac{1}{m}=\frac{1}{m\alpha_{p}}\cdot\alpha_{p}$
where $\left|m\alpha_{p}\right|_{p}=1$ for all $p$ and therefore
$m\alpha_{p}\in\ZZ_{p}^{\times}$. Finally, since 
\[
u_{1/m}=diag\left(1,m\alpha_{p}\right)u_{\alpha_{p}}diag\left(1,m\alpha_{p}\right)^{-1}
\]
and $diag\left(1,m\alpha_{p}\right)\in A_{p}\cap\GL_{2}\left(\ZZ_{p}\right)$,
we see that we can change $u_{\alpha_{p}}$ to simply $u_{1/m}$,
and this finishes the proof.
\end{proof}
With this last lemma in mind, we can assume that our translation is
by $\left(u_{n},u_{1/m}^{\left(f\right)}\right)$ for some $n,m\in\NN$
with $m\geq1$. It is also easy to see that $A_{\AA}\left(u_{n_{i}},u_{1/m_{i}}^{\left(f\right)}\right)$
diverges in $A_{\AA}\backslash G_{\AA}$ if and only if $n_{i}\to\infty$
or $m_{i}\to\infty$. If either the $n_{i}$ or $m_{i}$ are bounded,
we can change them with any other elements in some bounded set for
our equidistribution result, or in the $u_{1/m}^{\left(f\right)}$
case change it to the identity. In particular we may assume that $n_{i}\to\infty$
or $n_{i}=0$ for each $i$, which as we shall see lead us to $U$
or $A$-invariance needed in \ref{thm:main_lifting}.
\begin{assumption}
\label{assu:unipotent} The elements $g_{i}=\left(u_{n_{i}},u_{1/m_{i}}^{\left(f\right)}\right)$
are such that $n_{i},\;m_{i}\in\NN$, $m_{i}\geq1$ and $m_{i}\to\infty$
or $n_{i}\to\infty$. If $n_{i}\not\to\infty$ , then $n_{i}=0$ for
all $i$.
\end{assumption}

\begin{lem}
\label{lem:invariance_condition}Let $g_{i}=\left(u_{n_{i}},u_{1/m_{i}}^{\left(f\right)}\right)\in G_{\AA}$
as in \ref{assu:unipotent}. If $\left[g_{i}^{-1}\delta_{x_{\AA}A_{\AA}}\right]\to\left[\mu\right]$
for some probability measure $\mu$, then $\mu$ is either $A$-invariant
(if $n_{i}=0$), or $U$-invariant (if $n_{i}\to\infty$).
\end{lem}

\begin{proof}
The measure $\delta_{x_{\AA}A_{\AA}}$ is $A_{\AA}$-invariant, hence
$g_{i}^{-1}\delta_{x_{\AA}A_{\AA}}$ is $g_{i}^{-1}A_{\AA}g_{i}$
-invariant. Clearly, if \\
$x_{i}^{\left(\infty\right)}=Id$ for all $i$, then $g_{i}^{-1}\delta_{x_{\AA}A_{\AA}}$
are all $A$-invariant and so is their limit $\mu$. 

The $x_{i}\to\infty$ case is solved using a standard shearing argument.
In this case $g_{i}^{-1}\delta_{x_{\AA}A_{\AA}}$ is invariant under
\[
u_{-x_{i}}\left(\begin{array}{cc}
e^{-t/2} & 0\\
0 & e^{t/2}
\end{array}\right)u_{x_{i}}=\left(\begin{array}{cc}
e^{-t/2} & 0\\
0 & e^{t/2}
\end{array}\right)u_{-e^{t}x_{i}}u_{x_{i}}=\left(\begin{array}{cc}
e^{-t/2} & 0\\
0 & e^{t/2}
\end{array}\right)u_{x_{i}\left(1-e^{t}\right)}.
\]
Fixing some $C\in\RR$, we can choose $t_{i}$ such that $x_{i}\left(1-e^{t_{i}}\right)=C$,
and because $x_{i}\to\infty$ we get that $t_{i}\to0$. Thus, the
limit of $u_{-x_{i}}a\left(t_{i}\right)u_{x_{i}}\to u_{C}$, so by
\corref{measure_action} the limit measure $\left[\mu\right]$ is
$u_{C}$-invariant. Since $C$ was arbitrary we get that $\left[\mu\right]$
is $U$-invariant. 
\end{proof}
\begin{rem}
Note that we can write $\frac{1}{m}$ as $\prod_{p\in\PP}\left|m\right|_{p}$.
We use the first notation because it is simpler, however the second
notation is in a sense more accurate. Indeed, we study the behavior
of the translation by $u_{1/m}$ over the prime places, and it is
controlled by the $p$-adic norm. One more interesting observation,
is that while $\frac{1}{m}$ is not defined when $m=0$, the product
$\prod_{p\in\PP}\left|m\right|_{p}$ is zero, in which case 
\[
u_{\prod_{p\in\PP}\left|m\right|_{p}}=Id.
\]
In the same way, we could write $\left|n_{i}\right|_{\infty}$ instead
of $n_{i},\;n_{i}\geq0$ which will make our presentation uniform. 
\end{rem}

\subsection{\label{subsec:Uniform_prime_invariance}Uniform invariance over the
prime places }

The next step is to prove the prime invariance condition in \ref{thm:main_lifting},
namely, we want to show that if $\mu$ is a limit probability measure
of our translations, then for any $S\subseteq\PP_{\infty}$ finite,
the pushforward ${\displaystyle \det_{S\backslash\left\{ \infty\right\} }}\left(\pi_{S}^{\AA}\left(\mu\right)\right)$
to ${\displaystyle \prod_{p\in S\backslash\left\{ \infty\right\} }}\ZZ_{p}^{\times}$
is the Haar measure.

The determinant of a matrix is (almost) determined by a multiplication
by a diagonal matrix. Since in our measures we start with a diagonal
orbit measure which spend equal amount of time in each determinant,
and we translate it by unipotent matrices which have determinant 1,
we expect this condition to be automatically true. In particular,
we might want to try and already push each one of these measure down
to ${\displaystyle \prod_{p\in S\backslash\left\{ \infty\right\} }}\ZZ_{p}^{\times}$
and show that each of them is uniform there. However, these measure
are not finite, so the projection to the compact space ${\displaystyle \prod_{p\in S\backslash\left\{ \infty\right\} }}\ZZ_{p}^{\times}$
will give us nothing. What we will do instead is first find a way
to change our locally finite measures into finite measures and then
apply the argument above. \\

Our diagonal group is $A_{\AA}$, but as we saw before $x_{\AA}A_{\AA}=x_{\AA}A_{\AA}^{+}$,
and $A_{\AA}^{+}=A\times\prod_{p}A_{p}^{+}$. For each prime $p$,
the group $A_{p}^{+}\cong\left(\ZZ_{p}^{\times}\right)^{2}$ is compact,
so the only part that makes our measure infinite is $A$. Let us show
that for each translation, there is a compact subset of $A$ such
that outside of it our translated orbit goes quickly to the cusp.

To do that, we first want to present our translated orbit as element
in $X_{\AA}=\SL_{2}\left(\ZZ\right)\backslash H_{\AA}$ where by our
definition 
\[
H_{\AA}=\SL_{2}\left(\RR\right)\times\prod_{p}\GL_{2}\left(\RR\right).
\]

This means that for any $a\in A_{\AA}^{+}$ we want to find $\gamma\in\Gamma_{\AA}$
such that $\gamma a\left(u_{n},u_{1/m}^{\left(f\right)}\right)\in H_{\AA}$.
Considering a single prime place, we have
\[
\gamma^{\left(p\right)}\left(\begin{array}{cc}
\alpha & 0\\
0 & \beta
\end{array}\right)\left(\begin{array}{cc}
1 & 1/m\\
0 & 1
\end{array}\right)=\gamma^{\left(p\right)}\left(\begin{array}{cc}
1 & \alpha\beta^{-1}/m\\
0 & 1
\end{array}\right)\left(\begin{array}{cc}
\alpha & 0\\
0 & \beta
\end{array}\right),
\]
where $\alpha,\beta\in\ZZ_{p}^{\times}$. Since $\left(\begin{smallmatrix}\alpha & 0\\
0 & \beta
\end{smallmatrix}\right)\in\GL_{2}\left(\ZZ_{p}\right)$, for the product above to be in $\GL_{2}\left(\ZZ_{p}\right)$, the
element $\gamma^{\left(p\right)}$ should also be of the form $u_{q}$
where $q\in\QQ$ and $q+\frac{\alpha/\beta}{m}\in\ZZ_{p}$. This $q$
should solve this problem for each prime and for that we can use the
Chinese remainder theorem. With this in mind, we use the following
definition.
\begin{defn}
For each prime $p$ define $\xi_{p}:A_{p}^{+}\to\ZZ_{p}^{\times}$
by $\xi_{p}\left(diag\left(\alpha,\beta\right)\right)=\alpha\beta^{-1}$.
For $m=\prod p_{i}^{k_{i}}$ define
\[
\psi_{n}:A_{\AA}^{+}\to\prod_{i}A_{p_{i}}^{+}\overset{\prod\xi_{p}}{\longrightarrow}\prod_{i}\ZZ_{p_{i}}^{\times}\to\prod_{i}\inv{p_{i}^{k_{i}}}\overset{CRT}{\longrightarrow}\inv m.
\]
\end{defn}

\begin{lem}
\label{lem:gamma_fix}Let $a\in A_{\AA}^{+}$, $m,n\in\NN$ and let
$\ell\in\left\{ 1,...,m-1\right\} $ such that $\psi_{m}\left(a\right)\equiv_{m}\ell$.
Then
\[
\left(u_{-\ell/m},u_{-\ell/m}^{\left(f\right)}\right)a\cdot\left(u_{n},u_{1/m}\right)\in H_{\AA}.
\]
\end{lem}

\begin{proof}
For any prime $p$ write $a^{\left(p\right)}=\left(\begin{array}{cc}
\alpha^{\left(p\right)} & 0\\
0 & \beta^{\left(p\right)}
\end{array}\right)$. If $m=\prod p_{i}^{k_{i}}$, then by the Chinese remainder correspondence,
we have that $\ell=\psi_{m}\left(a\right)\equiv_{p_{i}^{k_{i}}}\frac{\alpha^{\left(p_{i}\right)}}{\beta^{\left(p_{i}\right)}}$.
It then follows that 
\[
\frac{\left(\frac{\alpha^{\left(p_{i}\right)}}{\beta^{\left(p_{i}\right)}}-\ell\right)}{m}\in\ZZ_{p_{i}}^{\times},
\]
which means that $u_{-\ell/m}a^{\left(p_{i}\right)}u_{1/m}\in\GL_{2}\left(\ZZ_{p_{i}}\right)$.
As this is true for any prime $p$, we get that $\left(u_{-\ell/m},u_{-\ell/m}^{\left(f\right)}\right)a\cdot\left(u_{n},u_{1/m}\right)\in H_{\AA}.$
\end{proof}
Now that we know how to present our translated orbits in $\SL_{2}\left(\ZZ\right)\backslash H_{\AA}$,
we can ask which part is close to the cusp, and therefore doesn't
contribute too much to the integration. In the presentation in $H_{\AA}=\SL_{2}\left(\RR\right)\times\prod_{p}\GL_{2}\left(\RR\right)$
the only noncompact part is the real place $\SL_{2}\left(\RR\right)$,
so whether a part of the orbit is close to the cusp is mainly determined
by the real entry $a^{\left(\infty\right)}=a\left(t\right)$ of the
diagonal matrix. What we will do is restrict $t$ to the part of the
translated orbit ``before'' it diverges to the cusp.
\begin{defn}
For a segment $I\subseteq\RR$ set $A_{\AA}^{I}$ to be the set $\left\{ \left(a\left(t\right),a^{\left(f\right)}\right)\in A_{\AA}^{+}\;\mid\;t\in I\right\} $,
and denote by $\delta_{x_{\AA}A_{\AA}^{I}}$ to be the restriction
of the orbit measure $\delta_{x_{\AA}A_{\AA}}$ to $x_{\AA}A_{\AA}$.
Note that by our choice of normalization, if $I$ is finite, then
$\left|I\right|^{-1}\delta_{x_{\AA}A_{\AA}^{I}}$ is a probability
measure.
\end{defn}

\begin{lem}
\label{lem:restricted_measures}Let $g_{i}=\left(u_{n_{i}},u_{1/m_{i}}^{\left(f\right)}\right)\in G_{\AA}$
as in \ref{assu:unipotent} and set $T_{i}=\ln\left(\max\left\{ 1,n_{i}\right\} \cdot m_{i}\right)$.
If $\frac{1}{T_{i}}g_{i}^{-1}\left(\delta_{x_{\AA}A_{\AA}^{\left[0,T_{i}\right]}}\right)$
equidistribute, then $g_{i}^{-1}\left[\delta_{x_{\AA}A_{\AA}}\right]$
equidistribute.
\end{lem}

\begin{proof}
We will prove this lemma in two steps. First we will use a symmetry
argument to get rid of half of the $A_{\AA}^{+}$-orbit, and then
use Mahler criterion to show that most of the remaining orbit is near
the cusp and therefore doesn't contribute anything. The intuition
behind the ideas here were given in \ref{subsec:Symmetries_horospheres}
and \ref{subsec:Shearing}.\\

Recall that our symmetry was switching between the $x$ and $y$ coordinates,
which is multiplying by the matrix $\tau=\tau^{-1}=\left(\begin{array}{cc}
0 & 1\\
1 & 0
\end{array}\right)$, which in our context is also a matrix in $\Gamma_{\AA}=\GL_{2}\left(\QQ\right)$. 

Before we consider the translation, let us consider this symmetry
on the orbit $x_{\AA}A_{\AA}^{+}$. The first observation, is that
in each prime place we integrate over all the diagonal matrices in
$\GL_{2}\left(\ZZ_{p}\right)$, and conjugating by $\tau^{\left(f\right)}$
just switch the entries on the diagonal, which doesn't change our
measure. In the real place we integrate over $a\left(t\right)$ and
$\tau a\left(t\right)\tau=a\left(-t\right)$ so together we get that
for each $t\in\RR$ we have
\[
\int_{A_{\AA}^{+\left(f\right)}}\delta_{x_{\AA}\left(a\left(t\right),a^{\left(f\right)}\right)\left(\tau,\tau^{\left(f\right)}\right)}\da^{\left(f\right)}=\int_{A_{\AA}^{+\left(f\right)}}\delta_{x_{\AA}\left(\tau,\tau^{\left(f\right)}\right)\left(a\left(-t\right),a^{\left(f\right)}\right)}\da^{\left(f\right)}=\int_{A_{\AA}^{+\left(f\right)}}\delta_{x_{\AA}\left(a\left(-t\right),a^{\left(f\right)}\right)}\da^{\left(f\right)}.
\]
Integrating over $t\in\left[0,\infty\right]$ on both sides, we get
that $\left(\tau,\tau^{\left(f\right)}\right)\delta_{x_{\AA}A_{\AA}^{\left[0,\infty\right]}}=\delta_{x_{\AA}A_{\AA}^{\left[-\infty,0\right]}}$
, namely the two ``halves'' of the orbits are mirror images of one
another.\\

We want to have a similar result for our translation, namely $g_{i}^{-1}\left[\delta_{x_{\AA}A_{\AA}^{\left[0,\infty\right]}}\right]=\left(g_{i}k_{i}\right)^{-1}\left[\delta_{x_{\AA}A_{\AA}^{\left[-\infty,0\right]}}\right]$
for some bounded sequence $k_{i}$. This is not true for $g_{i}=\left(u_{n_{i}},u_{1/m_{i}}^{\left(f\right)}\right)$
but we can fix it once we choose the right point $t_{i}\in\RR$ around
which there is a symmetry, or more formally translate with $\left(a\left(t_{i}\right),Id\right)g_{i}$
instead of $g_{i}$.

For $g_{i},\tilde{g}_{i}\in G_{\AA}$ we will write $g_{i}\approx\tilde{g}_{i}$
if $g_{i}^{-1}\tilde{g}_{i}$ are all contained in a compact set.
This implies that translation by $g_{i}$ equidistribute if and only
if translations by $\tilde{g}_{i}$ equidistribute.

First, for the symmetry argument in the real place consider the hyperbolic
matrix 
\begin{align*}
h\left(y\right) & :=\left(\begin{array}{cc}
\cosh\left(y/2\right) & \sinh\left(y/2\right)\\
\sinh\left(y/2\right) & \cosh\left(y/2\right)
\end{array}\right)\\
 & =\left(\begin{array}{cc}
\cosh^{-1/2}\left(y\right) & 0\\
0 & \cosh^{1/2}\left(y\right)
\end{array}\right)\left(\begin{array}{cc}
1 & \sinh\left(y\right)\\
0 & 1
\end{array}\right)\left[\overbrace{\left(\cosh\left(y\right)\right)^{-1/2}\left(\begin{array}{cc}
\cosh\left(y/2\right) & -\sinh\left(y/2\right)\\
\sinh\left(y/2\right) & \cosh\left(y/2\right)
\end{array}\right)}^{k\left(y\right)}\right]\\
 & =a\left(\ln\left(\cosh\left(y\right)\right)\right)u_{\cosh\left(y\right)}u_{\sinh\left(y\right)-\cosh\left(y\right)}k\left(y\right).
\end{align*}
Note that $\sinh\left(y\right)-\cosh\left(y\right)=-e^{-y}$ which
is uniformly bounded for $y\geq0$. Assuming that $n_{i}\to\infty$
we can set $y_{i}=\cosh^{-1}\left(n_{i}\right)$ to get that $h\left(y_{i}\right)\approx a\left(\ln\left(n_{i}\right)\right)u_{n_{i}}$.
Moreover, since $h\left(y\right)=\tau h\left(y\right)\tau$, we get
that 
\[
a\left(\ln\left(n_{i}\right)\right)u_{n_{i}}\approx\tau a\left(\ln\left(n_{i}\right)\right)u_{n_{i}}.
\]
In other words, the symmetry coming from the real place is going to
be around the time $\ln\left(n_{i}\right)$. To add the $n_{i}=0$
case we can write instead
\[
a\left(\ln\left(\max\left\{ 1,n_{i}\right\} \right)\right)u_{n_{i}}\approx\tau a\left(\ln\left(\max\left\{ 1,n_{i}\right\} \right)\right)u_{n_{i}}.
\]
\\
Similarly, instead of translating by $u_{1/m}$ in the prime places,
we translate instead by 
\[
v_{m}=\left(\begin{array}{cc}
\frac{m}{\flr{m^{1/2}}} & 0\\
0 & \frac{1}{\flr{m^{1/2}}}
\end{array}\right)u_{1/m}=\flr{m^{1/2}}^{-1}\left(\begin{array}{cc}
m & 1\\
0 & 1
\end{array}\right).
\]
We divide by $\flr{m^{1/2}}$ and not $m^{1/2}$ since later we will
move it to the $\SL_{2}\left(\RR\right)$ part of $H_{\AA}$. The
lattice $v_{m}\ZZ^{2}$ is invariant under $\tau$, or more specifically
$\tau v_{m}=v_{m}\ss{\left(\begin{array}{cc}
-1 & 0\\
m & 1
\end{array}\right)}\in v_{m}\GL_{2}\left(\ZZ\right)$. In the prime places we have that $\ss{\left(\begin{array}{cc}
-1 & 0\\
m & 1
\end{array}\right)}^{\left(f\right)}\in\prod_{p}\GL_{2}\left(\ZZ_{p},\right)$ which is compact, implying that $\tau v_{m}\approx v_{m}$. Together
we get that 
\[
\left(a\left(\max\left\{ 1,n_{i}\right\} \right)u_{n_{i}},v_{m_{i}}^{\left(f\right)}\right)\approx\left(\tau,\tau^{\left(f\right)}\right)\left(a\left(\max\left\{ 1,n_{i}\right\} \right)u_{n_{i}},v_{m_{i}}^{\left(f\right)}\right),
\]
so that translations by $\left(a\left(\ln\left(\max\left\{ 1,n_{i}\right\} \right)\right)u_{n_{i}},v_{m_{i}}^{\left(f\right)}\right)^{-1}$
of $\delta_{x_{\AA}A_{\AA}^{\left[0,\infty\right]}}$ equidistribute
if and only if the translation by $\delta_{x_{\AA}A_{\AA}^{\left[-\infty,0\right]}}$
equidsitribute.

For $I=\left[-\infty,0\right]$ or $\left[0,\infty\right]$ we get
that 
\[
\left(a\left(\ln\left(\max\left\{ 1,n_{i}\right\} \right)\right)u_{n_{i}},v_{m_{i}}^{\left(f\right)}\right)^{-1}\delta_{x_{\AA}A_{\AA}^{I}}=\left(\left(a\left(\ln\left(m_{i}\cdot\max\left\{ 1,n_{i}\right\} \right)\right)\cdot\frac{\flr{m^{1/2}}}{m^{1/2}},Id\right)g_{i}\right)^{-1}\delta_{x_{\AA}A_{\AA}^{I}}.
\]
The part $\frac{\flr{m^{1/2}}}{m^{1/2}}$ is a scalar that is always
in $\left[\frac{1}{2},1\right]$ so we can put it inside the compact
set. Hence, we see that the center of our symmetry is around $T_{i}=\ln\left(m_{i}\cdot\max\left\{ 1,n_{i}\right\} \right)$,
or equivalently $g_{i}^{-1}\left[\delta_{x_{\AA}A_{\AA}^{\left[-\infty,T_{i}\right]}}\right]$
equidistribute if and only if $g_{i}^{-1}\left[\delta_{x_{\AA}A_{\AA}^{\left[T_{i},\infty\right]}}\right]$
equidistribute.\\

We are now left with only half of the orbit, and next we want to cut
it even more and leave just a finite segment by removing the part
which is too close to the cusp. Using Mahler's criterion from \lemref{generalized_mahler}
for $x=x_{\AA}\cdot h\in X_{\AA}$ with $h\in H_{\AA}$, the height
is defined by 
\[
ht(x):={\displaystyle \max_{0\neq v\in\ZZ^{2}}}\norm{vh_{\infty}}_{\infty}^{-1}.
\]

While computing the height can be quite difficult, in order to show
that the height is large, it is enough to find one vector ``witness''
which have a small norm. Here we will use the vector $v=e_{2}=\left(0,1\right)$
as our witness. Using \lemref{gamma_fix} for the presentation in
$H_{\AA}$, we have 
\begin{align*}
e_{2}\left(u_{-\ell/m}a\left(t\right)u_{n}\right) & =\left(0,e^{t/2}\right)
\end{align*}
so that 
\begin{align*}
ht\left(u_{-\ell/m}a\left(t\right)u_{n}\right) & \geq e^{-t/2}
\end{align*}

In particular, for any $C>0$ if $t\leq-2\ln\left(C\right)$, then
the height is at least $C$.

If $f\in C_{c}\left(X_{\AA}\right)$, then we can bound its support
in a set of the form $M_{C}=\left\{ x\in X_{\AA}\;\mid\;ht\left(x\right)\leq C\right\} $
for some $C\geq1$. We then have that 
\[
\frac{1}{T_{i}}\left|g_{i}^{-1}\delta_{x_{\AA}A_{\AA}^{(-\infty,T_{i}]}}\left(f\right)-g_{i}^{-1}\delta_{x_{\AA}A_{\AA}^{\left[0,T_{i}\right]}}\left(f\right)\right|=\frac{1}{T_{i}}\left|g_{i}^{-1}\delta_{x_{\AA}A_{\AA}^{(-\infty,0]}}\left(f\right)\right|\leq\frac{\norm f_{\infty}\cdot2\ln\left(C\right)}{T_{i}}.
\]
Since we fixed $f$, which in turn fix $C$, and because $T_{i}\to\infty$
the upper bound goes to zero. As this is true for all $f\in C_{c}\left(X_{\AA}\right)$,
we get that 
\[
\limfi i{\infty}\frac{1}{T_{i}}g_{i}^{-1}\delta_{x_{\AA}A_{\AA}}=\limfi i{\infty}\frac{1}{T_{i}}g_{i}^{-1}\delta_{x_{\AA}A_{\AA}^{\left[0,T_{i}\right]}}=\mu_{Haar,\AA},
\]
which is what we wanted to prove.
\end{proof}
Now we can shift our attention to the probability measures from the
last lemma, and to show that it satisfies the prime uniformity condition
that we need in \ref{thm:main_lifting}.
\begin{lem}
\label{lem:prime_uniformity}Let $\mu=\left|I\right|^{-1}\left(u_{n},u_{1/m}\right)^{-1}\left(\delta_{x_{\AA}A_{\AA}^{I}}\right)$
for some finite segment $I\subseteq\RR$. Then for any $S\subseteq\PP_{\infty}$
finite, the pushforward $\eta={\displaystyle \det_{S\backslash\left\{ \infty\right\} }}\left(\pi_{S}^{\AA}\left(\mu\right)\right)$
to ${\displaystyle \prod_{p\in S\backslash\left\{ \infty\right\} }}\ZZ_{p}^{\times}$
is the Haar measure.
\end{lem}

\begin{proof}
The main idea here is that while in $A_{\AA}^{I}$ we restrict the
elements in the real place, we still have the entire group in the
prime place. Recall that the determinant ${\displaystyle \det_{S\backslash\left\{ \infty\right\} }}$
is define only in the presentation $\SL_{2}\left(\ZZ\right)\backslash H_{\AA}$.
However, by \ref{lem:gamma_fix} we know that for any $a\in A_{\AA}^{+}$
there is some $\ell=\psi_{m}\left(a\right)$ such that 
\[
x_{\AA}a\left(u_{n},u_{1/m}\right)=x_{\AA}\overbrace{\left(u_{-\ell/m},u_{-\ell/m}^{\left(f\right)}\right)a\cdot\left(u_{n},u_{1/m}\right)}^{\in H_{\AA}}.
\]
Since unipotent matrices (on the right side) have determinant $1$,
we see that our translation doesn't change the determinant, and in
other words
\[
\det_{S\backslash\left\{ \infty\right\} }\left(\mu\right)=\det_{S\backslash\left\{ \infty\right\} }\left(\left|I\right|^{-1}\left(\delta_{x_{\AA}A_{\AA}^{I}}\right)\right).
\]
Since $\delta_{x_{\AA}A_{\AA}^{+}}$ is $A_{p}^{+}\cong\left(\ZZ_{p}^{\times}\right)^{2}$-invariant
for every $p$, it follows that the determinant is $\ZZ_{p}^{\times}$-invariant
in every $p$, which is what we wanted to show.
\end{proof}
We can now put all the results together to show that for the equidistribution
of translated orbit, we only need to show equidistribution in the
real place.
\begin{thm}
Let $\left(u_{n_{i}},u_{1/m_{i}}^{\left(f\right)}\right)\in G_{\AA}$
as in \ref{assu:unipotent} and set $T_{i}=\ln\left(m_{i}\max\left\{ 1,n_{i}\right\} \right)$.
Assume that for each such $m_{i},n_{i}$ with $T_{i}\to\infty$ the
measures $\pi_{\RR}^{\AA}\left(\left|T_{i}\right|^{-1}\left(u_{n_{i}},u_{1/m_{i}}\right)^{-1}\left(\delta_{x_{\AA}A_{\AA}^{\left[0,T_{i}\right]}}\right)\right)$
equidistribute in $X_{\RR}$. Then for any $g_{i}\in G_{S}$ such
that $A_{\AA}g_{i}$ diverge in $A_{\AA}\backslash G_{\AA}$ we have
that $g_{i}^{-1}\left[\delta_{x_{\AA}A_{\AA}}\right]$ equidistribute
in $X_{\AA}$.
\end{thm}

\begin{proof}
First, using Iwasawa decomposition from \ref{lem:Iwasawa_reduction}
we may assume that $g_{i}=\left(u_{n_{i}},u_{1/m_{i}}^{\left(f\right)}\right)$.
Let $\mu_{i}=\left|T_{i}\right|^{-1}g_{i}^{-1}\left(\delta_{x_{\AA}A_{\AA}^{\left[0,T_{i}\right]}}\right)$
be the probability measure restrictions of our translated orbits.
If $\mu_{i}\wstar\mu$, then by \ref{lem:restricted_measures} we
see that our original locally finite measures $g_{i}^{-1}\left[\delta_{x_{\AA}A_{\AA}}\right]$
converge to $\left[\mu\right]$ as well. It is now enough to show
that every convergent subsequence of $\mu_{i}$ converge to the Haar
measure, so let us assume that $\mu_{i}$ converge.

First, since by assumption $\pi_{\RR}^{\AA}\left(\mu\right)=\mu_{Haar,\RR}$,
we conclude that $\mu$ is a probability measure (there is no escape
of mass) and we also have the $\RR$-uniformity condition from our
lifting result in \ref{thm:main_lifting}. The $\RR$-invariance and
prime uniformity conditions follow from \ref{lem:invariance_condition}
and \ref{lem:prime_uniformity}, so applying \ref{thm:main_lifting}
we conclude that $\mu=\mu_{Haar,\AA}$.
\end{proof}

\newpage{}

\section{\label{sec:real_case}The real translations}

\global\long\def\norm#1{\left\Vert #1\right\Vert }%
\global\long\def\AA{\mathbb{A}}%
\global\long\def\QQ{\mathbb{Q}}%
\global\long\def\PP{\mathbb{P}}%
\global\long\def\CC{\mathbb{C}}%
\global\long\def\HH{\mathbb{H}}%
\global\long\def\ZZ{\mathbb{Z}}%
\global\long\def\NN{\mathbb{N}}%
\global\long\def\KK{\mathbb{K}}%
\global\long\def\RR{\mathbb{R}}%
\global\long\def\FF{\mathbb{F}}%
\global\long\def\oo{\mathcal{O}}%
\global\long\def\aa{\mathcal{A}}%
\global\long\def\bb{\mathcal{B}}%
\global\long\def\ff{\mathcal{F}}%
\global\long\def\mm{\mathcal{M}}%
\global\long\def\limfi#1#2{{\displaystyle \lim_{#1\to#2}}}%
\global\long\def\pp{\mathcal{P}}%
\global\long\def\qq{\mathcal{Q}}%
\global\long\def\da{\mathrm{da}}%
\global\long\def\dt{\mathrm{dt}}%
\global\long\def\dg{\mathrm{dg}}%
\global\long\def\ds{\mathrm{ds}}%
\global\long\def\dm{\mathrm{dm}}%
\global\long\def\dmu{\mathrm{d\mu}}%
\global\long\def\dx{\mathrm{dx}}%
\global\long\def\dy{\mathrm{dy}}%
\global\long\def\dz{\mathrm{dz}}%
\global\long\def\dnu{\mathrm{d\nu}}%
\global\long\def\flr#1{\left\lfloor #1\right\rfloor }%
\global\long\def\nuga{\nu_{\mathrm{Gauss}}}%
\global\long\def\diag#1{\mathrm{diag}\left(#1\right)}%
\global\long\def\bR{\mathbb{R}}%
\global\long\def\Ga{\Gamma}%
\global\long\def\PGL{\mathrm{PGL}}%
\global\long\def\GL{\mathrm{GL}}%
\global\long\def\PO{\mathrm{PO}}%
\global\long\def\SL{\mathrm{SL}}%
\global\long\def\PSL{\mathrm{PSL}}%
\global\long\def\SO{\mathrm{SO}}%
\global\long\def\mb#1{\mathrm{#1}}%
\global\long\def\wstar{\overset{w^{*}}{\longrightarrow}}%
\global\long\def\vphi{\varphi}%
\global\long\def\av#1{\left|#1\right|}%
\global\long\def\inv#1{\left(\mathbb{Z}/#1\mathbb{Z}\right)^{\times}}%
\global\long\def\cH{\mathcal{H}}%
\global\long\def\cM{\mathcal{M}}%
\global\long\def\bZ{\mathbb{Z}}%
\global\long\def\bA{\mathbb{A}}%
\global\long\def\bQ{\mathbb{Q}}%
\global\long\def\bP{\mathbb{P}}%
\global\long\def\eps{\epsilon}%
\global\long\def\on#1{\mathrm{#1}}%
\global\long\def\nuga{\nu_{\mathrm{Gauss}}}%
\global\long\def\set#1{\left\{  #1\right\}  }%
\global\long\def\smallmat#1{\begin{smallmatrix}#1\end{smallmatrix}}%
\global\long\def\len{\mathrm{len}}%
\global\long\def\idealeq{\trianglelefteqslant}%

Using the results from the previous sections, in order to show full
equidistribution, we are left to show that the projections of our
restricted measures to the real place equidistribute. 

For this section let us fix the following notation. The integers $n_{i},m_{i}$
will always used for the translation by $\left(u_{n_{i}},u_{1/m_{i}}\right)$,
and we will denote $T_{i}=\ln\left(\max\left\{ n_{i},1\right\} \cdot m_{i}\right)$.
We want to show that $\frac{1}{T_{i}}\pi_{\RR}^{\AA}\left(\left(u_{n_{i}},u_{1/m_{i}}\right)^{-1}\delta_{x_{\AA}A_{\AA}^{\left[0,T_{i}\right]}}\right)\wstar\mu_{Haar}$.
The proof where $m_{i}$ are bounded (so we may assume that $u_{1/m_{i}}=Id$)
is more or less described in \subsecref{Shearing}, and can be found
in much more details in \cite{oh_limits_2014}. The case where $n_{i}$
are bounded were originally proved in \cite{david_equidistribution_2018}
and most of the ideas were presented in \subsecref{continuous_analogue}
and \subsecref{Symmetries_horospheres}. In this section we will assume
that both $m_{i}$ and $n_{i}$ diverge to infinity. In the proof
we need to combine the equidistribution coming from the real part
(for $n_{i}\to\infty$) and from the finite prime part (for $m_{i}\to\infty$).
All of the details needed for each one of the bounded cases can be
found in the full proof, however, while we can probably write a proof
that encompass all the three parts, it seems that the notation for
it will be quite confusing. For example, we can now simply write $T_{i}=\ln\left(n_{i}m_{i}\right)$.
Thus, we restrict ourselves to the $m_{i},n_{i}\to\infty$ and leave
the bounded cases as an exercise to the reader. \\

Our first step will be to write our projections in a simpler way,
and for that we begin with the following notation.
\begin{defn}
\begin{enumerate}
\item For any finite set $\Lambda\subseteq X_{\RR}$ we write 
\[
\delta_{\Lambda}=\frac{1}{\left|\Lambda\right|}\sum_{x\in\Lambda}\delta_{x}.
\]
\item For a segment $I\subseteq\RR$ and a probability measure $\mu$ on
$X_{\RR}$ we write 
\[
\mu^{I}:=\int_{I}a\left(-t\right)\mu\dt.
\]
\end{enumerate}
\end{defn}

\begin{lem}
Let $n,m\in\NN_{>0}$ and set $T=\ln\left(m\cdot n\right)$ and $\Lambda_{m}=\left\{ x_{\RR}u_{\ell/m}\;\mid\;\left(m,\ell\right)=1,\;1\leq\ell\leq m\right\} $.
Then
\[
\pi_{\RR}^{\AA}\left(\left(u_{n},u_{1/m}^{\left(f\right)}\right)^{-1}\delta_{x_{\AA}A_{\AA}^{\left[0,T\right]}}\right)=u_{n}^{-1}\delta_{\Lambda_{m}}^{\left[0,T\right]}.
\]
\end{lem}

\begin{proof}
We begin by decomposing $A_{\AA}^{+}$ to $A_{\AA}^{\left[0,T\right]}=\bigsqcup_{\ell}\left(A_{\AA}^{\left[0,T\right]}\cap\psi^{-1}\left(\ell\right)\right)$
where $\psi_{m}$ is the function from \ref{lem:gamma_fix} where
each part has exactly $\frac{1}{\varphi\left(m\right)}$ mass. If
$a\in A_{\AA}^{\left[0,T\right]}\cap\psi^{-1}\left(\ell\right)$,
then by \ref{lem:gamma_fix} we have that 
\[
\pi_{\RR}^{\AA}\left(x_{\AA}a\left(u_{n},u_{1/m}\right)\right)=\pi_{\RR}^{\AA}\left(x_{\AA}\overbrace{\left(u_{-\ell/m},u_{-\ell/m}^{\left(f\right)}\right)a\left(u_{n},u_{1/m}^{\left(f\right)}\right)}^{\in H_{\AA}}\right)=x_{\RR}u_{-\ell/m}a^{\left(\infty\right)}u_{n}.
\]
Thus, this decomposition shows that the integral on $x_{\AA}\left(A_{\AA}^{\left[0,T\right]}\cap\psi^{-1}\left(\ell\right)\right)\cdot\left(u_{n},u_{1/m}\right)$
is mapped down to the integral $u_{-n}\delta_{x_{\RR}}^{\left[0,T\right]}$.
Finally, we need to average over the $\ell$ and we get the required
result.
\end{proof}
\newpage{}

As can be seen in the lemma above, the translation in the prime places
lead to the discrete average over $\Lambda_{m}$. This set will appear
in many of our computations and in many forms, and by abusing the
notation to no end we will identify this set with other sets via
\[
x_{\AA}u_{\ell/m}\sim u_{\ell/m}\sim\frac{\ell}{m}\sim\ell
\]
as elements in $\Gamma_{\RR}U\leq X_{\RR}$, $U$, $\left[0,1\right]$,
$\nicefrac{\RR}{\ZZ}$, $\ZZ$ and $\left(\nicefrac{\ZZ}{m\ZZ}\right)^{\times}$.

Also, now that we have our measures in $X_{\RR}$, we will write $\Gamma,G$
and $X$ instead for $\SL_{2}\left(\ZZ\right),\SL_{2}\left(\RR\right)$
and $X_{\RR}$. Of course, we also keep the notation of {\footnotesize{}
\begin{align*}
A & =\left\{ a\left(t\right)=\left(\begin{array}{cc}
e^{-t/2} & 0\\
0 & e^{t/2}
\end{array}\right)\;\mid\;t\in\RR\right\} ,\\
U & =\left\{ u_{h}=\left(\begin{array}{cc}
1 & h\\
0 & 1
\end{array}\right)\;\mid\;h\in\RR\right\} .
\end{align*}
}We are now left with the following problem.
\begin{problem}
Show that $u_{n_{i}}^{-1}\delta_{\Lambda_{m_{i}}}^{\left[0,T_{i}\right]}$
equidistribute in $X$ whenever $n_{i},m_{i}\to\infty$.
\end{problem}

The main idea here, that was already presented in \subsecref{Shearing},
is to take small segments from $\left[0,T_{i}\right]$ and to show
that if they are chosen suitably, then they approximate expanding
horocycles, and the closer we are to $T_{i}$ the more expanded these
horocycles are. The translation in $u_{n}^{-1}$ will be dealt with
the standard shearing technique, while the average over $\Lambda_{m}$
we be dealt with a result about equidistribution of relatively prime
numbers.

We begin the equidistribution of the relatively prime numbers. In
every consecutive $m$ integers, exactly $\varphi\left(m\right)$
of which are coprime to $m$. The next result shows that up to a small
error, this is true for any segment in $\RR$.
\begin{lem}
\label{lem:equi_coprime}Fix some $m\in\NN$. Denote by $\omega\left(m\right)$
be the number of distinct prime divisors of $m$ and set $\Lambda_{m}^{\star}=\left\{ \ell\in\ZZ\;\mid\;\left(\ell,m\right)=1\right\} $.
Then for any finite interval $I\subseteq\RR$ we have that 
\[
\left|\left|\Lambda_{m}^{\star}\cap I\right|-\frac{\varphi\left(m\right)}{m}\left|I\right|\right|\leq2^{\omega\left(m\right)}.
\]
\end{lem}

\begin{proof}
To count the number of points in $\Lambda_{m}^{\star}\cap I$ we use
the inclusion exclusion principle. For each $P\mid m$ let $U_{P}=P\ZZ\cap I$
so that $\left|\left|U_{P}\right|-\frac{\left|I\right|}{P}\right|\leq1$
and note that $U_{P}={\displaystyle \bigcap_{p\mid P\;prime}U_{p}}$.
Applying the inclusion exclusion principle, we obtain that 
\[
\left|\Lambda_{m}^{\star}\cap I\right|=\left|U_{1}\backslash\bigcup_{\text{prime }p\mid m}U_{p}\right|=\sum_{\substack{p\mid m\\
p\ prime
}
}\mu\left(P\right)\left|U_{P}\right|,
\]
where $\mu\left(P\right)$ is the M{\"o}bius function. On the other
hand 
\[
\sum_{P\mid m}\mu\left(P\right)\frac{\left|I\right|}{P}=\left|I\right|\sum_{P\mid m\;SQF}\prod_{\substack{p\mid P\\
p\;prime
}
}\frac{-1}{p}=\left|I\right|\prod_{\substack{p\mid m\\
p\;prime
}
}\left(1-\frac{1}{p}\right)=\left|I\right|\frac{\varphi\left(m\right)}{m}.
\]
We conclude that 
\[
\left|\left|\Lambda_{m}^{\star}\cap I\right|-\left|I\right|\frac{\varphi\left(m\right)}{m}\right|\leq\sum_{P\mid m\;SQF}\left|\left|U_{P}\right|-\frac{\left|I\right|}{P}\right|\leq2^{\omega\left(P\right)}.
\]
\end{proof}
It is known (see \cite{robin_estimation_1983}) that the function
$\omega\left(m\right)$ is in $O\left(\frac{\ln\left(m\right)}{\ln\left(\ln\left(m\right)\right)}\right)$
so that $2^{\omega\left(m\right)}$ is bounded from above by $m^{c/\ln\left(\ln\left(m\right)\right)}$
for some constant $c>0$. In particular, for a given $\varepsilon>0$
and all $m$ large enough it is smaller than $m^{\varepsilon}$. Given
this upper bound, it is an exercise to show that $\frac{\ln\left(\varphi\left(m\right)\right)}{\ln\left(m\right)}\to1$
as $m\to\infty$, or equivalently for any $\varepsilon$ we have that
$\varphi\left(m\right)\geq m^{1-\varepsilon}$ for all $m$ big enough.
Hence, as long as $\left|I\right|$ is not too small, namely it is
at least $m^{\varepsilon}$, we get that for all $m$ large enough
the error is much smaller compared to $\frac{\varphi\left(m\right)}{m}\left|I\right|$
and $\left|\Lambda_{m}^{\star}\cap I\right|$.

We now want to extend this result to our expanding horocycles and
the points from $\delta_{\Lambda_{m}}$ on them. At time $t$, the
horocycle $x_{\RR}Ua\left(t\right)$ is isometric to a cycle of length
$e^{t}$. In particular, at time $t=\ln\left(m\right)$, our points
from $\delta_{\Lambda_{m}}$ ``become'' integers in the cycle of
length $m$. At that time we can use the result above, though it is
only useful if the error is much smaller than $\frac{\varphi\left(m\right)}{m}\left|I\right|$.
When considering intervals on some horocycle at time $t$, we first
need to move to the horocycle at time $\ln\left(m\right)$ to use
the result above.

For example, if we take an interval $I$ of length $\frac{1}{2}$
at time $t=\ln\left(m\right)$ where our points are just integers,
then $\left|\Lambda_{m}^{\star}\cap I\right|,\frac{\varphi\left(m\right)}{m}\left|I\right|\leq1$.
Thus, an error $2^{\omega\left(m\right)}\geq1$ already makes the
result useless. However an interval $J$ of length $\frac{1}{2}$
at time $t=0$ is expanded to an interval $I$ of length $\frac{m}{2}$
at time $t=\ln\left(m\right)$ so that $\frac{\varphi\left(m\right)}{m}\left|I\right|=\frac{\varphi\left(m\right)}{2}$.
Compared to this, $2^{\omega\left(m\right)}$ is rather small. In
general, a constant size interval $J$ at time $t$ will become an
interval of length $\left|J\right|e^{\ln\left(m\right)-t}$ at time
$\ln\left(m\right)$, so if we assume that $t<\left(1-\varepsilon\right)\ln\left(m\right)$,
then $2^{\omega\left(m\right)}$ is small compared to this expanded
interval.

If there was no translation in the real place, so that $T=\ln\left(m\right)$,
then this will be enough for the full equidistribution. However, with
the real place translation we have that $T=\ln\left(mn\right)$ with
$n$ big, so that the distance between two points in $\Lambda_{m}a\left(t\right)$
become too big for this approximation. In this case, we can use the
shearing coming from the real translation, so that instead of approximating
using a Riemann sum technique, we integrate a little bit over the
horocycle to get that the distribution function is close to uniform
1. We start with the result of integrating over the horocycle and
then do the full shearing result.
\begin{lem}
\label{lem:approx_horo}Fix some $f\in C_{c}\left(\nicefrac{\RR}{\ZZ}\right)$
and consider $\Lambda_{m}$ as rationals in $\nicefrac{\RR}{\ZZ}$.
Then for any $h,t>0$ and $m\in\NN$ we have that 
\[
\left|\frac{1}{h}\int_{0}^{h}\delta_{\Lambda_{m}-s}\left(f\right)\ds-\int_{0}^{1}f\left(s\right)\ds\right|\leq\frac{2^{\omega\left(m\right)}}{\varphi\left(m\right)h}\norm f_{\infty}
\]
\end{lem}

\begin{proof}
We start by rewriting the integral in the lemma:
\[
\int_{0}^{h}\delta_{\Lambda_{m}-s}\left(f\right)\ds=\frac{1}{\varphi\left(m\right)}\sum_{\ell\in\Lambda_{m}}\int_{0}^{h}f\left(\ell/m-s\right)\ds.
\]
Extend the function $f$ to a $\ZZ$-periodic function on $\RR$.
Then 
\begin{align*}
\frac{1}{\varphi\left(m\right)}\int_{\RR}\sum_{\ell\in\Lambda_{m}}f\left(\ell/m-s\right)\chi_{\left[0,h\right]}\left(s\right)\ds & =\frac{1}{\varphi\left(m\right)}\int_{\RR}f\left(r\right)\sum_{\ell\in\Lambda_{m}}\chi_{\left[0,h\right]}\left(\ell/m-r\right)\mathrm{dr}\\
 & =\frac{1}{\varphi\left(m\right)}\int_{\left[0,1\right]}f\left(r\right)\sum_{\substack{\ell\in\ZZ\\
\left(\ell,m\right)=1
}
}\chi_{\left[0,h\right]}\left(\ell/m-r\right)\mathrm{dr}\\
 & =\frac{1}{\varphi\left(m\right)}\int_{\left[0,1\right]}f\left(r\right)\left|\left[mr,mr+hm\right]\cap\Lambda_{m}^{\star}\right|\mathrm{dr}.
\end{align*}
Using \lemref{equi_coprime} we get that $\left|\left[mr,mr+hm\right]\cap\Lambda_{m}^{\star}\right|$
is $hm\frac{\varphi\left(m\right)}{m}$ up to a $2^{\omega\left(m\right)}$
error. In other words
\[
\left|\frac{1}{h}\int_{0}^{h}\left(u_{s}\delta_{\Lambda_{m}}\right)\left(f\right)\ds-\int_{\left[0,1\right]}f\left(r\right)\mathrm{dr}\right|\leq\frac{1}{\varphi\left(m\right)h}\norm f_{\infty}2^{\omega\left(m\right)},
\]
which is what we wanted to show.
\end{proof}
Next we show how to transform a small segment in $\left[0,T\right]$
from our measure $u_{-n}\delta_{\Lambda_{m}}^{\left[0,T\right]}$,
into an integral over expanded horocycle, which is closed to the Haar
measure.
\begin{lem}
\label{lem:small_interval-1-2}Let $f\in C_{c}\left(X\right)$ , $1>\varepsilon>0$
and $x\in\left[0,1-2\varepsilon\right]\cdot\ln\left(nm\right)$. For
all $\Delta=\Delta\left(f,\varepsilon\right)>0$ small enough and
for all $m,n$ big enough we have that 
\[
\left|\frac{1}{\Delta}\left(u_{-n}\right)_{*}\delta_{\Lambda_{m}}^{\left[x,x+\Delta\right]}\left(f\right)-\mu_{\Gamma U}\left(a\left(x\right)f\right)\right|\leq\varepsilon+4\norm f\varepsilon+\frac{1}{n^{\varepsilon}m^{\varepsilon}}\frac{\norm f_{\infty}}{\Delta}.
\]
\end{lem}

\begin{proof}
By setting $\tilde{f}=a\left(x\right)f$ we need instead to bound
\[
\left|\frac{1}{\Delta}\left(a\left(x\right)u_{-n}\right)_{*}\delta_{\Lambda_{m}}^{\left[x,x+\Delta\right]}\left(\tilde{f}\right)-\mu_{\Gamma U}\left(\tilde{f}\right)\right|=\left|\frac{1}{\Delta}\left(u_{-ne^{-x}}\right)_{*}\delta_{\Lambda_{m}}^{\left[0,\Delta\right]}\left(\tilde{f}\right)-\mu_{\Gamma U}\left(\tilde{f}\right)\right|.
\]

Setting $C=-ne^{-x}$ , we rewrite our integral as 
\begin{align*}
\left(u_{C}\right)_{*}\delta_{\Lambda_{m}}^{\left[0,\Delta\right]}\left(\tilde{f}\right) & =\frac{1}{\Delta}\int_{0}^{\Delta}\left(\left(u_{C}a\left(-t\right)\right)_{*}\delta_{\Lambda_{m}}\right)\left(\tilde{f}\right)\dt=\\
 & =\frac{1}{\Delta}\int_{0}^{\Delta}\left(\left(u_{Ce^{-t}}\right)_{*}\delta_{\Lambda_{m}}\right)\left(a\left(t\right)\tilde{f}\right)\dt.
\end{align*}
The function $f$ is uniform continuous by \lemref{uniform_continuity},
so we may assume that $\varepsilon\geq\Delta=\Delta\left(f,\varepsilon\right)>0$
is small enough so that $\left|t\right|\leq\Delta$ implies that $\norm{a\left(t\right)f-f}_{\infty}<\varepsilon$.
Because $a\left(x\right)$ commutes with $a\left(t\right)$ we also
get that $\norm{a\left(t\right)\tilde{f}-\tilde{f}}_{\infty}<\varepsilon$
and hence 
\begin{equation}
\left|\frac{1}{\Delta}\int_{0}^{\Delta}\left(\left(u_{Ce^{-t}}\right)_{*}\delta_{\Lambda_{m}}\right)\left(a\left(t\right)\tilde{f}\right)-\frac{1}{\Delta}\int_{0}^{\Delta}\left(\left(u_{Ce^{-t}}\right)_{*}\delta_{\Lambda_{m}}\right)\left(\tilde{f}\right)\right|<\varepsilon.\label{eq:11}
\end{equation}

The measure $\frac{1}{\Delta}\int_{0}^{\Delta}u_{Ce^{-t}}\delta_{\Lambda_{m}}\dt$
is supported on a single horocycle, but the integration is not the
uniform $U$-invariant measure there. However, it is a good approximation,
and in order to show it we need to (1) change the $e^{-t}$ to a linear
function, and (2) change the uniform measure on the finite set $\Lambda_{m}$
to the continuous uniform measure.\\

For part (1), we set $s=Ce^{-t}$ to get
\begin{align*}
\frac{1}{\Delta}\int_{0}^{\Delta}\left(u_{Ce^{-t}}\delta_{\Lambda_{m}}\right)\left(\tilde{f}\right)\dt & =-\frac{1}{\Delta}\int_{C}^{Ce^{-\Delta}}\left(u_{s}\delta_{\Lambda_{m}}\right)\left(\tilde{f}\right)\frac{1}{s}\ds.
\end{align*}
For any $\frac{1}{2}>\Delta>0$ small enough we have that $\Delta^{2}\geq e^{-\Delta}-\left(1-\Delta\right)\geq0$,
so let us use it to change the $e^{-\Delta}$ in the upper bound of
the integral to $\left(1-\Delta\right)$:
\begin{equation}
\left|\frac{1}{\Delta}\int_{C}^{Ce^{-\Delta}}\left(u_{s}\delta_{\Lambda_{m}}\right)\left(\tilde{f}\right)\frac{1}{s}\ds-\frac{1}{\Delta}\int_{C}^{C\left(1-\Delta\right)}\left(u_{s}\delta_{\Lambda_{m}}\right)\left(\tilde{f}\right)\frac{1}{s}\ds\right|\norm{\tilde{f}}\frac{e^{-\Delta}-\left(1-\Delta\right)}{\Delta\left(1-\Delta\right)}\leq2\norm f\varepsilon.\label{eq:22}
\end{equation}
Next, we want to get rid of the $\frac{1}{s}$ part, by noting that
$s$ is almost constant. Indeed, we have that $\frac{1}{C\left(1-\Delta\right)}\leq\frac{1}{s}\leq\frac{1}{C}$
and therefore $\left|\frac{1}{s}-\frac{1}{C}\right|\leq\frac{\Delta}{C\left(1-\Delta\right)}\leq\frac{2\Delta}{C}$.It
follows that 
\begin{equation}
\left|\frac{1}{\Delta}\int_{C}^{C\left(1-\Delta\right)}\left(u_{s}\delta_{\Lambda_{m}}\right)\left(\tilde{f}\right)\frac{1}{s}\ds-\frac{1}{\Delta}\int_{C}^{C\left(1-\Delta\right)}\left(u_{s}\delta_{\Lambda_{m}}\right)\left(\tilde{f}\right)\frac{1}{C}\ds\right|\leq2\norm f_{\infty}\varepsilon.\label{eq:33}
\end{equation}
Finally, we use \lemref{approx_horo} to get 
\begin{equation}
\left|\frac{1}{\left|C\right|\Delta}\int_{C}^{C\left(1-\Delta\right)}\left(u_{s}\delta_{\Lambda_{m}}\right)\left(\tilde{f}\right)\ds-\delta_{x_{\RR}U}\left(\tilde{f}\right)\right|\leq\frac{2^{\omega\left(m\right)}}{\varphi\left(m\right)\left|C\right|\Delta}\norm f_{\infty}=\frac{2^{\omega\left(m\right)}}{\varphi\left(m\right)ne^{-x}\Delta}\norm f_{\infty}.\label{eq:44}
\end{equation}
Since we assume that $x\leq\left(1-2\varepsilon\right)\ln\left(mn\right)$,
and for all $m$ large enough we have that $2^{\omega\left(m\right)}\leq m^{\varepsilon/2}$
and $m^{1-\varepsilon/2}\leq\varphi\left(m\right)$, then the upper
bound is at most 
\[
\frac{m^{\varepsilon/2}}{m^{1-\varepsilon/2}n\left(nm\right)^{2\varepsilon-1}}\frac{\norm f_{\infty}}{\Delta}\leq\frac{1}{n^{\varepsilon}m^{\varepsilon}}\frac{\norm f_{\infty}}{\Delta}.
\]
Putting all the triangle inequalities in \eqref{11}, \eqref{22},
\eqref{33} and \eqref{44} together, we get that 
\[
\left|\frac{1}{\Delta}\left(a\left(x\right)u_{-n}\right)_{*}\delta_{\Lambda_{m}}^{\left[x,x+\Delta\right]}\left(\tilde{f}\right)-\mu_{\Gamma U}\left(\tilde{f}\right)\right|<\varepsilon+4\norm f\varepsilon+\frac{1}{n^{\varepsilon}m^{\varepsilon}}\frac{\norm f_{\infty}}{\Delta}
\]
which completes the proof.
\end{proof}
Finally, we can use the result about equidistribution of expanding
horocycles to show the equidistribution of our measures and the last
condition from \thmref{main_lifting}.
\begin{thm}
Let $n_{i},m_{i}\in\NN$ we diverge to infinity and set $T_{i}=\ln\left(n_{i}m_{i}\right)$.
Then \\
$\frac{1}{T_{i}}u_{-n_{i}}\delta_{\Lambda_{m_{i}}}^{\left[0,T_{i}\right]}\wstar\mu_{Haar,\RR}$
equidistributes.
\end{thm}

\begin{proof}
Fix some $f\in C_{c}\left(X\right)$ , $\varepsilon>0$ and let $\Delta=\Delta\left(f,\varepsilon\right)\leq\varepsilon$
as in \lemref{small_interval-1-2} where we may assume that $\Delta\mid T_{i}$.
Given $m,n\in\NN$ we write 
\[
\left(u_{-n_{i}}\right)_{*}\delta_{\Lambda_{m_{i}}}^{\left[0,T_{i}\right]}\left(f\right)=\Delta\sum_{1}^{T_{i}/\Delta}\frac{1}{\Delta}\left(u_{-n_{i}}\right)_{*}\delta_{\Lambda_{m_{i}}}^{\left[\left(k-1\right)\Delta,k\Delta\right]}\left(f\right).
\]
Pick $k$ such that $\varepsilon T_{i}\leq k\Delta<\left(1-2\varepsilon\right)T_{i}$.
By \lemref{small_interval-1-2} we get that 
\[
\left|\frac{1}{\Delta}\left(u_{-n}\right)_{*}\delta_{\Lambda_{m}^{-}}^{\left[\left(k-1\right)\Delta,k\Delta\right]}\left(f\right)-\mu_{\Gamma U}\left(a\left(k\Delta\right)f\right)\right|\leq\varepsilon+4\norm f\varepsilon+\frac{1}{n^{\varepsilon}m^{\varepsilon}}\frac{\norm f_{\infty}}{\Delta}.
\]
Since $\varepsilon,\Delta$ are fixed and $n_{i},m_{i}\to\infty$,
then for all $i$ large enough the last bound is smaller than $\varepsilon\left(1+5\norm f_{\infty}\right)$.

Using the equidistribution of expanding horocycle, since $T_{i}\to\infty$
we get that for every $i$ big enough we can approximate the integral
over the horocycle by
\[
\left|\mu_{\Gamma U}\left(a\left(k\Delta\right)f\right)-\mu_{Haar}\left(f\right)\right|\leq\varepsilon.
\]
There are at most $6\varepsilon\frac{T_{i}}{\Delta}$ integers $k$
which do not satisfy our condition above, for which we have the trivial
bound 
\[
\left|\mu_{\Gamma U}\left(a\left(k\Delta\right)f\right)-\mu_{Haar}\left(f\right)\right|\leq2\norm f.
\]
Putting it all together, we get that for all $i$ big enough
\begin{align*}
\left|\left(u_{-n_{i}}\right)_{*}\delta_{\Lambda_{m_{i}}}^{\left[0,T_{i}\right]}\left(f\right)-\mu_{Haar}\left(f\right)\right| & \leq\frac{\Delta}{T_{i}}\sum_{1}^{T_{i}/\Delta}\left|\left(\frac{1}{\Delta}\left(u_{-n_{i}}\right)_{*}\delta_{\Lambda_{m_{i}}}^{\left[\left(k-1\right)\Delta,k\Delta\right]}\left(f\right)-\mu_{Haar}\left(f\right)\right)\right|\\
 & \leq\varepsilon\left(1+5\norm f_{\infty}\right)+12\varepsilon\norm f_{\infty}.
\end{align*}
As this is true for every $\varepsilon>0$, we conclude that $\left(u_{-n_{i}}\right)_{*}\delta_{\Lambda_{m_{i}}}^{\left[0,T_{i}\right]}\left(f\right)\to\mu_{Haar}\left(f\right)$,
and since $f\in C_{c}\left(X\right)$ was arbitrary we get the required
equidistribution. 

\end{proof}

\appendix

\section{\label{sec:Mahler_criterion}The generalized Mahler's criterion}

\global\long\def\norm#1{\left\Vert #1\right\Vert }%
\global\long\def\AA{\mathbb{A}}%
\global\long\def\QQ{\mathbb{Q}}%
\global\long\def\PP{\mathbb{P}}%
\global\long\def\CC{\mathbb{C}}%
\global\long\def\HH{\mathbb{H}}%
\global\long\def\ZZ{\mathbb{Z}}%
\global\long\def\NN{\mathbb{N}}%
\global\long\def\KK{\mathbb{K}}%
\global\long\def\RR{\mathbb{R}}%
\global\long\def\FF{\mathbb{F}}%
\global\long\def\oo{\mathcal{O}}%
\global\long\def\aa{\mathcal{A}}%
\global\long\def\bb{\mathcal{B}}%
\global\long\def\ff{\mathcal{F}}%
\global\long\def\mm{\mathcal{M}}%
\global\long\def\limfi#1#2{{\displaystyle \lim_{#1\to#2}}}%
\global\long\def\pp{\mathcal{P}}%
\global\long\def\qq{\mathcal{Q}}%
\global\long\def\da{\mathrm{da}}%
\global\long\def\dt{\mathrm{dt}}%
\global\long\def\dg{\mathrm{dg}}%
\global\long\def\ds{\mathrm{ds}}%
\global\long\def\dm{\mathrm{dm}}%
\global\long\def\dmu{\mathrm{d\mu}}%
\global\long\def\dx{\mathrm{dx}}%
\global\long\def\dh{\mathrm{dh}}%
\global\long\def\dy{\mathrm{dy}}%
\global\long\def\dz{\mathrm{dz}}%
\global\long\def\dnu{\mathrm{d\nu}}%
\global\long\def\flr#1{\left\lfloor #1\right\rfloor }%
\global\long\def\nuga{\nu_{\mathrm{Gauss}}}%
\global\long\def\diag#1{\mathrm{diag}\left(#1\right)}%
\global\long\def\bR{\mathbb{R}}%
\global\long\def\Ga{\Gamma}%
\global\long\def\PGL{\mathrm{PGL}}%
\global\long\def\GL{\mathrm{GL}}%
\global\long\def\PO{\mathrm{PO}}%
\global\long\def\SL{\mathrm{SL}}%
\global\long\def\PSL{\mathrm{PSL}}%
\global\long\def\SO{\mathrm{SO}}%
\global\long\def\mb#1{\mathrm{#1}}%
\global\long\def\wstar{\overset{w^{*}}{\longrightarrow}}%
\global\long\def\vphi{\varphi}%
\global\long\def\av#1{\left|#1\right|}%
\global\long\def\inv#1{\left(\mathbb{Z}/#1\mathbb{Z}\right)^{\times}}%
\global\long\def\cH{\mathcal{H}}%
\global\long\def\cM{\mathcal{M}}%
\global\long\def\bZ{\mathbb{Z}}%
\global\long\def\bA{\mathbb{A}}%
\global\long\def\bQ{\mathbb{Q}}%
\global\long\def\bP{\mathbb{P}}%
\global\long\def\eps{\epsilon}%
\global\long\def\on#1{\mathrm{#1}}%
\global\long\def\nuga{\nu_{\mathrm{Gauss}}}%
\global\long\def\set#1{\left\{  #1\right\}  }%
\global\long\def\smallmat#1{\begin{smallmatrix}#1\end{smallmatrix}}%
\global\long\def\len{\mathrm{len}}%
\global\long\def\idealeq{\trianglelefteqslant}%

Mahler's criterion is very useful when trying to study the space of
Euclidean lattices. Since in this notes we work with the generalized
version of $S$- adic lattices, in this section we give the definition
and proofs for the generalized Mahler criterion.for boundedness in
the space of Euclidean lattices is very useful \\

\begin{defn}
For $v\in\QQ_{\nu}^{n}$, we write $\norm v_{\nu}=\max\left|v_{i}\right|_{\nu}$.
For $S\subseteq\PP_{\infty}$ and $\left(v^{\left(\nu\right)}\right)\in\QQ_{S}^{d}$
we set $\norm v_{S}=\prod_{\nu\in S}\norm{v^{\left(\nu\right)}}_{\nu}$.
\end{defn}

The ``norm'' function above is the generalization of the standard
norm that we use in Euclidean spaces. As with our $\left|\cdot\right|_{S}$
notation on $\QQ_{S}$, it is possible for $\norm v_{S}=0$ without
$v=0$, though for $0\neq v\in\QQ^{n}\leq\QQ_{S}^{n}$ the norm will
always be nonzero.

For the real place, we have a natural geometric intuition regarding
the norm. For the $p$-prime case we have instead an algebraic interpretation.
For $v\in\QQ_{p}^{d}$, it is easy to check that the $\ZZ_{p}$ module
$\left\langle v\right\rangle _{\ZZ_{p}}:=span_{\ZZ_{p}}\left\{ v_{i}\mid1\leq i\leq d\right\} $
satisfy $\left\langle v\right\rangle _{\ZZ_{p}}=\norm v_{p}\ZZ_{p}$.
In particular we get that for $M\in M_{d}\left(\ZZ_{p}\right)$ we
have that $\left\langle vM\right\rangle _{\ZZ_{p}}\leq\left\langle v\right\rangle _{\ZZ_{p}}$
so that $\norm{vM}_{p}\geq\norm v_{p}$, and if $M\in\GL_{n}\left(\ZZ_{p}\right)$
(e.g. $M$ is diagonal over $\ZZ_{p}^{\times}$), then we have equality.
In other words, $\GL_{d}\left(\ZZ_{p}\right)$ preserve the norm.

Continuing with the generalization of Mahler's criterion, recall that
for a Euclidean lattice $L\leq\RR^{n}$ we define the height to be
$ht\left(L\right)=\left(\min_{0\neq v\in L}\norm v\right)^{-1}$.
We now generalize to to $S$-adic lattices.
\begin{defn}
For $g=\left(g^{\left(p\right)}\right)\in\GL_{n}\left(\QQ_{S}\right)$
we define the \emph{height function} $ht_{S}\left(g\right)=\left({\displaystyle \inf_{0\neq v\in\ZZ\left[S^{-1}\right]^{n}}}\norm{vg}\right)^{-1}$.
\end{defn}

Note that $ht_{S}\left(g\right)$ is constant on left orbits of $\Gamma_{S}=\GL_{d}\left(\ZZ\left[S^{-1}\right]\right)$,
so that it is actually a function on $X_{S}=\Gamma_{S}\backslash G_{S}$.
Since $H_{S}$ acts transitively on $X_{S}$, we can always find a
representative $h\in H_{S}$ such that $\Gamma_{S}h=\Gamma g$, and
therefore $ht_{S}\left(g\right)=ht_{S}\left(h\right)$. Using this
presentation we can describe the height in a more familiar way, and
in particular show that the infimum is a minimum.
\begin{lem}
\label{lem:height_function}Fix some $S\subseteq\PP_{\infty}.$
\begin{enumerate}
\item For any $v\in\QQ_{S}^{n}$ and $q\in\ZZ\left[S^{-1}\right]^{\times}=\left\langle \pm S\right\rangle $
we have that $\norm{qv}=\norm v$. For $q\in\left\langle \PP\backslash S\right\rangle $
we have that $\norm{qv}=\left|q\right|_{\infty}\norm v$.
\item For any $g\in\GL_{d}\left(\QQ_{S}\right)$ we have that ${\displaystyle \inf_{0\neq v\in\ZZ\left[S^{-1}\right]^{n}}}\norm{vg}={\displaystyle \inf_{\substack{0\neq v\in\ZZ^{n}\\
v\;primitive
}
}}\norm{vg}$.
\item If $h\in H_{S}$, then ${\displaystyle \inf_{0\neq v\in\ZZ\left[S^{-1}\right]^{n}}}\norm{vh}={\displaystyle \inf_{\substack{0\neq v\in\ZZ^{n}\\
v\;primitive
}
}}\norm{vh^{\left(\infty\right)}}$.
\end{enumerate}
\end{lem}

\begin{proof}
\begin{enumerate}
\item By definition we have that $\norm{qv}=\left(\prod_{\nu\in S}\left|q\right|_{\nu}\right)\norm v$.
For for $q\in\left\langle \pm S\right\rangle $ we have the product
formula $\prod_{\nu\in S}\left|q\right|_{\nu}=1$, while for $q\in\left\langle \PP\backslash S\right\rangle $
we have that $\left|q\right|_{p}=1$ for all primes $p\in S$ so that
$\left(\prod_{p\in S}\left|q\right|_{p}\right)=\left|q\right|_{\infty}$. 
\item Given $0\neq v\in\ZZ\left[S^{-1}\right]^{n}$ , we can write it as
$v=q_{S}mu$ for some $q_{S}\in\left\langle \pm S\right\rangle $
and $m\in\ZZ$ such $m\in\left\langle \PP\backslash S\right\rangle $
and $u\in\ZZ^{n}$ is primitive. By part 1 we have that 
\[
\norm{vg}=\norm{q_{S}mug}=\left|m\right|_{\infty}\norm{ug}\geq\norm{ug}.
\]
Thus, we get that ${\displaystyle \inf_{0\neq v\in\ZZ\left[S^{-1}\right]^{n}}}\norm{vg}\geq{\displaystyle \inf_{0\neq v\in\ZZ^{n}}}\norm{vg}$.
The converse is clearly true, so we have an equality. 
\item Assume now that $h\in H_{S}$. For any finite prime $p\in S$, we
have that $h^{\left(p\right)}\in\GL_{n}\left(\ZZ_{p}\right)$ so in
particular $h^{\left(p\right)}\;(mod\;p)$ is well defined and invertible.
Given $v\in\ZZ^{n}$ primitive, we get that $vh^{\left(p\right)}\in\ZZ_{p}^{n}$
is a nonzero vector mod $p$, and therefore $\norm{vh^{\left(p\right)}}_{p}=1$
for all finite prime $p$. The proof is now completed by the fact
that 
\[
\norm{vh}=\norm{vh^{\left(\infty\right)}}_{\infty}\prod_{\infty\neq p\in S}\norm{vh^{\left(p\right)}}_{p}=\norm{vh^{\left(\infty\right)}}_{\infty}.
\]
\end{enumerate}
\end{proof}
Note that part (3) says that for $h\in H_{S}$ we have that $ht_{S}\left(h\right)=ht_{\infty}\left(h^{\left(\infty\right)}\right)$.
This allows us to generalize Mahler's criterion from Euclidean to
$S$-adic lattices.
\begin{thm}[Mahler's criterion]
 A set $\Omega\subseteq X_{S}$ is bounded if and only if $\left\{ ht_{S}\left(x\right)\;\mid\;x\in\Omega\right\} $
is bounded.
\end{thm}

\begin{proof}
Recall that we have the projection $\pi_{\RR}^{S}:X_{S}\to X_{\RR}$
induced from the projection \\
$H_{S}\to H_{\RR}=\SL_{d}\left(\RR\right)$. Since $\pi_{\RR}^{S}$
is proper (the preimage of every point is an orbit of the compact
group ${\displaystyle \prod_{p\in S\backslash\left\{ \infty\right\} }}\GL_{d}\left(\ZZ_{p}\right)$),
it follows that $\Omega\subseteq X_{S}$ is bounded if and only if
its image $\pi_{\RR}^{S}\left(\Omega_{S}\right)$ is bounded. The
standard Mahler's criterion tell us that $\pi_{\RR}^{S}\left(\Omega_{S}\right)$
is bounded if and only if the standard height function $ht_{\infty}$
is bounded on this set. We now use part (3) of \ref{lem:height_function}
to show that this is equivalent to $ht_{S}\left(\Omega_{S}\right)$
being bounded.
\end{proof}
\begin{rem}
Note that since the height function is well defined, we get that if
$L\leq\QQ_{S}^{d}$ is a lattice, then $\norm v_{S}>0$ for any $0\neq v\in L$.
In particular, if $v\in\QQ_{S}^{d}$ satisfy $\norm v_{S}=0$, that
it is not contained in any lattice.
\end{rem}

\newpage{}

\section{\label{sec:Disintegration-of-measures}Disintegration of measures}

\global\long\def\norm#1{\left\Vert #1\right\Vert }%
\global\long\def\AA{\mathbb{A}}%
\global\long\def\QQ{\mathbb{Q}}%
\global\long\def\PP{\mathbb{P}}%
\global\long\def\CC{\mathbb{C}}%
\global\long\def\HH{\mathbb{H}}%
\global\long\def\ZZ{\mathbb{Z}}%
\global\long\def\NN{\mathbb{N}}%
\global\long\def\KK{\mathbb{K}}%
\global\long\def\RR{\mathbb{R}}%
\global\long\def\FF{\mathbb{F}}%
\global\long\def\oo{\mathcal{O}}%
\global\long\def\aa{\mathcal{A}}%
\global\long\def\bb{\mathcal{B}}%
\global\long\def\ff{\mathcal{F}}%
\global\long\def\mm{\mathcal{M}}%
\global\long\def\limfi#1#2{{\displaystyle \lim_{#1\to#2}}}%
\global\long\def\pp{\mathcal{P}}%
\global\long\def\qq{\mathcal{Q}}%
\global\long\def\da{\mathrm{da}}%
\global\long\def\dt{\mathrm{dt}}%
\global\long\def\dg{\mathrm{dg}}%
\global\long\def\ds{\mathrm{ds}}%
\global\long\def\dm{\mathrm{dm}}%
\global\long\def\dmu{\mathrm{d\mu}}%
\global\long\def\dx{\mathrm{dx}}%
\global\long\def\dh{\mathrm{dh}}%
\global\long\def\dy{\mathrm{dy}}%
\global\long\def\dz{\mathrm{dz}}%
\global\long\def\dnu{\mathrm{d\nu}}%
\global\long\def\flr#1{\left\lfloor #1\right\rfloor }%
\global\long\def\nuga{\nu_{\mathrm{Gauss}}}%
\global\long\def\diag#1{\mathrm{diag}\left(#1\right)}%
\global\long\def\bR{\mathbb{R}}%
\global\long\def\Ga{\Gamma}%
\global\long\def\PGL{\mathrm{PGL}}%
\global\long\def\GL{\mathrm{GL}}%
\global\long\def\PO{\mathrm{PO}}%
\global\long\def\SL{\mathrm{SL}}%
\global\long\def\PSL{\mathrm{PSL}}%
\global\long\def\SO{\mathrm{SO}}%
\global\long\def\mb#1{\mathrm{#1}}%
\global\long\def\wstar{\overset{w^{*}}{\longrightarrow}}%
\global\long\def\vphi{\varphi}%
\global\long\def\av#1{\left|#1\right|}%
\global\long\def\inv#1{\left(\mathbb{Z}/#1\mathbb{Z}\right)^{\times}}%
\global\long\def\cH{\mathcal{H}}%
\global\long\def\cM{\mathcal{M}}%
\global\long\def\bZ{\mathbb{Z}}%
\global\long\def\bA{\mathbb{A}}%
\global\long\def\bQ{\mathbb{Q}}%
\global\long\def\bP{\mathbb{P}}%
\global\long\def\eps{\epsilon}%
\global\long\def\on#1{\mathrm{#1}}%
\global\long\def\nuga{\nu_{\mathrm{Gauss}}}%
\global\long\def\set#1{\left\{  #1\right\}  }%
\global\long\def\smallmat#1{\begin{smallmatrix}#1\end{smallmatrix}}%
\global\long\def\len{\mathrm{len}}%
\global\long\def\idealeq{\trianglelefteqslant}%

Disintegration of measures is a well known process used to study probability
measures. In this notes we deal with locally finite measures on homogeneous
spaces, so for completeness we add the proofs to the generalization
of the disintegration for these measures.

We start by recalling the standard theorem.
\begin{thm}[Disintegration of probability measures]
 Let $\pi:Y\to X$ be Borel measurable function of Radon spaces,
$\mu$ a probability measure on $Y$ and set $\nu=\pi_{*}\mu$ the
push forward probability on $X$. Then there exists $\nu$ almost
everywhere uniquely determined probability measures $\mu_{x}$ on
$Y$ such that
\begin{enumerate}
\item For each $B\subseteq Y$ measurable, the function $x\mapsto\mu_{x}\left(B\right)$
is measurable.
\item For almost every $x$ we have that $\mu_{x}\left(\pi^{-1}\left(x\right)\right)=1$
.
\item For every Borel function $f:Y\to\left[0,\infty\right]$ we have that
\[
\int_{Y}f\left(y\right)\dmu=\int_{X}\left(\int_{Y}f\left(y\right)\dmu_{x}\left(y\right)\right)\dnu\left(x\right).
\]
\end{enumerate}
\end{thm}

We would like to extend this theorem to locally finite measures on
some group $H$ with respect to a map $\pi:H\to\nicefrac{H}{W}=K$.
The problem there is that $\pi_{*}\mu$ is usually infinite on many
sets. To solve this problem we instead apply the theorem to increasing
parts of $H$ and then make sure that this defines a good measure
in the limit.

For the rest of this section we will use this assumption.
\begin{assumption}
\label{assu:disintegration}Let $H$ be a group with decomposition
$H=K\cdot W$ with $W\cap K=\left\{ e\right\} $ where the map $K\times W\to G$
is a homeomorphism. All these groups are second countable, locally
compact and Hausdorff. We denote $\pi:H\to\nicefrac{H}{W}\cong K$
the natural projection and assume that $K$ is compact. Finally, we
let $\mu$ be a right $W$-invariant measure on $G$ and denote by
$\mu_{W}$ a right $W$-invariant measure on $W$ (which is unique
up to a scalar).
\end{assumption}

Fix some $U\subseteq W$ open with compact closure. Since $W$ is
second countable, we can find countably many $w_{i}\in W$ such that
$\bigcup_{1}^{\infty}Uw_{i}=W$ and therefore
\[
\mu\left(H\right)=\mu\left(\bigcup_{1}^{\infty}KUw_{i}\right)\leq\sum_{1}^{\infty}\mu\left(KUw_{i}\right).
\]
If $\mu\left(KU\right)=0$, then by the right $W$-invariance of $\mu$
we get that $\mu\left(H\right)=0$ - contradiction. Hence we must
have that $\mu\left(KU\right)>0$. 

We can now apply the disintegration theorem to $\frac{1}{\mu\left(KU\right)}\mu\mid_{KU}$
and get in particular that if \\
$f:KU\to\left[0,\infty\right]$ is Borel measurable , then
\[
\int_{H}f\left(y\right)\dmu=\int_{KU}f\left(y\right)\dmu=\mu\left(KU\right)\int_{K}\left(\int_{KU}f\left(y\right)\dmu_{k,U}\left(y\right)\right)\dnu_{U}\left(k\right).
\]
Since for almost every $k$, the measure $\mu_{k,U}$ is supported
on $kU$, we will instead consider its induced measure on $U\subseteq W$
and write instead $\int_{W}f\left(kh\right)\dmu_{k,U}\left(h\right)$,
so that 
\[
\int_{H}f\left(y\right)\dmu=\int_{K}\left(\int_{W}f\left(kh\right)\mu\left(KU\right)\dmu_{k,U}\left(h\right)\right)\dnu_{U}\left(k\right).
\]

The first step is to show that $\nu_{U}$ is actually independent
of the choice of $U$.

\newpage{}
\begin{lem}
The probability measure $\nu_{U}$ is independent of $U$.
\end{lem}

\begin{proof}
Given $\Omega\subseteq K$ measurable, we have that 
\[
\mu\left(\Omega U\right)=\int_{H}\chi_{\Omega U}\dmu=\int_{K}\left(\int_{W}\chi_{\Omega U}\left(kh\right)\mu\left(KU\right)\dmu_{k,U}\left(h\right)\right)\dnu_{U}\left(k\right)=\mu\left(KU\right)\nu_{U}\left(\Omega\right).
\]
For any fixed measurable subset $\Omega\subseteq K$, we define the
measure $V\mapsto\mu\left(\Omega V\right)$ on $W$ which is right
$W$-invariant (because $\mu$ is right $W$-invariant). Thus, we
have that $\mu\left(\Omega V\right)=\lambda_{\Omega}\mu_{W}\left(V\right)$
for some scalar $\lambda_{\Omega}\geq0$. Since $\mu\left(KU\right)>0$,
the equality above show that $\nu_{U}\left(\Omega\right)=\frac{\mu\left(\Omega U\right)}{\mu\left(KU\right)}=\frac{\lambda_{\Omega}}{\lambda_{K}}$
doesn't depend on $U$.
\end{proof}
If $f$ is measurable on $KU$ and $U\subseteq V$ are open, then
it is measurable on $KV$ as well, and we can apply the result above
for both spaces. The next step is to show that $\mu\left(KU\right)\mu_{k,U}$
are for almost every $k$ independent of $U$, so afterwards we can
take the limit as $U_{i}\nearrow H$.
\begin{lem}
Let $U\subseteq V$ open in $W$. For $\nu$ almost every $k\in K$
and for every $f$ measurable on $KU$ we have that $\mu\left(KU\right)\mu_{k,U}\left(L_{k}\circ f\right)=\mu\left(KV\right)\mu_{k,V}\left(L_{k}\circ f\right)$.
\end{lem}

\begin{proof}
Write $\xi\left(f,k\right)=\mu\left(KU\right)\mu_{k,U}\left(L_{k}\circ f\right)-\mu\left(KV\right)\mu_{k,V}\left(L_{k}\circ f\right)$
and note that \\
$\xi\left(\chi_{\Omega W}\cdot f,k\right)=\xi\left(f,k\right)\chi_{\Omega}\left(k\right)$
for any $\Omega\subseteq K$. 

For such any $f$ on $KU$ we have that 
\[
\int_{K}\left(\int_{W}f\left(kh\right)\mu\left(KV\right)\dmu_{k,V}\left(h\right)\right)\dnu\left(k\right)=\int_{H}f\left(y\right)\dmu=\int_{K}\left(\int_{W}f\left(kh\right)\mu\left(KU\right)\dmu_{k,U}\left(h\right)\right)\dnu\left(k\right),
\]
so that $\int_{K}\xi\left(f,k\right)\dnu\left(k\right)=0$. Applying
this to $\chi_{\Omega W}\cdot f$ we get that $\int_{\Omega}\xi\left(f,k\right)\dnu\left(k\right)=0$
for all $\Omega\subseteq K$ measurable, hence $\xi\left(f,k\right)=0$
for $\nu$ almost every $k\in K$.

Denote by $B_{f}=\left\{ k\in K\mid\xi\left(f,k\right)\neq0\right\} $.
Choose some countable family of functions $\left\{ f_{i}\right\} $
which is dense in $C\left(K\overline{U}\right)$. Then $f\in\overline{\left\{ f_{i}\right\} }$
implies that $B_{f}\subseteq\bigcup B_{f_{i}}$, and in particular,
outside of the zero $\nu$-measure set $\bigcup B_{f_{i}}$ we have
that $\xi\left(f,k\right)=0$ for all $f$ measurable on $KU$.
\end{proof}
\begin{defn}
Let $U_{i}\nearrow H$ be open with compact closure. For $f\in C_{c}\left(G\right)$
with $supp\left(f\right)\subseteq KU_{i}$, define $\mu_{k}\left(f\right)=\int_{W}f\left(h\right)\mu\left(KU_{i}\right)\dmu_{k,U_{i}}\left(h\right)$.
\end{defn}

\begin{cor}
By the previous lemma, the definition of $\mu_{k}$ doesn't depend
on $i$ for almost every $k$. Hence we have that 
\[
\int_{H}f\left(g\right)\dmu\left(g\right)=\int_{K}\left(\int_{W}f\left(kh\right)\dmu_{k}\left(h\right)\right)\dnu\left(k\right).
\]
\end{cor}

Finally, we want to show that $\mu_{k}$ are the Haar measure on $H$.
\begin{thm}[Disintegration of measures on $H$]
T Let $H,W,K,\mu$ be as in \assuref{disintegration}, Then there
exist $r_{k}\geq0$ such that 
\[
\int_{H}f\left(g\right)\dmu\left(g\right)=\int_{K}\left(\int_{W}f\left(kh\right)\dmu_{W}\left(h\right)\right)r_{k}\dnu\left(k\right).
\]
\end{thm}

\begin{proof}
This is done similar to the previous lemma. For any continuous function
$f$ with compact support on $H$, any $h_{0}\in W$ and any subset
$K_{0}\subseteq K$ we have that
\begin{align*}
\int_{K_{0}}\mu_{k}\left(L_{k}\left(f\right)\right)\dnu\left(k\right) & =\int_{H}\left(\chi_{K_{0}W}\cdot f\right)\left(g\right)\dmu\left(g\right)=\int_{H}\left(\chi_{K_{0}W}\cdot f\right)\left(gh_{0}\right)\dmu\left(g\right)\\
 & =\int_{K}\left(\int_{W}\left(\chi_{K_{0}W}f\right)\left(khh_{0}\right)\dmu_{k}\left(h\right)\right)\dnu\left(k\right)=\int_{K_{0}}\mu_{k}\left(R_{h_{0}}L_{k}\left(f\right)\right)\dnu\left(k\right).
\end{align*}
Since this is true for any $K_{0}$ we get that $\mu_{k}\left(L_{k}\left(f\right)\right)=\left(\mu_{k}\right)\left(R_{h_{0}}L_{k}\left(f\right)\right)$
for almost every $k$. Again, using separability we get that for almost
every $k$ this is true for all $f$. Since $H$ is also separable
we get that for almost every $k$ we have $\mu_{k}=\mu_{k}\circ R_{h_{0}}$
for all $h_{0}\in H$, namely $\mu_{k}$ is right $W$-invariant,
so we can write $\mu_{k}=r_{k}\mu_{W}$ for some $r_{k}\geq0$. To
sum up, we have that 
\[
\int_{G}f\left(g\right)\dmu\left(g\right)=\int_{K}\left(\int_{H}f\left(hk\right)\dmu_{W}\left(h\right)\right)r_{k}\dnu\left(k\right).
\]
\end{proof}

\bibliographystyle{plain}
\bibliography{bib}

\end{document}